\documentclass{amsart}
\usepackage{amsfonts,amssymb,verbatim,amsmath,amsthm,latexsym,textcomp,amscd,caption}
\usepackage[dvipsnames]{xcolor}
\usepackage{tikz,tikz-cd}
\usepackage{graphics,textcomp,hyperref}
\usepackage{graphicx}
\usepackage{url}
\usepackage{psfrag}
\usepackage{stackrel}
\usepackage{subfiles} 
\usepackage[mathscr]{euscript}

\DeclareMathOperator{\Int}{int}

\newcommand{\pwfm}{\textrm{piecewise Fuchsian Markov\,}}
\newcommand{\cfm}{\textrm{completely folding map\,}}
\newcommand{\hdm}{\textrm{higher degree map without folding\,}}

\newtheorem{theorem}{Theorem}[section]
\newtheorem{prop}[theorem]{Proposition}
\newtheorem{lemma}[theorem]{Lemma}
\newtheorem{cor}[theorem]{Corollary}

\newtheorem{qn}[theorem]{Question}
\newtheorem{defn}[theorem]{Definition}

\newtheorem{rmk}[theorem]{Remark}

\newcommand{\hhat}{\widehat}
\newcommand{\C}{{\mathbb C}}

\newcommand{\G}{{\Gamma}}

\newcommand\BB{{\mathcal B}}

\newcommand\FF{{\mathcal F}}
\newcommand\GG{{\mathcal G}}

\newcommand\LL{{\mathcal L}}
\newcommand\MM{{\mathcal M}}

\newcommand\PP{{\mathcal P}}

\newcommand\PMF{{\PP\kern-2pt\MM\FF}}

\newcommand\PML{{\PP\kern-2pt\MM\LL}}

\newcommand\Z{{\mathbb Z}}
\newcommand\R{{\mathbb R}}

\newcommand{\disk}{{\mathbb D}}

\newcommand{\fsubd}{\mathrel{{\scriptstyle\searrow}\kern-1ex^d\kern0.5ex}}
\newcommand{\bsubd}{\mathrel{{\scriptstyle\swarrow}\kern-1.6ex^d\kern0.8ex}}
\newcommand{\fsubeq}{\mathrel{\raise-.7ex\hbox{$\overset{\searrow}{=}$}}}
\newcommand{\bsubeq}{\mathrel{\raise-.7ex\hbox{$\overset{\swarrow}{=}$}}}

\newcommand{\bbar}{\overline}

\newcommand{\bS}{{\mathbb S}}

\makeatletter
\DeclareFontFamily{U}{tipa}{}
\DeclareFontShape{U}{tipa}{m}{n}{<->tipa10}{}
\newcommand{\arc@char}{{\usefont{U}{tipa}{m}{n}\symbol{62}}}%

\newcommand{\arc}[1]{\mathpalette\arc@arc{#1}}

\newcommand{\arc@arc}[2]{%
  \sbox0{$\m@th#1#2$}%
  \vbox{
    \hbox{\resizebox{\wd0}{\height}{\arc@char}}
    \nointerlineskip
    \box0
  }%
}
\makeatother

\newcommand*\circcled[1]{\tikz[baseline=(char.base)]{
\node[shape=circle,draw,inner sep=1.3pt] (char) {#1};}}
\newcommand{\ciq}{\circcled{?}}

\begin{document}

\title{The Sullivan dictionary and Bowen-Series maps}

\author{Mahan Mj}

\address{School of Mathematics, Tata Institute of Fundamental Research, Mumbai-400005, India}
\email{mahan@math.tifr.res.in}
\email{mahan.mj@gmail.com}
\urladdr{http://www.math.tifr.res.in/~mahan}

\author{Sabyasachi Mukherjee}

\address{School of Mathematics, Tata Institute of Fundamental Research, Mumbai-400005, India}
\email{sabya@math.tifr.res.in}
\email{mukherjee.sabya86@gmail.com}
\urladdr{http://www.math.tifr.res.in/~sabya/}


\thanks{Both authors were  supported by  the Department of Atomic Energy, Government of India, under project no.12-R\&D-TFR-5.01-0500 as also  by an endowment of the Infosys Foundation.
	MM was also supported in part by   a DST JC Bose Fellowship. SM was supported in part by SERB research project grant  SRG/2020/000018. }

\begin{abstract}
	The Sullivan dictionary between Kleinian groups and rational dynamics describes striking similarities between the fields, both in terms of the objects of study as well as the techniques used.
 We give an expository account of a recent bridge between the two sides of the dictionary by describing a framework for combining a Fuchsian group with a complex polynomial into a single dynamical system on the Riemann sphere.
\end{abstract}

\maketitle

\hfill To Dennis with admiration and affection.\\

\date{\today}

\section{Introduction}
In this expository article, we draw heavily from and build upon two strands of Dennis Sullivan's work:
\begin{enumerate}
	\item The Sullivan dictionary between Kleinian groups and rational dynamics \cite{sullivan-dict}.
	\item The Patterson-Sullivan measure \cite{sullivan_cm,haus-dim-sul}.
\end{enumerate}
We shall survey these two themes in the light of a recent combination theorem or a bridge between Kleinian groups and polynomial maps discovered by the authors \cite{mj-muk-1}. An essential ingredient in the building of this bridge is the Bowen-Series map \cite{bowen,bowen-series}.

Sullivan's  dictionary \cite[p. 405]{sullivan-dict} was based on the empirical insight that Kleinian groups and rational dynamics  share many common features. For instance, the \emph{limit set} (resp. the \emph{domain of discontinuity}) of a Kleinian group corresponds to the \emph{Julia set} (resp. the \emph{Fatou set}) of a rational map.
Sullivan extended these similarities to a deeper similarity between \emph{techniques}
by  introducing quasiconformal methods into the field of rational dynamics. This led
to the proof of his celebrated \emph{no wandering domains} theorem.
In fact, using these techniques, Sullivan gave a new proof of Ahlfors' finiteness theorem 
along the lines  of the no wandering domains  theorem.

Our focus here is on a line in the Sullivan dictionary that observes the similarity between the following:
\begin{enumerate}
	\item The Bers simultaneous uniformization theorem in Kleinian groups
	\item Polynomial mating in rational dynamics,  introduced by Douady and Hubbard  \cite{douady-mating}. 
\end{enumerate}

The first step is to replace the Kleinian group by a single map that captures its dynamics. This brings us to the notion of a mateable map (see Section \ref{sec-mateability} below for details).
With the context of mateable maps in place, we address the following question:

\begin{qn}\label{qn-basic}
	Which mateable maps  and polynomials can be mated in the spirit of Douady and Hubbard?
\end{qn}

It turns out that  Bowen-Series maps \cite{bowen,bowen-series} for punctured sphere groups provide
such examples. Surprisingly, there exists a new class of  related maps which we call
\emph{higher Bowen-Series maps} that also fit the bill and 
give rise to combination theorems as well as `dynamically natural' homeomorphisms between limit and Julia sets. 
As the name suggests, higher Bowen-Series maps are closely related to Bowen-Series maps. Indeed, higher Bowen-Series maps appear as second iterates of suitable Bowen-Series maps. Higher Bowen-Series maps can also be characterized as `amalgams' of several Bowen-Series maps of the same Fuchsian group with overlapping fundamental domains. This part of the story is complex analytic in flavor and is taken largely from \cite{mj-muk-1}. 

It is worth mentioning that examples of dynamically natural homeomorphisms between limit sets of Kleinian reflection groups (i.e., discrete subgroups of $\mathrm{Aut}(\widehat{\C})$ generated by reflections in finitely many Euclidean circles) including the classical Apollonian gasket limit set and Julia sets of anti-holomorphic rational maps were first constructed in \cite{LLMM4}, and this phenomenon was studied systematically in a general framework in \cite{LMM20,LLM1}. To the best of our knowledge, \cite[Theorem~7.16]{mj-muk-1} gives the first example of such an explicit connection between limit sets and Julia sets in the holomorphic setting.

In the last section of this survey, we turn to the measurable dynamics of mateable maps and the resulting matings. From the point of view of group theory, the measure-theoretic framework, naturally and rather appropriately, turns out to be that of Patterson-Sullivan measures. On the other hand, since mateable maps share features of rational maps, the limit set of a mateable map supports a natural dynamically defined measure: the measure of maximal entropy (the existence of a unique maximal entropy measure for a rational map was proved in \cite{Lyubich_mme} and independently in \cite{Mane,FLM}). The fact that a mateable map is an object halfway between groups and polynomials is reflected in close connections between maximal entropy measures of mateable maps and suitable Patterson-Sullivan measures. We conclude the article with some estimates of Hausdorff dimensions of maximal entropy measures of (higher) Bowen-Series maps and related open questions.

The phenomenon of ``mating'' of rational maps with Fuchsian groups was discovered in the 1990s by Bullett and Penrose in the context of iterated algebraic correspondences \cite{bullett-penrose} and was studied comprehensively in \cite{BL20,BL21}. Specifically, they constructed a family of algebraic correspondences of bi-degree $(2,2)$, and showed that the members of this family can be interpreted in an appropriate sense as matings of the modular group with quadratic rational maps. This is quite different from our mating framework as we extract a non-invertible map (a mateable map) from a Kleinian group (i.e., a semi-group dynamics from the dynamics of a non-commutative group) and then combine this map with the dynamics of a polynomial thereby producing a hybrid dynamical system in one complex variable. It would be quite interesting to know if our mating framework has deeper connections with that of Bullett-Penrose-Lomonaco.

\section*{Acknowledgments} We thank Caroline Series for the proof of Proposition \ref{not_abs_cont_lem}. Both of us have been enriched
and inspired by the beauty and simplicity of Dennis Sullivan's work. This survey article is an attempt to record our debt.

The authors are extremely grateful to the anonymous referees for detailed comments and suggestions for improvement.

\section{Mateability}\label{sec-mateability}

Let $\textrm{Aut}({\disk})$ denote the group of all conformal automorphisms of the unit disk $\disk$. A Fuchsian group $\Gamma$ is a discrete subgroup of $\textrm{Aut}({\disk})$. The aim of this section is to spell out what it means to mate a Fuchsian group with a polynomial. 
We provide the definition of mateability at the outset. The definition below will imply that  $\Gamma$ is a lattice (Lemma \ref{lem-gammalattice}).

\begin{defn}\label{def-mateable}
A continuous map $A:\mathbb{S}^1\to\mathbb{S}^1$ is a \emph{mateable} map associated with a Fuchsian group $\Gamma$ if the following are satisfied:
	
	\begin{enumerate}
		\item\label{oe} $A$ is orbit equivalent to $\Gamma$.
		\item\label{pwa} $A$ is piecewise analytic on $\bS^1$.
		\item\label{exp} $A$ is an expansive covering map of degree greater than one.
		\item\label{markov} $A$ is Markov.
		\item\label{asym_hyp} No periodic break-point of $A$ is asymmetrically hyperbolic.
	\end{enumerate}
\end{defn}  
The failure of any of the conditions in Definition \ref{def-mateable} provides an obstruction to mateability. Somewhat surprisingly, it turns out that these necessary conditions are also sufficient (see Proposition \ref{conformal_mating_general_prop}).

We elaborate now on the terms used in Definition \ref{def-mateable}.
	Let $A:\bS^1\to \bS^1$ be a (not necessarily continuous) map. The \emph{grand orbit} of a point $x\in\bS^1$ under $A$ is defined as 
	$$
	\mathrm{GO}_A(x):=\{ x'\in\bS^1: A^{ m}(x)=A^{ n}(x'),\ \textrm{for\ some}\ m, n\geq 0\}.
	$$

	Let $\Gamma$ be a Fuchsian group with limit set equal to $\Lambda \subset \mathbb{S}^1$. We say that a (not necessarily continuous) map $A:\mathbb{S}^1\to\mathbb{S}^1$ is \emph{orbit equivalent} to $\Gamma$ on $\Lambda$ if for every $x\in\Lambda$,  $$\Gamma\cdot x= \mathrm{GO}_A(x).$$

A (not necessarily continuous) map $A:\mathbb{S}^1\to\mathbb{S}^1$  is \emph{piecewise M{\"o}bius} if there exist $k\in\mathbb{N}$, closed arcs $I_j\subset\mathbb{S}^1$, and $g_j\in\textrm{Aut}({\disk})$ for $j\in\{1,\cdots, k\}$,  such that
	\begin{enumerate}
		\item $\displaystyle\mathbb{S}^1=\bigcup_{j=1}^k I_j,$
		
		\item $\Int{I_m}\cap\Int{I_n}=\emptyset$ for $m\neq n$, and
		
		\item $A\vert_{I_j}=g_j$.
	\end{enumerate}
	
	\noindent A piecewise M{\"o}bius map $A$ as above is called \emph{piecewise Fuchsian} if $g_1,\cdots, g_k$ generate a Fuchsian group, which we denote by $\Gamma_A$.
	If the maps $g_j$ are assumed only to be complex-analytic in some small neighborhoods of $I_j$ (without requiring them to be M{\"o}bius), then $f$ is said to be \emph{piecewise analytic}. 
	
	The maps $g_j$ will be called the \emph{pieces} of $A$. We shall occasionally refer to the domains $I_j$ of $g_j$ also as  \emph{pieces}
	of $A$ when there is no scope for confusion.

\begin{rmk}\label{pwm_def_rem}
We think of the partition of $\mathbb{S}^1$ into the closed arcs $\{I_j\}$ as a part of the data of the piecewise M{\"o}bius/analytic map $A$. This can be formalized by defining a piecewise M{\"o}bius/analytic map $A$ as a pair $\left(\{g_j\}_{j=1}^k,\{I_j\}_{j=1}^k\right)$.
\end{rmk}

Lemma \ref{pwa_is_pwm_lem} below upgrades the regularity of $A$
considerably.

\begin{lemma}\cite[Lemma~2.8]{mj-muk-1}\label{pwa_is_pwm_lem}
	Let $A:\mathbb{S}^1\to\mathbb{S}^1$ be a (not necessarily continuous) piecewise analytic map  that is orbit equivalent to a finitely generated Fuchsian group $\Gamma$. Then, $A$ is piecewise Fuchsian, and the pieces of $A$ form a generating set for $\Gamma$.
\end{lemma}

Suppose that $x_1, \cdots, x_k$ are a
cyclically ordered collection of $k$ points on $\bS^1$ defining the pieces $I_j=[x_j,x_{j+1}]$ of $A$ ($j+1$ taken modulo $k$).
	We shall say that $A$ is \emph{minimal}, if the decomposition of $\bS^1$ given by $x_1, \cdots, x_k$ is minimal;
	i.e., there does not exist $i$ and $h \in \Gamma_A$ such that
	\begin{enumerate}
		\item $A|_{[x_i,x_{i+1}]} = h|_{[x_i,x_{i+1}]}$, and
		\item $A|_{[x_{i-1},x_{i}]} = h|_{[x_{i-1},x_{i}]}$.
	\end{enumerate}
	Thus, a minimal $A$ has no superfluous break-points.

	Let $A$ be a continuous 
	piecewise M\"obius map on the circle. Let $\disk$ denote the unit disk.
	Let $I_1, \cdots, I_k$
	be a circularly ordered family of intervals with disjoint interiors such that
	\begin{enumerate}
		\item $I_j \cap I_{j+1} = \{x_{j+1}\}$ (the indices being taken mod $k$).
		\item $A |_{I_j} = g_j$.
	\end{enumerate}
	Let $\gamma_j$ be the semi-circular arc in $\disk$ between $x_{j}, x_{j+1}$ meeting $\bS^1$ at right angles at $x_{j}, x_{j+1}$, and let $\mathcal{D}_j \subset \bbar{\disk}$ be the closed region bounded by $I_j$ and $\gamma_j$. Then 
	$\widehat{A}$, the \emph{canonical extension of $A$ to a piecewise M{\"o}bius map in $\overline{\disk}$} 
	is defined on $\cup_j \mathcal{D}_j$ as $\widehat A = g_j$ on $\mathcal{D}_j$.

Set $\mathcal{D} :=\cup_j \mathcal{D}_j$ and call $\mathcal{D}$ the \emph{canonical domain of definition} of $\widehat A$.
	Let $R = \disk \setminus \mathcal{D}$. We shall call $R$  the 
	\emph{fundamental domain} of $A$, as well  as the fundamental domain of $\widehat{A}$.
	Each bi-infinite hyperbolic geodesic contained in the boundary $\partial R$ will be called an \emph{edge} of $R$.
	The ideal vertices of $R$ will be called the \emph{vertices} of $R$. Let $S$
	be the set of vertices of $R$.
A 	pair of non-adjacent points in $S$, or equivalently the bi-infinite geodesic joining them in $R$ will be called a \emph{diagonal} of $R$.

\begin{rmk}
We note that the fundamental domain of a piecewise Fuchsian map $A$ may not be a fundamental domain for the Fuchsian group $\Gamma_A$ generated by the pieces of $A$ (see Subsection~\ref{sec-cfm}).
\end{rmk}

\begin{rmk}\label{rmk-ai} Let $A:\bS^1\to\bS^1$ be a continuous piecewise 
	M\"obius map with pieces $\{g_j\}_{j=1}^k$. By continuity, $g_j(x_{j+1}) = g_{j+1}(x_{j+1})$; i.e., $a_j=g_j^{-1}\circ g_{j+1}\in\Gamma_A$ fixes $x_{j+1}$ (indices taken modulo $k$). 
	Then, $a_k\cdots a_1 =1$ as a group element, or equivalently, $a_1\circ\cdots\circ a_k=\mathrm{id}$. 
	Moreover, if $A$ is orbit equivalent to a Fuchsian group $\Gamma$ on $S^1$, then
	$\Gamma$ is generated by $\{g_1, a_1, \cdots, a_k\}$ by Lemma \ref{pwa_is_pwm_lem}.
\end{rmk}

	A continuous map $f\colon \mathbb{S}^1\to \mathbb{S}^1$ is said to be \emph{expansive} if there exists $\delta>0$ such that for any $a \neq b\in \mathbb{S}^1$, there exists $n\in \mathbb{N}$ such that $d(f^{ n}(a), f^{ n}(b))>\delta$.

We endow $\mathbb{S}^1$ with the counter-clockwise orientation. For $a, b\in\bS^1$, we denote the counter-clockwise arc of $\bS^1$ connecting $a, b$ by $\arc{ab}$.
Suppose that $y_0$ is a \textit{periodic point} of period $n$ of a piecewise M{\"o}bius covering map $A: \mathbb{S}^1\to \mathbb{S}^1$. Then, $A^{ n}$ is orientation-preserving, and it maps an arc of the form $\arc{y_1y_0}$ to an arc of the form $\arc{y_2 y_0}$. 
We define the \textit{one-sided multipliers} of $A$ at $y_0$ to be the one-sided derivatives of $A^{ n}$:
$$
(A^{ n})'(y_0^+)= \lim_{\substack{y\to y_0\\ \tiny{y\in \arc{y_0\widetilde{y}}}}} \frac{A^{ n}(y)-y_0}{y-y_0},\ \quad\ 
(A^{ n})'(y_0^-)=  \lim_{\substack{y\to y_0\\ \tiny{y\in \arc{\widetilde{y} y_0}}}} \frac{A^{ n}(y)-y_0}{y-y_0},
$$
where $\widetilde{y}\neq y_0$ is any point on $\mathbb{S}^1$. See \cite[Section 2]{mj-muk-1} for properties of one-sided  multipliers of $A$.

	Let $x$ be a periodic point (of period $n$)  of a piecewise M{\"o}bius, expansive circle covering $A$. Then $x$ is said to be \textit{parabolic on the right} (resp., \textit{on the left}) if $(A^{ n})'(x^+)=1$ (resp., $(A^{ n})'(x^-)=1$). Likewise, $x$ is \textit{hyperbolic on the right} (respectively, \textit{on the left}) if $(A^{ n})'(x^+)>1$ (resp., $(A^{ n})'(x^-)>1$). 
	Also, $x$ is \textit{symmetrically parabolic} (respectively, \textit{symmetrically hyperbolic}) if $(A^{ n})'(x^+)=(A^{ n})'(x^-)=1$ (respectively, if $(A^{ n})'(x^+)=(A^{ n})'(x^-)>1$). 
	The point
	$x$ is called \textit{asymmetrically hyperbolic} if it is hyperbolic on both sides, but $(A^{ n})'(x^+)\neq (A^{ n})'(x^-)$.
		Finally, $x$ is said to be a \textit{periodic point of mixed type} if it is hyperbolic on one side, but parabolic on the other.

\begin{lemma}\cite[Lemma~2.15]{mj-muk-1}\label{no_mixed}
	Let $A:\mathbb{S}^1\to\mathbb{S}^1$ be a piecewise Fuchsian expansive covering map having $x_1,\cdots, x_k$ as the break-points of its piecewise definition. Further, let $x_j$ be a periodic point of $A$. Then, $x_j$ is not of mixed type.
\end{lemma}

\begin{defn}\label{markov_def}
Let $X$ be a topological space and $f : X \rightarrow X$ be a continuous map. A collection of closed subsets $\lbrace X_1, X_2, \cdots X_n\rbrace$ of $X$ is called a Markov partition for $(X,f)$ if the following properties are satisfied:
\begin{enumerate}
\item $X=\cup_{i=1}^n X_i$,
\item $\Int{X_i} \cap \Int{X_j} = \emptyset$ for $i\neq j$,
\item $\overline{\Int{X_i}} = X_i$ for $i\in\{1,2,\cdots,n\}$,
\item $f\vert_{X_i}$ is injective, and
\item if $f(\Int{X_i}) \cap \Int{X_j} \neq \emptyset$, then $f(X_i) \supset X_j$.
\end{enumerate}
\end{defn}
It is well-known that continuous, open and distance expanding self-maps of compact metric spaces admit Markov partitions (see \cite[\S 3]{PU}). In particular, the polynomial map $z \mapsto z^d$, restricted to the unit circle $\mathbb{S}^1$, admits a Markov partition (in fact, explicit Markov partitions for $z^d$ can be easily constructed).

	We call $A: \bS^1 \to \bS^1$ a \emph{\pwfm} map if it is a piecewise Fuchsian expansive covering map (of degree $d$ at least two) such that the pieces $I_j$ (intervals of definition) of $A$ in $\bS^1$ give a Markov partition for $A: \bS^1 \to \bS^1$. The restrictions $A|_{I_j} = g_j (\in \Gamma_A)$ of $A$ to  $I_j$ will be  referred to 
 as \emph{pieces} of $A$.

By the Markov property of $A$, each interval $I_j$
has exactly $d$ pre-images under $A$.
This gives us a natural  transition matrix for $A^{-1}$ given by $a_{jl} = 1$ if there exists a point in the interior of $I_l$ mapped to $I_j$ under $A$, and $a_{jl}=0$ otherwise. Further, there is a naturally associated
topological Markov chain, which we now describe (compare \cite{series-etds, thurston-wordp}). We construct a $d-$regular
directed graph $\GG$ with $k$ vertices (one for each $I_j$) and a directed
edge from vertex $j$ to vertex $l$ if and only if $a_{jl} = 1$. Further, we label such a directed edge from $j$ to $l$ by $g_l^{-1}$ (since the piece of $A$ on $I_l$ is $g_l$, the inverse branch from $I_j$ to $I_l$ is $g_l^{-1}$). Note that there are exactly $d$ \emph{branches} of $A^{-1}$ at each interior point of an $I_j$ and any such  branch is given by the inverse of one of the pieces of $A$; i.e., for each piece
$g_i$ of $A$, $g_i^{-1}$ is a label of some edge of $\GG$ and each label of an edge of $\GG$ is of this form.  

We now follow a point $z\in \bS^1$ under backward iteration of $A$. Let $\{z=z_0, z_1, \cdots\}$
be a (finite or infinite) sequence of points in $\bS^1$ such that $A(z_{i+1})= z_i$. Then any such sequence
encodes a geodesic in $\GG$; i.e., an isometric immersion of an interval $[0,a]$, or $[0, \infty)$ into $\GG$ such that
$[i,i+1]$ maps isometrically to an edge of $\GG$ labeled by (the unique) $g$ satisfying by the following:
\begin{enumerate}
	\item $z_i\in I_{j(i)}$.
	\item $z_{i+1}\in I_{j(i+1)}$.
	\item $A$ restricted to $I_{j(i+1)}$ equals $g^{-1}$.
	\item $g(z_i) = z_{i+1}$.
\end{enumerate}

The labeled directed graph $\GG$ (also known as a topological Markov chain) imposes a structure
akin to that of 
an automatic group \cite{thurston-wordp} on backward orbits of points via
backward orbits of intervals $I_j$. Thus, a sequence of backward orbits of an interval $I_j$ may be given by $I_j = I_{j(0)}, I_{j(1)}, \cdots, I_{j(n)}, \cdots$ such that $I_{j(i)} \subset A(I_{j(i+1)})$.
This sequence is also encoded by the same geodesic in $\GG$ described above, since the pair $\{I_{j(i)}, I_{j(i+1)}\}$ corresponds to a unique edge
in $\GG$, and the label on the edge is the unique $g\in \Gamma$ such that $g^{-1}$ is a piece of $A$
satisfying $I_{j(i)} \subset g^{-1}(I_{j(i+1)})$.

A more concise version of Definition \ref{def-mateable} can now be furnished by saying that 

\begin{defn}\label{def-mateable2}
	A \pwfm map $A: \bS^1 \to \bS^1$ is  \emph{mateable} if
	$A$ is orbit equivalent to the Fuchsian group $\Gamma_A$ generated by its pieces, and
		 none of the periodic break-points of $A$ are asymmetrically hyperbolic.
\end{defn}

We finally provide the lemma promised before Definition \ref{def-mateable} guaranteeing that 
$\Gamma_A$ is a lattice.
\begin{lemma}\cite[Lemma~2.18]{mj-muk-1}\label{lem-gammalattice}
	If $A$ is mateable, then $\Gamma_A$ is a lattice (or equivalently, $\Gamma_A$ is a finitely generated Fuchsian group such that $\disk/\Gamma_A$ has finite hyperbolic area).
\end{lemma}

For a complex polynomial $P$, its \emph{filled Julia set} $\mathcal{K}(P)$ is the completely invariant set of all points whose forward orbits (under $P$) stay bounded. A polynomial is said to be \emph{hyperbolic} if each of its critical points converges to an attracting cycle under forward iteration. The set of all hyperbolic polynomials (of a given degree) is open in the parameter space. A connected component of degree $d$ hyperbolic polynomials is called a \emph{hyperbolic component} in the parameter space of degree $d$ polynomials. The hyperbolic component of degree $d$ polynomials containing the map $z^d$ is called the \emph{principal hyperbolic component}, and is denoted by $\mathcal{H}_d$. The filled Julia set of each map in $\mathcal{H}_d$ is a quasidisk, and the dynamics of such a map on its Julia set is quasisymmetrically conjugate to the action of $z^d$ on $\bS^1$.

The next proposition says that the conditions of Definition \ref{def-mateable2} are sufficient to guarantee conformal mateability of \pwfm maps and polynomials in principal hyperbolic components.

For a Jordan curve $\mathfrak{J}$ on the Riemann sphere, we denote its complementary components by $\mathbf{D}^\textrm{in}$ and $\mathbf{D}^\textrm{out}$. The canonical extension $\widehat{A}:\mathcal{D}\to\overline{\disk}$ of a mateable map is said to be \emph{conformally mateable} with a polynomial $P$ in a principal hyperbolic component if there exist a holomorphic map $F$ defined on a subset of $\widehat{\C}$, a Jordan curve $\mathfrak{J}\subset\mathrm{Dom}(F)$, and a pair of conformal maps $\phi^\textrm{in}:\overline{\disk}\to\overline{\mathbf{D}^\textrm{in}}$ and $\phi^\textrm{out}:\mathcal{K}(P)\to\overline{\mathbf{D}^\textrm{out}}$ that conjugate $\widehat{A}$ and $P$ (respectively) to $F$. The following is the first main result of \cite{mj-muk-1}.

\begin{prop}[\bf{Mateable maps are mateable}]\cite[Proposition~2.23]{mj-muk-1}\label{conformal_mating_general_prop}
	Let $A:\mathbb{S}^1\to\mathbb{S}^1$ be a mateable map of degree $d$, and $P\in\mathcal{H}_d$.
	Then, the maps $\widehat{A}:\mathcal{D}\to\overline{\disk}$ and $P:\mathcal{K}(P)\to\mathcal{K}(P)$ are conformally mateable. 
\end{prop}

\begin{rmk}
	A mateable map may have parabolic fixed points on $\bS^1$, and hence the topological conjugacy between $z^d$ and $A$ is not necessarily quasisymmetric. This renders classical quasiconformal tools (such as the ones used in the proof of Bers simultaneous uniformization
	theorem) insufficient for the purpose of conformally mating polynomials with mateable maps associated with Fuchsian groups. However, an appropriate class of `generalized quasiconformal maps', called David homeomorphisms (maps with suitable Sobolev regularity satisfying a quantitative control on the area of the region where the dilatation blows up), allows one to perform the conformal mating construction. Two results that lie at the analytic heart of the proof of Proposition~\ref{conformal_mating_general_prop} are the David integrability theorem (this can be seen as a generalization of the measurable Riemann mapping theorem, see \cite{David88}, \cite[Theorem~20.6.2]{AIM09}) and a David extension theorem for certain circle homeomorphisms (which plays the role of the Ahlfors-Beurling extension theorem in the current setting, see \cite[Theorem~4.9]{LMMN}). In fact, Item~\ref{asym_hyp} in Definition \ref{def-mateable} is required to guarantee the existence of a David extension of a circle homeomorphism conjugating $z^d$ to a mateable map.
\end{rmk}

\section{Bowen-Series maps of Fuchsian punctured sphere groups}\label{bs_sec}

\subsection{Bowen-Series maps for General Fuchsian groups} Archetypal
examples of \pwfm maps of the circle that are orbit equivalent to finitely generated Fuchsian groups are given by \emph{Bowen-Series maps}. These first  appeared in the work of Bowen and Series \cite{bowen,bowen-series}.

A finitely generated Fuchsian group $\Gamma$ (of the first kind) admits a fundamental domain $R\left(\subset\disk\right)$ that is a (possibly ideal) hyperbolic polygon. Denote the edges of $R$ by $\{s_i\}_{i=1}^n$ (labeled in counter-clockwise order around the circle). Each edge $s_i$ of $R$ is identified with another edge $s_j$ by a corresponding element $h(s_i)\in\Gamma$. The set $\{h(s_i)\}_{i=1}^n$ forms a generating set for $\Gamma$.

Let $C(s_i)$ be the Euclidean circular arc in $\disk$ containing $s_i$  and meeting  $\bS^1$ orthogonally. Further, let $\mathcal{N}$ be the net in $\disk$ consisting of all images of edges of $R$ under elements of $\Gamma$. The fundamental domain $R$ is said to satisfy the \emph{even corners} property if $C(s_i)$ lies completely in $\mathcal{N}$, for $i\in\{1,\cdots, n\}$.

\begin{defn}[\bf{Bowen-Series map}]\label{b_s_map_def}
	Suppose that a fundamental domain $R$ of $\Gamma$ satisfies the even corners property. Label (following \cite{bowen-series}) the endpoints of $C(s_i)$ on $\bS^1$, $P_i, Q_{i+1}$ (with $Q_{n+1}=Q_1$) with $P_i$ occurring before $Q_{i+1}$ in the counter-clockwise order. These points occur along the circle in the order $P_1, Q_1, P_2, Q_2$, $\cdots$, $P_n, Q_n$ (see Figure~\ref{closed_bs_fig}). The \emph{Bowen-Series map} $A_{\Gamma, \textrm{BS}}:\bS^1\to\bS^1$ of $\Gamma$ (associated with the fundamental domain $R$) is defined piecewise as $A_{\Gamma, \textrm{BS}}\equiv h(s_i)$, on the sub-arc $[P_i, P_{i+1})$ of $\bS^1$ (traversed in the counter-clockwise order).
\end{defn}

\begin{prop}\cite[Lemma~2.4]{bowen-series}\label{orbit_equiv_prop}
	The map $A_{\Gamma, \textrm{BS}}$ is orbit equivalent to $\Gamma$, except (possibly) at finitely many  points modulo the action of $\Gamma$.
\end{prop}

We shall simply denote $A_{\G, \textrm{BS}}$  by $A_{\G}$. 
The Bowen-Series maps corresponding to Fuchsian groups uniformizing positive genus surfaces (possibly with punctures) are discontinuous. Let us illustrate this with two examples.
In the left diagram in Figure~\ref{closed_bs_fig}, $R$ is a fundamental domain for a (closed) genus two surface where the color coding determines the side-pairings. Note that $h(s_1)(x)=y$, and $h(s_2)(P_2)=Q_5$. Thus, for continuity of the corresponding Bowen-Series map at $P_2$, the map $h(s_1)$ must send the geodesic ray from $x$ to $P_2$ to the geodesic ray from $y$ to $Q_5$. But the former ray lies in the net $\mathcal{N}$ (by the even corners property), while the latter ray passes through $\Int{R}$. This is absurd as $R$ is a fundamental domain, proving discontinuity of the Bowen-Series map at $P_2$.
In the right diagram in the above figure, $R$ is a fundamental domain for a once punctured torus where the sides are paired according to their colors. The side-pairing transformations $h(s_1)$ maps $P_2$ to $P_3$, while $h(s_2)$ carries $P_2$ to $P_1$. This causes discontinuity of the associated Bowen-Series map at $P_2$. 

Thus, to get continuous Bowen-Series maps, we need to restrict our attention to punctured sphere groups (possibly with orbifold points) equipped with special fundamental domains. In fact, it turns out that the Bowen-Series maps of Fuchsian punctured sphere groups constructed below are coverings of $\mathbb{S}^1$ with degree at least two.

\begin{figure}[h!]
	\captionsetup{width=0.96\linewidth}

	\begin{tikzpicture}
		\node[anchor=south west,inner sep=0] at (0,0) {\includegraphics[width=0.46\linewidth]{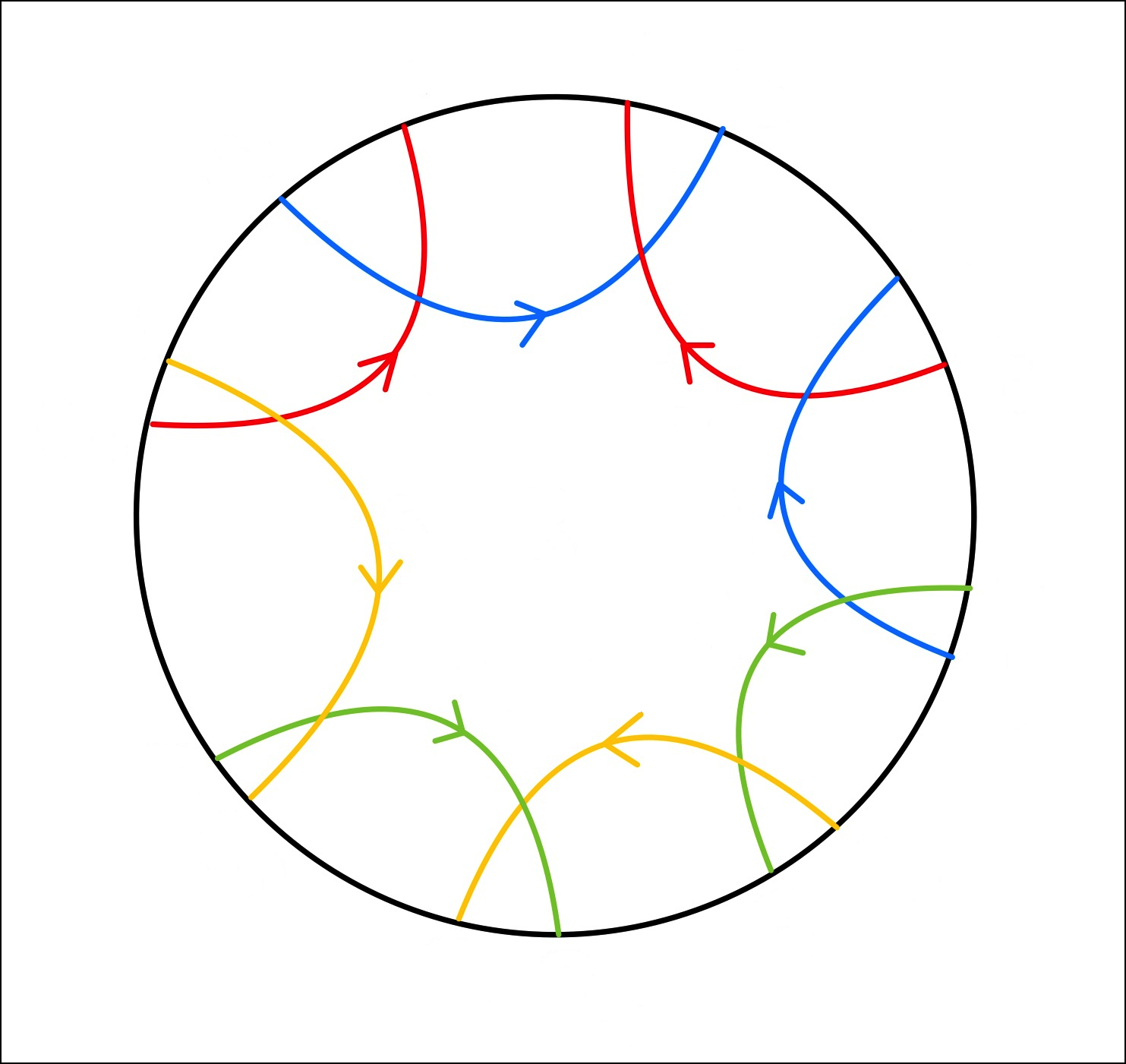}};
		\node[anchor=south west,inner sep=0] at (6,0) {\includegraphics[width=0.5\linewidth]{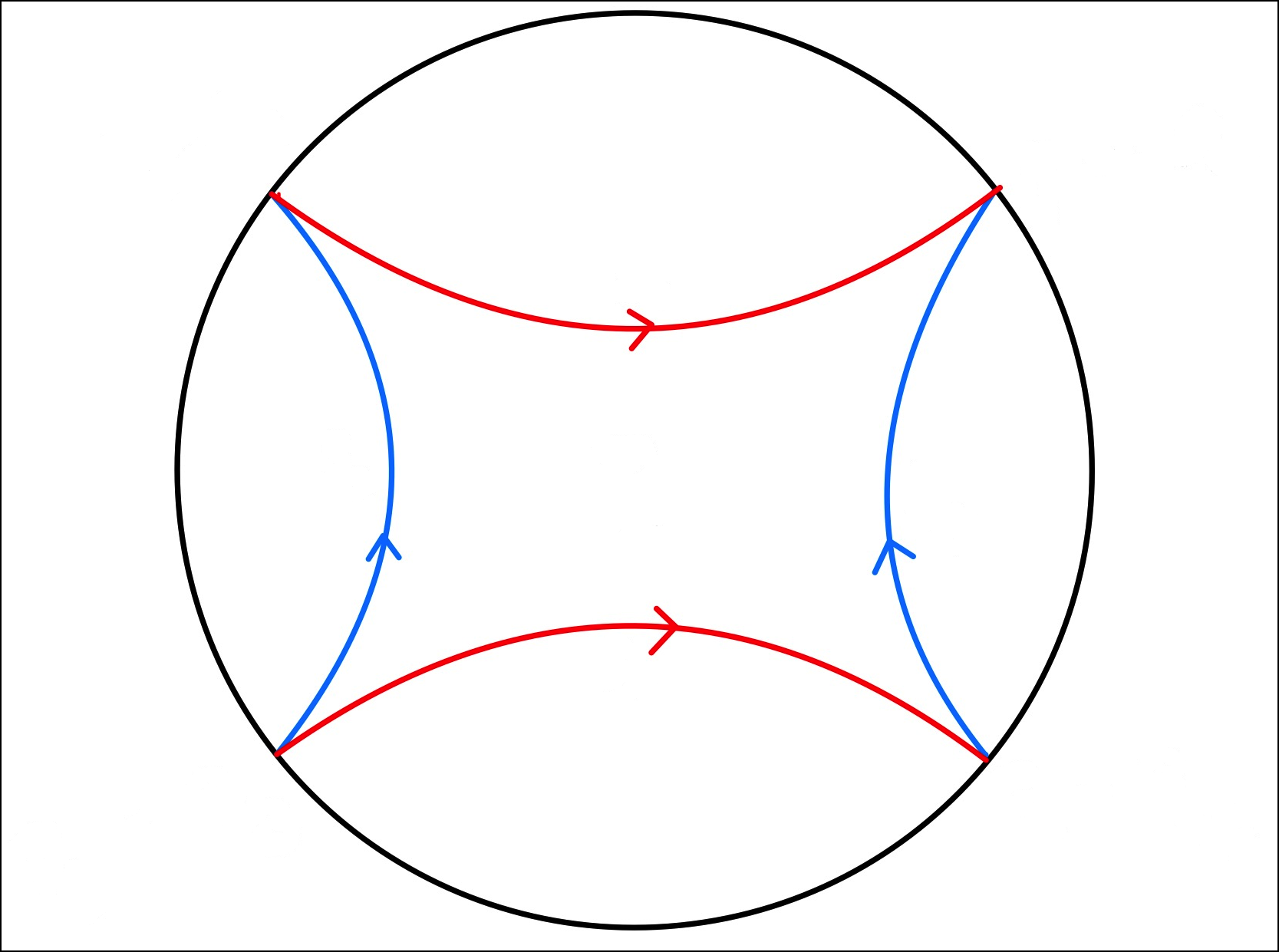}};
		\node at (1.9,5.09) {\begin{small}$P_1$\end{small}};
		\node at (1.4,4.8) {\begin{small}$Q_1$\end{small}};
		\node at (0.6,3.88) {\begin{small}$P_2$\end{small}};
		\node at (0.48,3.22) {\begin{small}$Q_2$\end{small}};
		\node at (0.86,1.54) {\begin{small}$P_3$\end{small}};
		\node at (1.16,1.16) {\begin{small}$Q_3$\end{small}};
		\node at (2.2,0.6) {\begin{small}$P_4$\end{small}};
		\node at (2.9,0.44) {\begin{small}$Q_4$\end{small}};
		\node at (4.08,0.7) {\begin{small}$P_5$\end{small}};
		\node at (4.48,1.04) {\begin{small}$Q_5$\end{small}};
		\node at (5.16,1.9) {\begin{small}$P_6$\end{small}};
		\node at (5.32,2.4) {\begin{small}$Q_6$\end{small}};
		\node at (5.25,3.6) {\begin{small}$P_7$\end{small}};
		\node at (4.9,4.2) {\begin{small}$Q_7$\end{small}};
		\node at (3.88,5) {\begin{small}$P_8$\end{small}};
		\node at (3.3,5.18) {\begin{small}$Q_8$\end{small}};
		\node at (1.4,3.1) {$x$};
		\node at (4.35,3.28) {$y$};
		\node at (2.8,2.8) {$R$};
		\node at (2.2,3.5) {\begin{small}$s_1$\end{small}};
		\node at (2.15,2.8) {\begin{small}$s_2$\end{small}};
		\node at (2.3,2.04) {\begin{small}$s_3$\end{small}};
		\node at (3,1.88) {\begin{small}$s_4$\end{small}};
		\node at (3.78,2.3) {\begin{small}$s_5$\end{small}};
		\node at (3.78,2.8) {\begin{small}$s_6$\end{small}};
		\node at (3.38,3.5) {\begin{small}$s_7$\end{small}};
		\node at (2.8,3.6) {\begin{small}$s_8$\end{small}};
		\node at (6.7,4) {\begin{small}$P_2=Q_2$\end{small}};
		\node at (6.7,0.8) {\begin{small}$P_3=Q_3$\end{small}};
		\node at (11.5,0.8) {\begin{small}$P_4=Q_4$\end{small}};
		\node at (11.6,3.8) {\begin{small}$P_1=Q_1$\end{small}};
		\node at (9.2,3.36) {\begin{small}$s_1$\end{small}};
		\node at (7.7,2.4) {\begin{small}$s_2$\end{small}};
		\node at (9.2,1.32) {\begin{small}$s_3$\end{small}};
		\node at (10.75,2.32) {\begin{small}$s_4$\end{small}};
		\node at (9.2,2.4) {$R$};
	\end{tikzpicture}
	\caption{Bowen-Series Maps for surfaces of higher genus}
	\label{closed_bs_fig}
\end{figure}
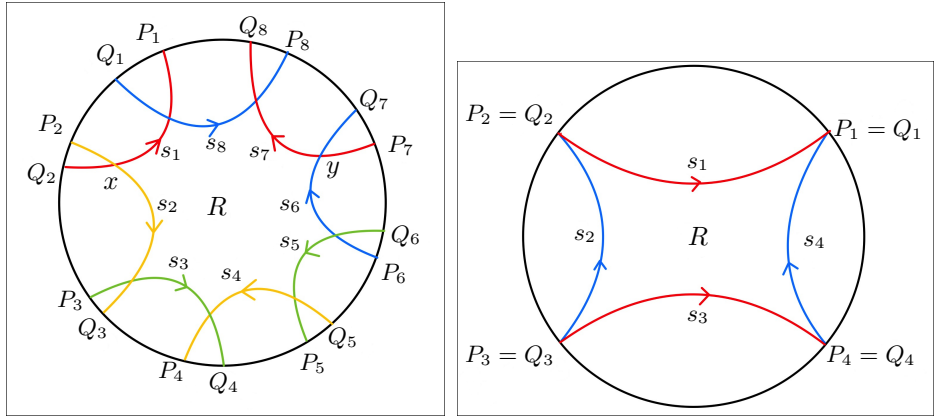

\subsection{Bowen-Series maps for punctured spheres}\label{b_s_punc_sphere_subsec}

We mention at the outset that we always associate Bowen-Series maps with Fuchsian groups decorated with preferred fundamental domains and side-pairing transformations.

We  first construct a specific Fuchsian group $G_d$ uniformizing a $(d+1)-$times punctured sphere equipped with a preferred fundamental domain.  The group $G_d$ (equipped with the preferred fundamental domain) will serve as a base-point in the Teichm{\"u}ller space of $(d+1)-$times punctured spheres. Since any (marked) group $\Gamma\in\textrm{Teich}(G_d)$ is conjugate to $G_d$ via a quasiconformal homeomorphism of $\widehat{\C}$, the Bowen-Series map of $\Gamma$ equipped with a marked fundamental domain determined by the quasiconformal conjugacy is easily seen to be a quasiconformal conjugate of the Bowen-Series map of $G_d$. 

Fix $d\geq 2$. For $j\in\{1,\cdots, d\}$, let $C_j$ be the hyperbolic geodesic of $\disk$ connecting $p_j:=e^{\pi i (j-1)/d}$ and $p_{j+1}:=e^{\pi i j/d}$, and $C_{-j}$ be the image of $C_j$ under reflection in the real axis. We further denote the complex conjugate of $p_j$ by $p_{-j}$, $j\in\{2,\cdots, d\}$. Choose a M{\"o}bius automorphism $g_j$ of $\disk$ defined as reflection in $C_j$ followed by complex conjugation. By construction, $g_j$ carries $C_j$ onto $C_{-j}$  (cf.\ Figure~\ref{fund_dom_punctured_sphere_fig}). Note that for $j\in\{1,\cdots,d-1\}$, the M{\"o}bius map $g_{j+1} g_j^{-1}$ is the composition of reflections in the circular arcs $C_{j+1}$ and $C_{j}$. Since $C_j$ and $C_{j+1}$ touch at $p_{j+1}$, a straightforward computation (using the formula of circular reflections) shows that $g_{j+1} g_j^{-1}$ fixes $p_{j+1}$ and has derivative equal to one at this fixed point. Therefore, $g_{j+1} g_j^{-1}$ is parabolic with its unique fixed point at $p_{j+1}$. Likewise, the maps $g_1, g_d$ fix $p_1, p_{d+1}$ (respectively), and have derivative equal to one there. Thus, $g_1, g_d$ are also parabolic with their unique fixed points at $p_1, p_{d+1}$, respectively. Let
$$
G_d:=\langle g_1,\cdots, g_d\rangle.
$$
We note that $G_d$ is a Fuchsian group with fundamental domain $R$ having $C_1,\cdots, C_d,$ $C_{-d}, \cdots, C_{-1}$ as its edges. Moreover, $\disk/G_d$ is a $(d+1)$-times punctured sphere.

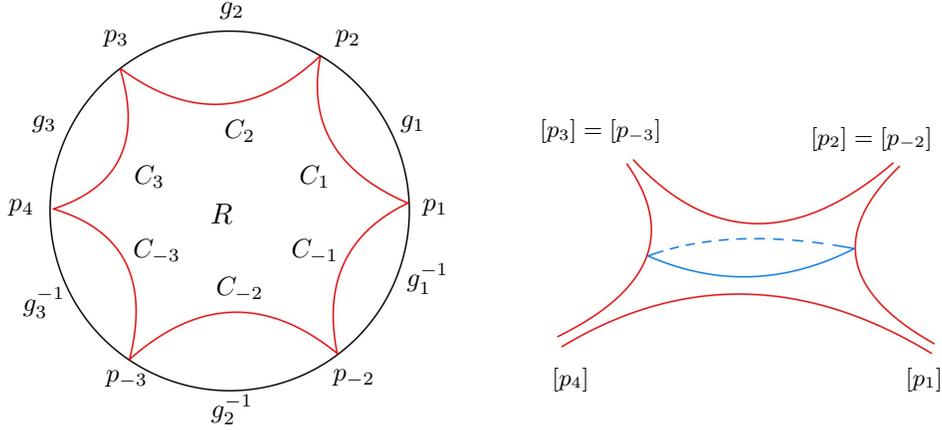
\begin{figure}[h!]
\captionsetup{width=0.96\linewidth}
	\begin{tikzpicture}
		\node[anchor=south west,inner sep=0] at (1,0) {\includegraphics[width=0.96\linewidth]{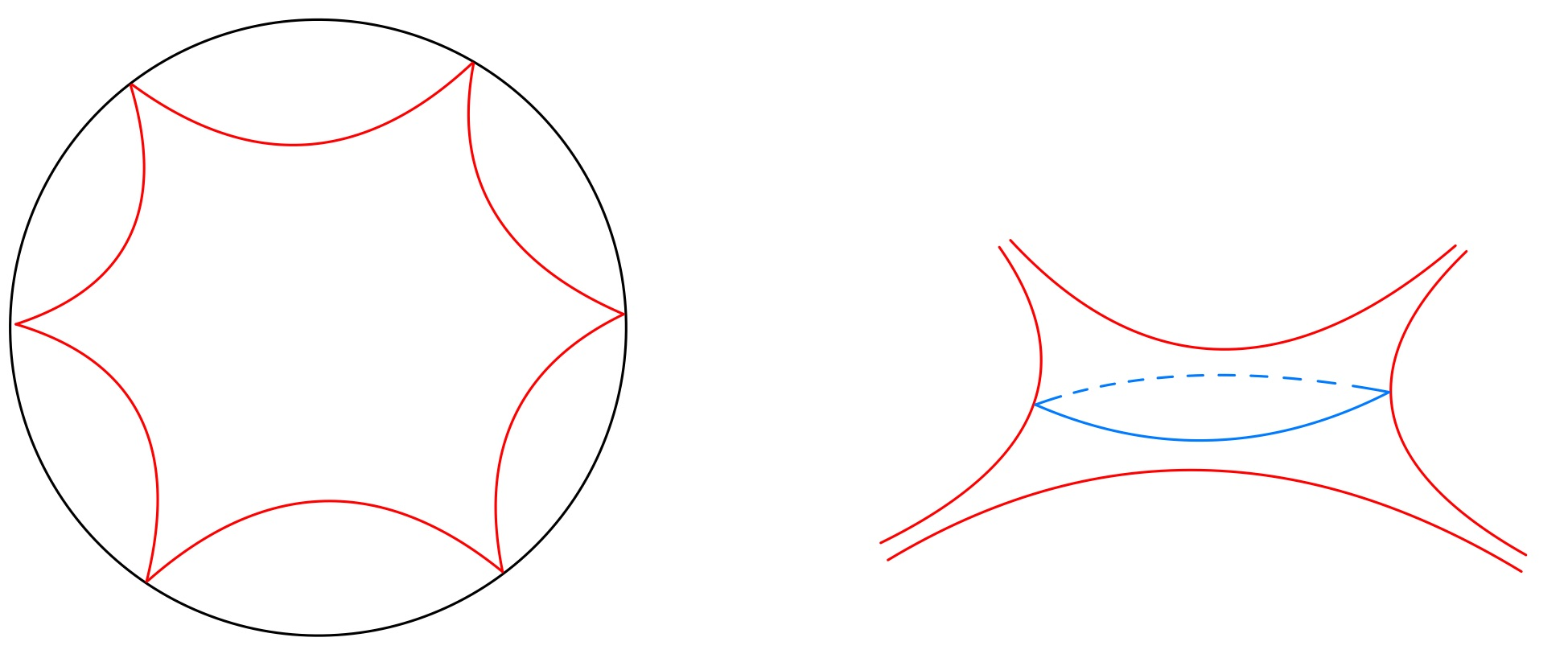}};
		\node at (3.36,2.5) {\begin{large}$R$\end{large}};
		\node at (4.6,3) {$C_1$};
		\node at (4.6,2) {$C_{-1}$};
		\node at (3.6,3.6) {$C_2$};
		\node at (3.6,1.5) {$C_{-2}$};
		\node at (2.4,3) {$C_3$};
		\node at (2.5,2) {$C_{-3}$};
		\node at (5.9,3.7) {$g_1$};
		\node at (6.1,1.64) {$g_{1}^{-1}$};
		\node at (3.5,5.2) {$g_2$};
		\node at (3.5,-0.1) {$g_{2}^{-1}$};
		\node at (1,3.7) {$g_3$};
		\node at (1,1.32) {$g_{3}^{-1}$};
		\node at (6.2,2.58) {$p_1$};
		\node at (5.04,4.84) {$p_{2}$};
		\node at (1.96,4.84) {$p_{3}$};
		\node at (0.7,2.58) {$p_{4}$};
		\node at (5.12,0.32) {$p_{-2}$};
		\node at (2.1,0.3) {$p_{-3}$};
		\node at (8.4,3.6) {\begin{small}$\left[p_3\right]=\left[p_{-3}\right]$\end{small}};
		\node at (12,3.5) {\begin{small}$\left[p_2\right]=\left[p_{-2}\right]$\end{small}};
		\node at (8,0.3) {\begin{small}$\left[p_4\right]$\end{small}};
		\node at (12.7,0.3) {\begin{small}$\left[p_1\right]$\end{small}};
	\end{tikzpicture}
	\caption{The preferred fundamental domain $R$ of $G_3$, which uniformizes a four times punctured sphere, is shown. The fundamental domain has all six vertices on $\mathbb{S}^1$, and they cut the circle into six arcs. The corresponding Bowen-Series map acts on these arcs by the generators $g_j^{\pm 1}$ displayed next to them.}
	\label{fund_dom_punctured_sphere_fig}
\end{figure}

We refer the reader to Figure \ref{fund_dom_punctured_sphere_fig}.
For $j\in\{1,\cdots, d\}$, let $I_j$ denote the counterclockwise sub-arc of $\mathbb{S}^1$ connecting $p_j$ to $p_{j+1}$. Let $I_{-j}$ denote the image of $I_j$ under reflection in the real axis. Note that the Bowen-Series map $A_{G_d}$ of $G_d$ (equipped with the fundamental domain $R$) acts on $I_{\pm j}$ by $g_j^{\pm 1}$. The following two properties hold.

\begin{prop}[\bf{Properties of Bowen-Series maps of punctured spheres}]\cite[Proposition~3.3]{mj-muk-1}\label{b_s_poly_conjugate_prop_1}
\noindent\begin{enumerate}\upshape
\item For $d\geq 2$, the Bowen-Series map $A_{G_d}$ of $G_d$ (equipped with the fundamental domain $R$) is a $C^1$ expansive degree $2d-1$ covering of $\mathbb{S}^1$, and hence is topologically conjugate to $z^{2d-1}\vert_{\mathbb{S}^1}$. Moreover, $A_{G_d}$ is a \pwfm map.
	
\item $A_{G_d}$ is orbit equivalent to $G_d$ on $\bS^1$.
\end{enumerate}
\end{prop}

We refer the reader to \cite[Propositions~3.4, 3.5]{mj-muk-1} for an orbifold variant of Proposition \ref{b_s_poly_conjugate_prop_1}.

\begin{rmk}
In  the above examples, the chosen fundamental domains of the groups coincide with those of the corresponding Bowen-Series maps.
\end{rmk}

\subsection{Mateability of Bowen-Series maps}\label{sec-puncturedspherebuildingblock}
We note now that Bowen-Series maps for punctured spheres fit into our mating framework. 
Recall that $\mathcal{H}_k$ stands for the principal hyperbolic component in the space of degree $k$ polynomials.

\begin{theorem}[\bf{Fuchsian punctured sphere Bowen-Series maps are mateable}]\cite[Theorem~3.7]{mj-muk-1}\label{moduli_interior_mating_thm}
	Let $\Gamma\in\mathrm{Teich}(G_d)$, and $P\in\mathcal{H}_{2d-1}$. Then, the map $\widehat{A}_{\Gamma}:\mathcal{D}_{A_{\Gamma}}\to\overline{\disk}$ and $P:\mathcal{K}(P)\to\mathcal{K}(P)$ are conformally mateable.
\end{theorem}

\section{Folding and higher Bowen-Series maps for Fuchsian groups}\label{sec-fold}
The aim of this section is to describe a new class of \pwfm maps (following \cite{mj-muk-1}), beyond the Bowen-Series examples that are mateable with polynomials. We start with classes of maps that we shall be considering in this section.
Recall that the fundamental domain of a \pwfm map $A$ is denoted by $R$. The set $\mathcal{D}=\overline{\disk}\setminus R$ is the canonical domain of definition of $\widehat{A}$ in $\overline{\disk}$, and a bi-infinite geodesic in $R$ joining a pair of non-adjacent vertices of $R$ is called a \emph{diagonal} of $R$. 

It is instructive to go through the following two definitions in conjunction with the two explicit examples of \pwfm maps given in Subsection~\ref{sec-cfm} (cf. Figure~\ref{cfd}).

\begin{defn}[\bf{Completely folding map}]\label{def-cfm}
	A \pwfm map $A: \bS^1 \to \bS^1$ is said to 
	be a \emph{\cfm} if there exist  finitely many diagonals $\delta_1, \cdots, \delta_l$ of $R$ such that the following hold:
	\begin{enumerate}
		\item For every edge $\alpha$ of $R$, $\widehat{A}(\alpha)$ is one of the  diagonals $\delta_1, \cdots, \delta_l$.
		\item The ideal endpoints $p_i, q_i$ of $\delta_i$ are fixed points of $A$ for all $i$; i.e., $A(p_i) = p_i$ and $A(q_i) = q_i$ whenever $p_i, q_i$ 
		are ideal endpoints of $\delta_i$.
		\item For $p_i, q_i$ as above, $q_i = p_{i+1}$. 
		\item $\delta_i \cap \delta_j = \emptyset$ for $i \neq j$. Further, $p_1 \neq q_l$; i.e., the sequence of diagonals $\delta_i$ forms a chain of non-intersecting
		bi-infinite geodesics such that, after adjoining the ideal endpoints, one obtains a `piecewise geodesic' embedding of the closed interval $[0,1]$
		in the closed disk $\overline{\disk}$.
	\end{enumerate}
\end{defn}

\begin{defn}[\bf{Higher degree map without folding}]\cite[Definition 4.2]{mj-muk-1}\label{def-nofold}
	A \pwfm map $A: \bS^1 \to \bS^1$ is said to have a \emph{diagonal fold} if there
	exist consecutive edges $\alpha_1, \alpha_2$ of $\partial R$
	and a diagonal $\delta$ of $R$ such that $\widehat{A}(\alpha_i) = \delta$ for $i=1,2$. Note that if $a_1, a_2$ (resp. $ a_2, a_3$) are the endpoints of $\alpha_1$ (resp. $\alpha_2$) and $p, q$ are the endpoints of $\delta$, then
	$A(a_1)=p=A(a_3)$ and $A(a_2)=q$ by continuity of $A$ on $\bS^1$.
	
	A \pwfm map $A: \bS^1 \to \bS^1$ is said to be a \emph{\hdm} if
	\begin{enumerate}
		\item there exists an (open) ideal polygon $D \subset R$ such that all the edges 
		$\delta_1, \cdots, \delta_l$ of $D$ are  (necessarily non-intersecting) diagonals of $R$. We assume further that $\delta_1, \cdots, \delta_l$ are  cyclically ordered along $\partial D$. We shall call $D$ the \emph{inner domain} of $A$.
		\item If $p$ is an ideal vertex of $D$, then $A(p)=p$.
		\item For every edge $\alpha$ of $R$, $\widehat{A}(\alpha)$ is one of the diagonals
		$\delta_1, \cdots, \delta_l$.
		\item $A$ has no diagonal folds. 
	\end{enumerate}
\end{defn}

Cyclically ordering the edges $\alpha_1, \cdots, \alpha_k$ of $R$, it follows
from Definition \ref{def-nofold}, that under a \hdm $A$, consecutive edges
$\alpha_i, \alpha_{i+1}$ of $R$ go to consecutive edges of $D$. Note however that an counter-clockwise cyclic ordering of edges of $R$ may be taken to a 
clockwise cyclic ordering of edges of $D$ under $A$.
In any case we have a continuous map $\widehat{A}: \partial R \to \partial D$.
Adjoining the ideal endpoints of $R$ and $D$, $\widehat{A}$ has a well-defined
degree $d$. 
Further, each edge of $D$ has exactly $|d|$ pre-images under $\widehat{A}$ since there are no folds. Also, since each $\delta_i$ is a diagonal of $R$, we have 
$|d| >1$. We call $|d|$ the \emph{polygonal degree} of $A$. (Since $|d|>1$, we call $A$ a \hdm.)

\begin{rmk}
A \pwfm map with a diagonal fold need not be a \cfm; see Subsection~\ref{noe_2_subsubsec} for an example.
\end{rmk}

\subsection{A \cfm and a \hdm for the sphere with three punctures}\label{sec-cfm}

We now give two simple examples: a 
\cfm and a \hdm  which are orbit equivalent to $\Gamma_0$ corresponding to a sphere with three punctures. Then $\Gamma_0 $ is isomorphic to $F_2$, the fundamental group of
$S_{0,3}$ (see Figure \ref{cfd}). We will denote a bi-infinite hyperbolic geodesic in $\disk$ having its (ideal) endpoints at $a, b\in\mathbb{S}^1$ by $\overline{ab}$.

Fix a (closed) fundamental domain $W$ of $\Gamma_0$, given by an ideal quadrilateral with its ideal vertices at the fourth roots of unity (the quadrilateral $1236$ in the figure).
The generators of $\Gamma_0$ are given by $h, g$, where $h$ takes the edge
$\overline{12}$ to $\overline{16}$, $g$ takes $\overline{32}$ to $\overline{36}$, and $g^{-1}h$ is parabolic.  The combinatorics in this case is relatively simple and the case-by-case analysis for proving orbit equivalence in Proposition \ref{prop-cfdhd-orbeq} is easy. \\

\subsubsection{A completely folding map for $S_{0,3}$}\label{cfm_special_subsec} We shall first construct a \cfm,
and then modify the construction slightly to obtain a \hdm. We define the fundamental domain
$R$ of the \cfm $A_{\Gamma_0, \mathrm{cfm}}$ (to be constructed) as 
$$
R= \Int{\left(W \cup h.W \cup g.W\right)}.
$$
Thus, $R$ is the interior of the octagon $12345678$ in Figure \ref{cfd}. We define the pieces
of $A_{\Gamma_0, \mathrm{cfm}}$ as follows. In the list below, an arc will be indicated by $\arc{ij}$ where the pair of numbers $i, j$ are its endpoints, provided there are no other break-points of $A_{\Gamma_0, \mathrm{cfm}}$ in the arc. 
Otherwise, we will denote the arc by all the break-points it contains. Further 
the label of the arrow will denote the piece of $A_{\Gamma_0, \mathrm{cfm}}$ that takes the domain arc to the range arc.

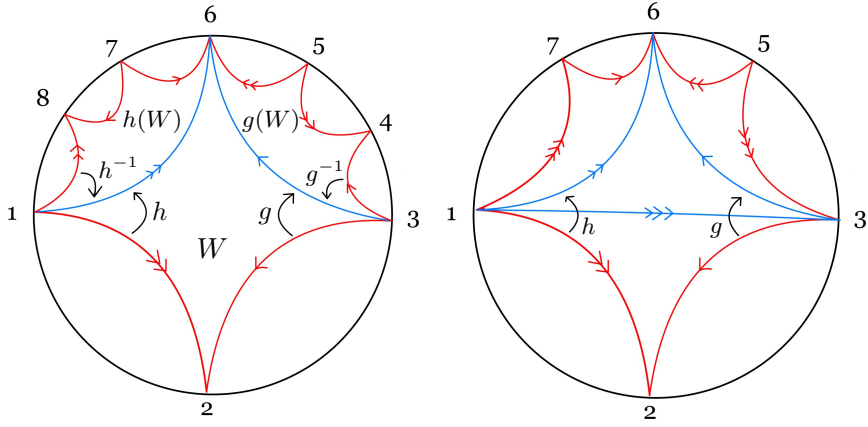
\begin{figure}[ht!]
\captionsetup{width=0.96\linewidth}
	\begin{tikzpicture}
		\node[anchor=south west,inner sep=0] at (0,0) {\includegraphics[width=0.47\linewidth]{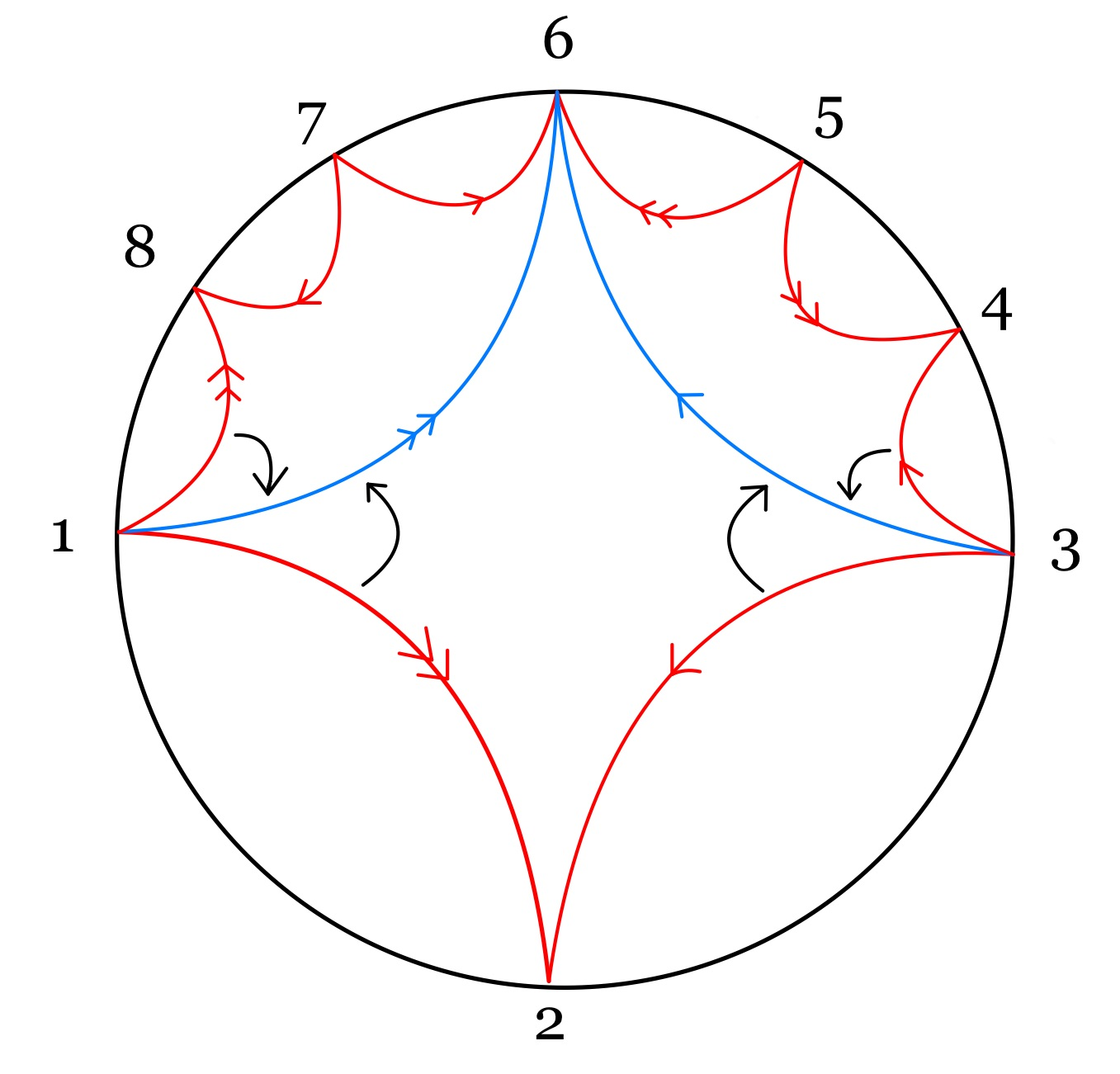}};
		\node[anchor=south west,inner sep=0] at (6,0) {\includegraphics[width=0.46\linewidth]{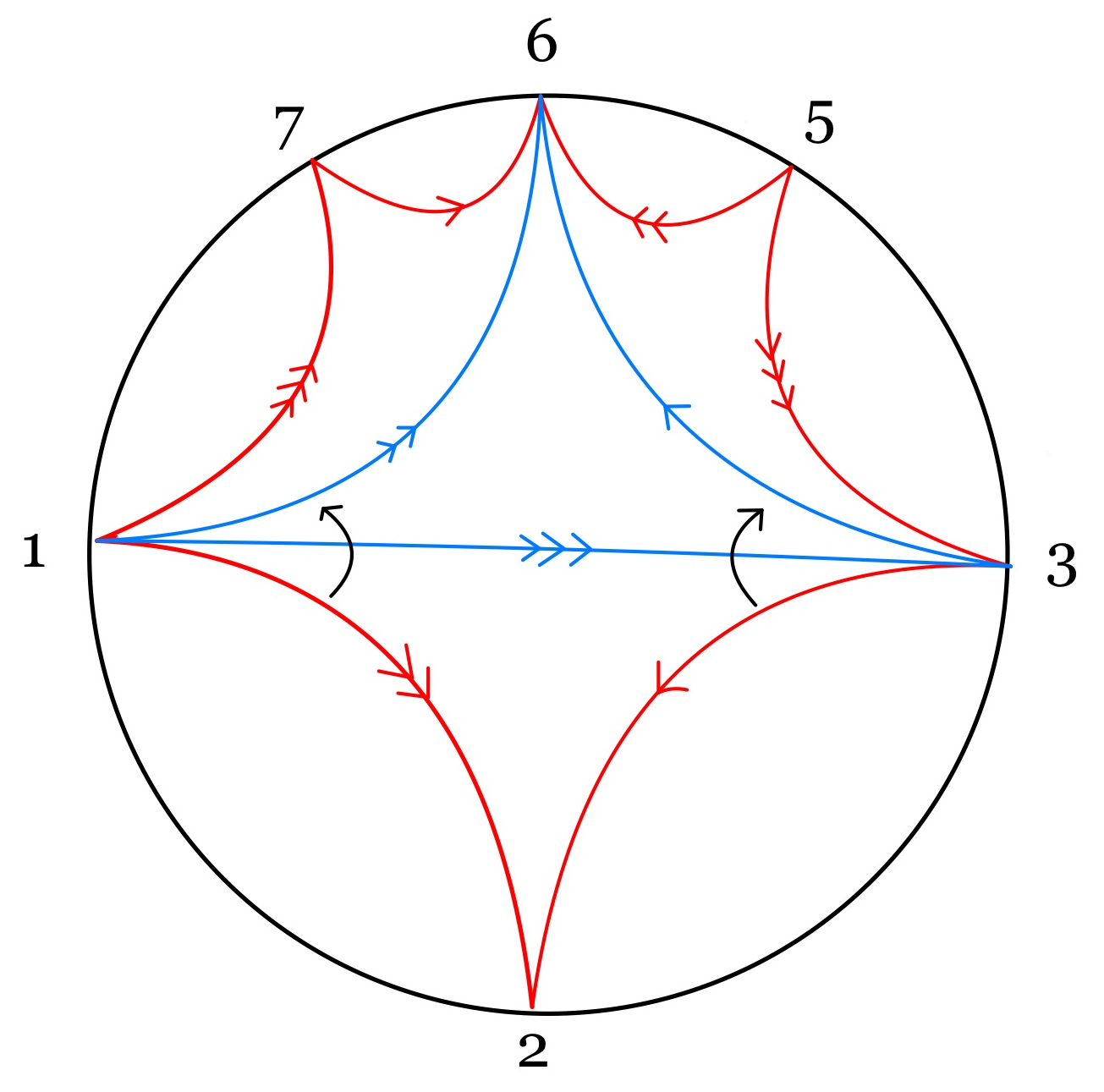}};
		\node at (2.3,2.9) {$h$};	
		\node at (3.7,2.84) {$g$};	
		\node at (1.75,3.48) {\begin{small}$h^{-1}$\end{small}};	
		\node at (4.5,3.4) {\begin{small}$g^{-1}$\end{small}};	
		\node at (3,2.4) {\begin{large}$W$\end{large}};	
		\node at (2.2,4.1) {\begin{small}$h(W)$\end{small}};	
		\node at (3.78,4.1) {\begin{small}$g(W)$\end{small}};	
		\node at (8,2.7) {\begin{small}$h$\end{small}};	
		\node at (9.7,2.64) {\begin{small}$g$\end{small}};	
		
	\end{tikzpicture}
	\caption{Fundamental domains for a completely folding map and a higher degree map without folding: $3$ punctures}
	\label{cfd}
\end{figure}

\begin{itemize}
	
	\item $ \arc{12} \stackrel{h}\longrightarrow \arc{123456}$
	\item $ \arc{23} \stackrel{g}\longrightarrow \arc{678123}$
	\item $ \arc{34} \stackrel{g^{-1}}\longrightarrow \arc{3456}$
	\item $ \arc{45} \stackrel{g^{-1}}\longrightarrow \arc{6781}$
	\item $ \arc{56} \stackrel{h\circ g^{-1}}\longrightarrow \arc{123456}$ (we use the convention that $\circ$  indicates composition of maps)
	\item $\arc{67} \stackrel{g\circ h^{-1}}\longrightarrow \arc{678123}$
	\item $\arc{78} \stackrel{h^{-1}}\longrightarrow \arc{3456}$
	\item $\arc{81} \stackrel{h^{-1}}\longrightarrow \arc{6781}$
\end{itemize}

\subsubsection{A higher degree map without folding for $S_{0,3}$}\label{hdm_special_subsec} There is a \hdm naturally associated with the \cfm above. Note that the \cfm $A_{\Gamma_0, \mathrm{cfm}}$ is not minimal. The pieces of $A_{\Gamma_0, \mathrm{cfm}}$ for the contiguous arcs $\arc{34}$ and $\arc{45}$ is 
$g^{-1}$. Similarly, the pieces of $A_{\Gamma_0, \mathrm{cfm}}$ for the contiguous arcs $\arc{78}$ and $\arc{81}$ is 
$h^{-1}$. We define $A_{\Gamma_0, \mathrm{hBS}}:\bS^1\to\bS^1$ to be the minimal \pwfm map agreeing with $A_{\Gamma_0, \mathrm{cfm}}$ everywhere (here `$\mathrm{hBS}$' is an acronym for `higher Bowen-Series', the reason behind this terminology will be explained in Remark~\ref{hbs_name_rem}). Although $A_{\Gamma_0, \mathrm{cfm}}$ and $A_{\Gamma_0, \mathrm{hBS}}$ agree pointwise, they are formally different \pwfm maps as $A_{\Gamma_0, \mathrm{cfm}}$ has more pieces (some of which are repeated). Consequently, their canonical extensions $\widehat{A}_{\Gamma_0, \mathrm{cfm}}$ and $\widehat{A}_{\Gamma_0, \mathrm{hBS}}$ have different domains of definition.

It is easy to see that $\widehat{A}_{\Gamma_0, \mathrm{hBS}}$ is a \hdm. The fundamental domain $R'$ for $\widehat{A}_{\Gamma_0, \mathrm{hBS}}$ is the interior of the ideal hexagon $123567$ contained in $R$. The inner domain of $\widehat{A}_{\Gamma_0, \mathrm{hBS}}$ is given by the ideal triangle $136$ (see Definition \ref{def-nofold}).
The pieces of $A_{\Gamma_0, \mathrm{hBS}}$ are given by the following list (note that $4, 8$ are not break-points of $A_{\Gamma_0, \mathrm{hBS}}$ and hence we omit them from the notation):
\begin{itemize}
	
	\item $\arc{12} \stackrel{h}\longrightarrow \arc{12356}$
	\item $\arc{23} \stackrel{g}\longrightarrow \arc{67123}$
	\item $\arc{35} \stackrel{g^{-1}}\longrightarrow \arc{35671}$ 
	\item $\arc{56} \stackrel{h\circ g^{-1}}\longrightarrow \arc{12356}$ 
	\item $\arc{67} \stackrel{g\circ h^{-1}}\longrightarrow \arc{67123}$
	\item $\arc{71} \stackrel{h^{-1}}\longrightarrow \arc{35671}$
\end{itemize}
The polygonal degree of $\widehat{A}_{\Gamma_0, \mathrm{hBS}}$ is $2$.

\subsubsection{Orbit equivalence}
\begin{prop}\label{prop-cfdhd-orbeq}
	Let $A_{\Gamma_0, \mathrm{cfm}},\ A_{\Gamma_0, \mathrm{hBS}}$ be as above. Then $A_{\Gamma_0, \mathrm{cfm}},\ A_{\Gamma_0, \mathrm{hBS}}$ are orbit equivalent to $\Gamma_0$.
\end{prop}
\begin{proof}
	Since $A_{\Gamma_0, \mathrm{cfm}}$ and $A_{\Gamma_0, \mathrm{hBS}}$ agree as maps
	on $\bS^1$, it suffices to check this for $A_{\Gamma_0, \mathrm{hBS}}$. It is easy to see  that $A_{\Gamma_0, \mathrm{hBS}}-$grand orbits
	are contained in $\Gamma_0-$orbits simply because the pieces of  $A_{\Gamma_0, \mathrm{hBS}}$ are elements of $\Gamma_0$. It therefore suffices to show that if $x, y$ are in the same $\Gamma_0-$orbit then they lie in the same $A_{\Gamma_0, \mathrm{hBS}}-$grand orbit.
	It suffices to check this for the generators $g, h$ and their inverses.

	Let $y = g.x$. We want to show that $x, y$ lie in the same grand orbit under $A_{\Gamma_0, \mathrm{hBS}}$.\\
	Case 1: $y \in \arc{32176}$. Then $x \in \arc{32}$ and the piece of $A_{\Gamma_0, \mathrm{hBS}}$ restricted to $\arc{32}$ is $g$. Hence $y=A_{\Gamma_0, \mathrm{hBS}}(x)$.\\
	Case 2: $y \in \arc{345}$. The branch of $A_{\Gamma_0, \mathrm{hBS}}$ restricted to $\arc{345}$ is $g^{-1}$. Rewriting $y = g(x)$ as $g^{-1}(y)=x$, we see that $A_{\Gamma_0, \mathrm{hBS}}(y) =x$.\\ 
	Case 3: $y \in \arc{56}$. Then $x \in \arc{12}$. Note that the branch of $A_{\Gamma_0, \mathrm{hBS}}$ restricted to $\arc{56}$ is $h \circ g^{-1}$, and the branch of $A_{\Gamma_0, \mathrm{hBS}}$ restricted to $\arc{12}$ is $h$. Hence, 
	$$
	A_{\Gamma_0, \mathrm{hBS}}(y) = h(g^{-1}(y))=h(g^{-1}(g(x))=h(x)=A_{\Gamma_0, \mathrm{hBS}}(x).
	$$
	This shows that $x$ and $y$ are grand orbit equivalent under $A_{\Gamma_0, \mathrm{hBS}}$.

	Next, if $y=g^{-1}.x$, then $x=g.y$ and 
	exchanging the roles of $x, y$ in the
	previous paragraph shows that  $x, y$ are grand orbit equivalent under $A_{\Gamma_0, \mathrm{hBS}}$. 
	Finally, by the symmetry of the setup, the same argument applies to $h, h^{-1}$.
\end{proof}

As a circle covering, the degree of $A_{\Gamma_0, \mathrm{hBS}}$ is equal to $4$. This can be easily seen from the actions of the pieces of $A_{\Gamma_0, \mathrm{hBS}}$ (along with their range) listed in Section~\ref{hdm_special_subsec}. Thus, we have now exhibited two different examples of \pwfm maps that are orbit equivalent to a thrice punctured sphere Fuchsian group; namely, the Bowen-Series map (of degree $3$) and the \hdm $A_{\Gamma_0, \mathrm{hBS}}$ defined above (of degree $4$). Moreover, the polygonal degree of $\widehat{A}_{\Gamma_0, \mathrm{hBS}}$ is $2$, while the Bowen-Series map induces a self-homeomorphism on the boundary of its fundamental domain.

\subsection{Folding and higher degree maps for general punctured spheres}\label{sec-genfold}

We follow the scheme of Section \ref{sec-cfm} above and generalize it to the case of $S_{0,k}$-a sphere with $k$ punctures, $k>3$. We shall use Figure~\ref{cfmgen} below as an illustration for the general case. Fix a (closed) fundamental domain of $\Gamma_0=G_{k-1}$ (see Subsection~\ref{b_s_punc_sphere_subsec} for the definition of $G_{k-1}$), given by an ideal $(2k-2)-$gon $W$ (the figure illustrates the $k=4$ case). For definiteness, let us assume that the ideal vertices of $W$ are the $(2k-2)$-th roots of unity. To make the book-keeping a little easier, we modify the notation as follows.
\begin{enumerate}
	\item The vertices of $W$ on the bottom semi-circle are numbered $1=1_-$, $2_-$ $\cdots$, $k_-=k$ in counter-clockwise order.
	\item  The vertices of $W$ on the top semi-circle are numbered $1, 2, \cdots, k$ in clockwise order.
	\item Between vertices $i, i+1$ (and including $i, i+1$) on the top semi-circle, there are 
	$2k-2$ vertices given by the vertices of $g_i.W$ (noting that $g_i.W \cap W$ equals the bi-infinite geodesic $\overline{i (i+1)}$). We label the 
	$2k-4$ vertices strictly between $i, i+1$ as $\{i,2\}, \{i,3\}, \cdots,
	\{i,2k-3\}$ in clockwise order. 
\end{enumerate}

The generators of $\Gamma_0$ are given by $g_1, \cdots, g_{k-1} $, where $g_i$ takes the edge $\overline{i_- (i+1)_-}$ to the  bi-infinite geodesic $\overline{i (i+1)}$. 

\subsubsection{A \cfm for $S_{0,k}$}\label{cfm_general_subsec}
Define  $R$ as
$$
R=\Int{\left(W \cup \bigcup_{i=1, \cdots, k-1} g_i.W\right)}, 
$$ 
so that $\overline{i (i+1)}$
are diagonals of $R$.

\begin{figure}[ht!]
\captionsetup{width=0.96\linewidth}
	\begin{tikzpicture}
		\node[anchor=south west,inner sep=0] at (0,0) {\includegraphics[width=0.5\linewidth]{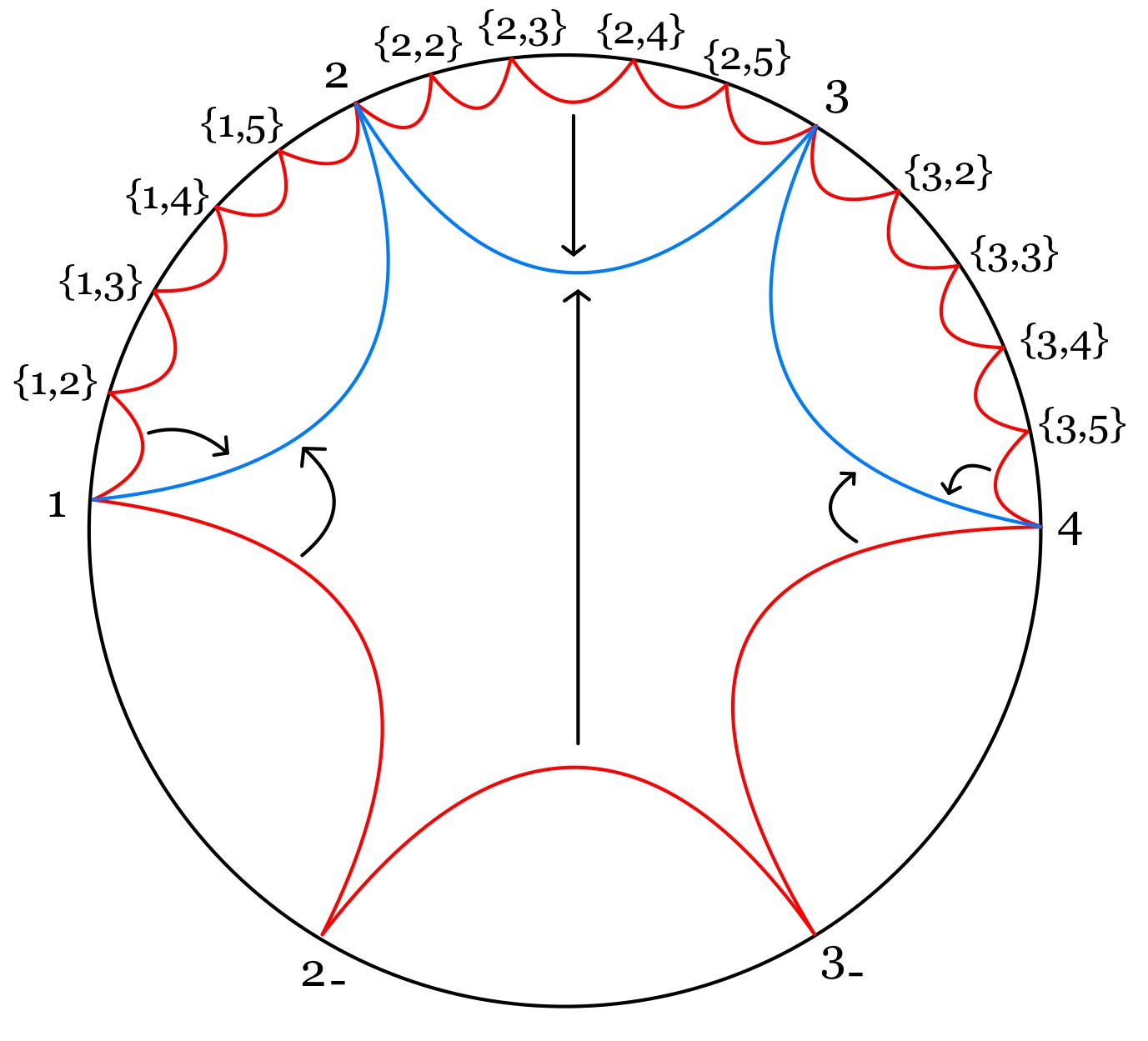}};
		\node[anchor=south west,inner sep=0] at (7,0) {\includegraphics[width=0.44\linewidth]{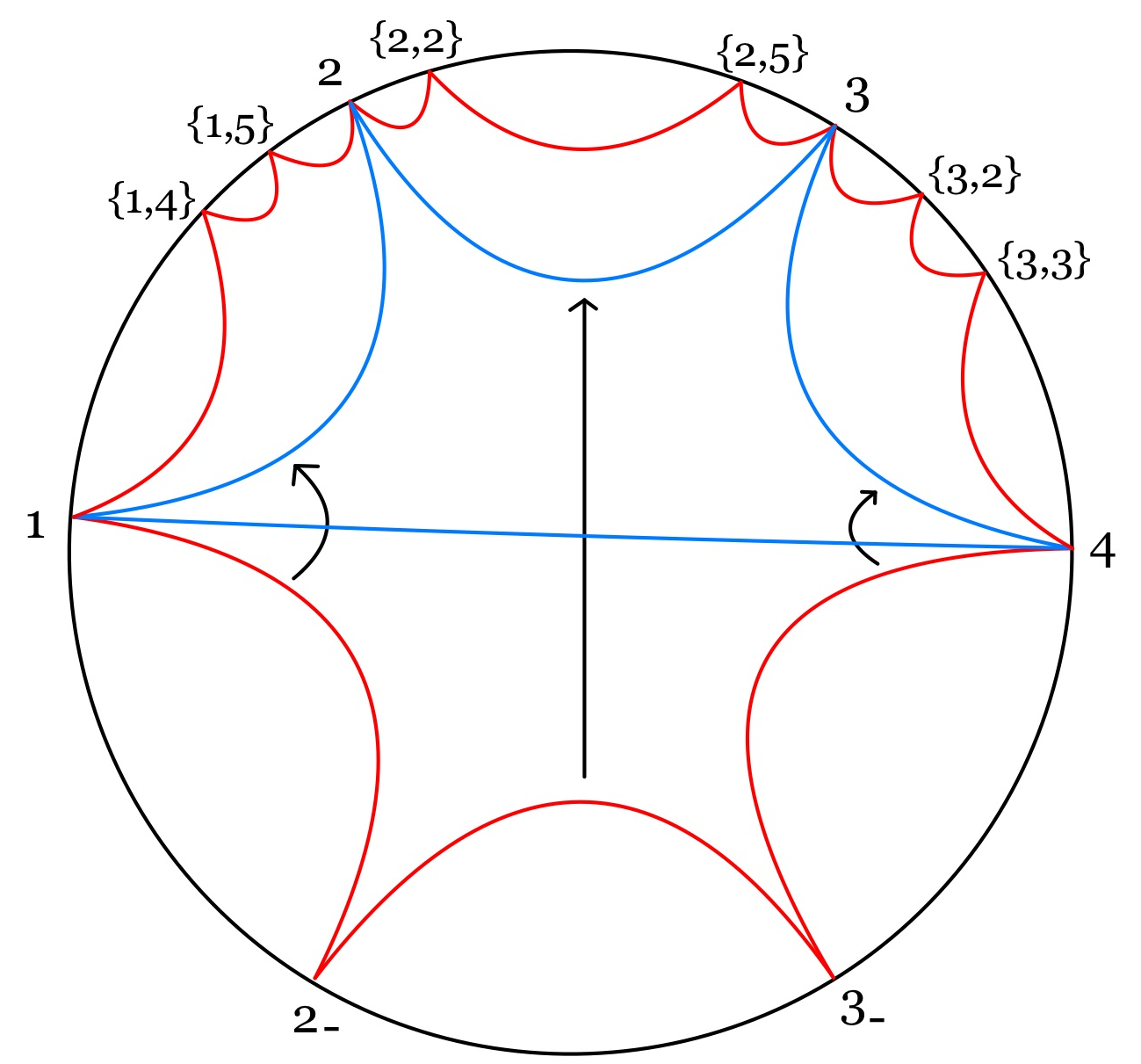}};
		
		\node at (4.4,2.9) {\begin{small}$g_3$\end{small}};
		\node at (3.5,3) {\begin{small}$g_2$\end{small}};
		\node at (2.1,2.9) {\begin{small}$g_1$\end{small}};
		\node at (1.25,3.6) {\begin{small}$g_1^{-1}$\end{small}};
		\node at (3.5,4.8) {\begin{small}$g_2^{-1}$\end{small}};
		\node at (5.12,3.38) {\begin{small}$g_3^{-1}$\end{small}};
		\node at (10.94,2.75) {\begin{small}$g_3$\end{small}};
		\node at (10.1,3) {\begin{small}$g_2$\end{small}};
		\node at (8.84,2.78) {\begin{small}$g_1$\end{small}};
	\end{tikzpicture}
	\caption{Fundamental domains for a completely folding map and a higher degree map without folding: $4$ punctures}
	\label{cfmgen}
\end{figure}
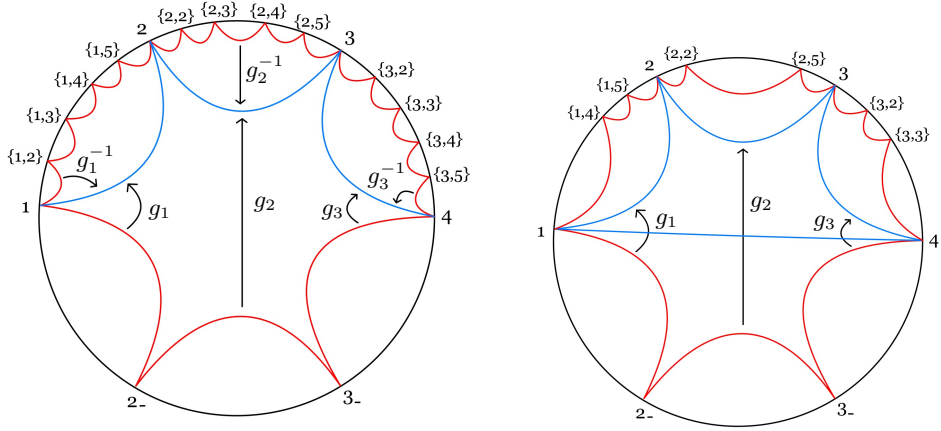

As in Section \ref{sec-cfm}, we define $A_{\Gamma_0, \mathrm{cfm}}$ in terms of its pieces as follows. Recall that $\arc{ij}$ stands for an arc with its endpoints at the break-points $i, j$ such that there are no other break-points of $A_{\Gamma_0, \mathrm{cfm}}$ in the arc.

\begin{itemize}
	\item On the arc $\arc{\ i_- (i+1)_-\ }$, define $A_{\Gamma_0, \mathrm{cfm}}$ to be $g_i$ for $i=1, \cdots, k-1$. Then $A_{\Gamma_0, \mathrm{cfm}}(\arc{\ i_- (i+1)_-\ })$ equals the complement of (the interior of) the arc $\arc{\ i (i+1)\ }$ in $\bS^1$.
	
	\item For every $i=1, \cdots, k-1$, and on each of the $k-1$ short arcs $\arc{\ \{i,j\} \{i,j+1\}\ }$ for $i \leq j \leq i+k-2$ between $i, i+1$, define $A_{\Gamma_0, \mathrm{cfm}}$ to be $g_i^{-1}$. Then $A_{\Gamma_0, \mathrm{cfm}}(\cup_{j=i}^{i+k-2}\arc{\ \{i,j\} \{i,j+1\} })$ equals the upper semi-circle between $1$ and $k$. (Here, for notational convenience, we identify $\{i,1\}$ with $i$ and $\{i, i+2k-2\}$ with $i+1$.)
	Also, for $i \leq j \leq i+k-2$, $A_{\Gamma_0, \mathrm{cfm}}$ maps the clockwise arc from $\{i,j\}$ to $\{i,j+1\}$ onto the clockwise arc from $j$ to $j+1$. We refer to the clockwise arcs from $\{i,j\}$ to $\{i,j+1\}$ (for $i \leq j \leq i+k-2$) as \emph{short folding arcs under $A_{\Gamma_0, \mathrm{cfm}}$}.
	
	\item For $i\in\{2,\cdots, k-1\}$ and $1\leq j\leq i-1$, set $j=i-s$, so that
	$1\leq s \leq i-1$. We  define $A_{\Gamma_0, \mathrm{cfm}}$ to be $g_s \circ g_i^{-1}$ on $\arc{\ \{i,j\} \{i,j+1\}\ }$. Thus, for $j\leq i-1$, $A(\arc{\{i,j\} \{i,j+1\}})$ equals the counter-clockwise (long) arc from $s$ to $s+1$.

	\item For $i\in\{1,\cdots, k-2\}$ and $i+k-1 \leq j\leq 2k-3$, let $j=i+k-1+t$, so that
	$0\leq t \leq k-2-i$.
	We  define $A_{\Gamma_0, \mathrm{cfm}}$ to be $g_{k-1-t} \circ g_i^{-1}$ on $\arc{\ \{i,j\} \{i,j+1\}\ }$.
	Thus, for $i+k-1 \leq j\leq 2k-3$, $A(\arc{ \{i,j\} \{i,j+1\}} )$ equals the counter-clockwise (long) arc from $k-1-t$ to $k-t$. 
	
	We refer to the clockwise arcs from $\{i,j\}$ to $\{i,j+1\}$ (for $j\leq i-1$ or $ 
	i+k-1 \leq j$) as \emph{long folding arcs under $A$}. 
	\item Note that $A_{\Gamma_0, \mathrm{cfm}}(i)=i$ for all $i=1,\cdots, k$.
\end{itemize}

It is easy to see from the above definition that $A_{\Gamma_0, \mathrm{cfm}}:\bS^1\to\bS^1$ is a \cfm. As any (marked) group $\Gamma\in\mathrm{Teich}(\Gamma_0)$ is conjugate to $\Gamma_0$ via a quasiconformal homeomorphism of $\widehat{\C}$ that preserves $\bS^1, \disk$ and respects the markings, we define the associated \cfm $A_{\Gamma, \mathrm{cfm}}$ to be the conjugate of $A_{\Gamma_0, \mathrm{cfm}}$ under such a quasiconformal homeomorphism.

\begin{rmk}
It is not hard to cook up other examples of completely folding maps. However, we do not know of any other \cfm that is orbit equivalent to the Fuchsian group generated by its pieces.
\end{rmk}

\subsubsection{A \hdm for $S_{0,k}$}\label{hdm_general_subsec} 

Again, as in Section \ref{sec-cfm}, define $A_{\Gamma_0, \mathrm{hBS}}$ to be the minimal \pwfm\ map coinciding with $A_{\Gamma_0, \mathrm{cfm}}$ on $\bS^1$. Denote the canonical extension of $A_{\Gamma_0, \mathrm{hBS}}$ by $\widehat{A}_{\Gamma_0, \mathrm{hBS}}$, its canonical domain of definition in $\overline{\disk}$ by $\mathcal{D}_{\Gamma_0, \mathrm{hBS}}$, and the fundamental domain of $\widehat{A}_{\Gamma_0, \mathrm{hBS}}$ by $R_{\Gamma_0, \mathrm{hBS}}$. Further, let $D$ be the open ideal polygon bounded by the bi-infinite geodesics $\overline{12}, \overline{23},\cdots, \overline{(k-1) k}, \overline{k 1}$. Evidently, all the edges of $D$ are (non-intersecting) diagonals of $R_{\Gamma_0, \mathrm{hBS}}$, each ideal vertex of $D$ is fixed by $A_{\Gamma_0, \mathrm{hBS}}$, each edge of $R_{\Gamma_0, \mathrm{hBS}}$ is mapped by $\widehat{A}_{\Gamma_0, \mathrm{hBS}}$ to an edge of $D$, and $\widehat{A}_{\Gamma_0, \mathrm{hBS}}$ has no diagonal folds. Therefore, $\widehat{A}_{\Gamma_0, \mathrm{hBS}}$ is a \hdm having $D$ as its inner domain. 

\begin{defn}[\bf{Higher Bowen-Series map}]\label{higher_bs_def}
	We call the \pwfm map $A_{\Gamma_0, \mathrm{hBS}}$ the \emph{higher Bowen-Series map} of $\Gamma_0$ (associated with the fundamental domain $W$). For any (marked) group $\Gamma\in\mathrm{Teich}(\Gamma_0)$, we define the higher Bowen-Series map of (the marked group) $\Gamma$ to be the conjugate of $A_{\Gamma_0, \mathrm{hBS}}$ under a quasiconformal homeomorphism of $\widehat{\C}$ that conjugates $\Gamma_0$ to $\Gamma$ (and respects the marking), and denote it by $A_{\Gamma, \mathrm{hBS}}$.
\end{defn}

Clearly, the higher Bowen-Series map of each $\Gamma\in\mathrm{Teich}(\Gamma_0)$ is a \hdm. We refer the reader to \cite[Proposition~5.2]{mj-muk-1} for a characterization of higher Bowen-Series maps among all higher degree maps without folding.

\subsubsection{Connections between Bowen-Series and higher Bowen-Series maps}\label{conn_bs_hbs_subsubsec}

The next two propositions are about the relationship between Bowen-Series maps and higher Bowen-Series maps (for $\Gamma\in\mathrm{Teich}(\Gamma_0)$). In fact, Proposition~\ref{hbs_alternative_prop} will give an alternative, more direct construction of the higher Bowen-Series map of $\Gamma$ in terms of the Bowen-Series maps of $\Gamma$ associated with various overlapping fundamental domains.

\begin{prop}[\bf{Characterizing higher Bowen-Series maps as piecewise Bowen-Series maps}]\label{hbs_alternative_prop}\cite[Proposition 4.5]{mj-muk-1}
	Let $W$ be a (closed) fundamental domain for a Fuchsian group $\Gamma\in\mathrm{Teich}(\Gamma_0)$ (uniformizing a $k$-times punctured sphere) which is an ideal $(2k-2)$-gon. We label the ideal vertices of $W$ as $1=1_-, 2_-,\cdots, (k-1)_-, k_-=k, k-1, \cdots, 2$ in counterclockwise order, and assume that the side-pairing transformations of $W$ (generating $\Gamma$) are given by $g_1, \cdots, g_{k-1} $, where $g_i$ takes the edge $\overline{i_- (i+1)_-}$ to the  edge  $\overline{i (i+1)}$.
	
	Further, let $D$ be the interior of the ideal polygon bounded by the bi-infinite geodesics $\overline{12}$, $\overline{23}$, $\cdots$, $\overline{(k-1) k}$, $\overline{k 1}$, and $P$ the interior of the ideal polygon bounded by the bi-infinite geodesics $\overline{1_- 2_-}$, $\overline{2_- 3_-}$, $\cdots$, $\overline{(k-1)_- k_-}$, $\overline{k_- 1_-}$. Then the following hold.
	
	\begin{enumerate}
		\item  $W=\overline{D}\cup\overline{P}$, and for each $j\in\{1,\cdots, k-1\}$, $\overline{D}\cup \overline{g_j(P)}$ is a (closed) fundamental domain for $\Gamma$.

		\item On the clockwise arc from $j$ to $j+1$, the higher Bowen-Series map $A_{\Gamma, \mathrm{hBS}}$ equals the Bowen-Series map of $\Gamma$ associated with the (closed) fundamental domain $\overline{D}\cup \overline{g_j(P)}$ ($j\in\{1,\cdots, k-1\}$), and on the counterclockwise arc from $1$ to $k$, $A_{\Gamma, \mathrm{hBS}}$ equals the Bowen-Series map of $\Gamma$ associated with the fundamental domain $W=\overline{D}\cup \overline{P}$.
	\end{enumerate}
	 Conversely, a map $A:\bS^1\to\bS^1$ defined as in condition (2) above is a higher Bowen-Series map orbit equivalent to $\Gamma$, and the fundamental domain of $A$ is given by $R=\Int{\left(W \cup \bigcup_{i=1, \cdots, k-1} g_i.W\right)}$.
\end{prop}

\begin{rmk}\label{hbs_name_rem}
	The preceding description of $A_{\Gamma, \mathrm{hBS}}$ shows that $A_{\Gamma, \mathrm{hBS}}$ is made up of Bowen-Series maps corresponding to various (overlapping) fundamental domains of $\Gamma$. This justifies the terminology `higher Bowen-Series maps'.
\end{rmk}

Higher Bowen-Series maps also arise as second iterates of suitable Bowen-Series maps.

\begin{prop}[\bf{Higher Bowen-Series as second iterate of Bowen-Series}]\cite[Corollary~5.6]{mj-muk-1}\label{second_iterate_hbs_cor} Let $d\geq 2$.
	\noindent\begin{enumerate}
		\item For $\Gamma\in\mathrm{Teich}(G_d)=\mathrm{Teich}(S_{0,d+1})$ (respectively, $\Gamma\in\mathrm{Teich}(G_{d,2})$), we have $A_{\Gamma,\mathrm{BS}}^{ 2}=A_{\Gamma',\mathrm{hBS}}$, where $\Gamma'$ is an index-two subgroup of $\Gamma$ with $\disk/\Gamma'\cong S_{0,2d}$. 
		
		\item For $\Gamma\in\mathrm{Teich}(G_{d,1})$, we have $A_{\Gamma,\mathrm{BS}}^{ 2}=A_{\Gamma',\mathrm{hBS}}$, where $\Gamma'$ is an index-two subgroup of $\Gamma$ with $\disk/\Gamma'\cong S_{0,2d-1}$. 
	\end{enumerate}
	
	In all cases, the second iterate of the Bowen-Series map of $\Gamma$ is orbit equivalent to an index-two subgroup of $\Gamma$.
	
The degree of the higher Bowen-Series map as a self-covering of $\bS^1$ is	$(\chi-1)^2$,
where $\chi = 2-k$ is the Euler characteristic of $S_{0,k}$.
\end{prop}

The last statement may be found in \cite[Section 4.3.2]{mj-muk-1}.
We do not know if higher iterates of Bowen-Series maps produce further examples of mateable maps (see Question \ref{question_oe_test} below). 
With careful combinatorial book-keeping, the arguments of the proof of Proposition~\ref{prop-cfdhd-orbeq} can be adapted for the general case.

\begin{prop}[\bf{Orbit equivalence}]\cite[Proposition~4.7]{mj-muk-1}\label{prop-cfdhd-orbeq-gen}
	Let $\Gamma\in\mathrm{Teich}(\Gamma_0)$, and $A_{\Gamma, \mathrm{cfm}},\ A_{\Gamma, \mathrm{hBS}}$ be as above. Then $A_{\Gamma, \mathrm{cfm}},\ A_{\Gamma, \mathrm{hBS}}$ are orbit equivalent to $\Gamma$.
\end{prop}

\subsection{Consequences}\label{sec-hbs-conseq} We now discuss some consequences. \\

\subsubsection{Interpolating between completely folding maps and higher degree  maps without folding}
The completely folding map $A_{\Gamma, \mathrm{cfm}}$ and the higher Bowen-Series map $A_{\Gamma, \mathrm{hBS}}$ described in Sections \ref{sec-cfm} and \ref{sec-genfold} agree on $\bS^1$. We denote the interior of the polygon in Section \ref{sec-genfold}
with vertices $1, \cdots, k$ by $D$. Note that $\overline{D}$ is `half' the (closed) fundamental domain $W$ in the sense that doubling $\overline{D}$ along the bi-infinite geodesic $\overline{1k}$ gives $W$.  Choose $1=i_1 < i_2 < \cdots < i_{l+1} = k$ to be a selection of vertices in clockwise 
cyclic order along the upper semi-circle.
Let $\cup_{1\leq j \leq l} (i_j, i_{j+1}) = \LL$ denote a finite
union of edges and diagonals of $W$ contained in $\overline{D}$. Let $W_0$ denote
the part of $W$  contained above $\LL$ and let $W_\LL=W_0 \cup \LL$.
Set 
$$
R_\LL = \Int{\left(W_\LL \cup \bigcup_{i=1, \cdots, k} g_i.W_\LL\right)}.
$$ 
Then $R_\LL$ is the fundamental domain of the \pwfm map $A_\LL$ whose canonical extension $\hhat{A_\LL}$ has domain $\mathcal{D}_\LL = \disk \setminus R_\LL$.

Note that, for all $\LL$, the map $A_\LL$ equals $A_{\Gamma, \mathrm{cfm}}$ on $\bS^1$. The map $A_{\Gamma, \mathrm{hBS}}$ 
is the unique minimal representative and corresponds to the case
$1=i_1 < i_2 = k$. The map $A_{\Gamma, \mathrm{cfm}}$ lies at the other end of the spectrum,
with $l+1 = k$, and $i_j=j$ for $j=1, \cdots, k$. The maps $A_\LL$ are 
non-minimal representatives whenever $l>1$.\\

\subsubsection{Mateability of completely folding maps and higher Bowen-Series maps} 
We now record the fact that higher Bowen-Series maps satisfy the conditions of Definition~\ref{def-mateable}, and hence can be conformally mated with hyperbolic complex polynomials (of appropriate degree) with Jordan curve Julia sets.

\begin{theorem}[\bf{Fuchsian higher Bowen-Series maps are mateable}]\cite[Theorem~4.8]{mj-muk-1}\label{thm-cfdmateable}
	Let $\Gamma\in\mathrm{Teich}(\Gamma_0)$, and $P\in\mathcal{H}_{(k-1)^2}$ (where $\mathcal{H}_d$ stands for the principal hyperbolic component in the space of degree $d$ polynomials). Then, $\widehat{A}_{\Gamma, \mathrm{hBS}}:\mathcal{D}_{\Gamma, \mathrm{hBS}}\to\overline{\disk}$ (respectively, $\widehat{A}_{\Gamma, \mathrm{cfm}}:\mathcal{D}_{\Gamma, \mathrm{cfm}}\to\overline{\disk}$) and $P:\mathcal{K}(P)\to\mathcal{K}(P)$ are conformally mateable. 
\end{theorem}

In light of Proposition~\ref{conformal_mating_general_prop}, we make the following definition.
\begin{defn}[\bf{Moduli space of matings}]\label{defn-moduli}
	The \emph{moduli space of matings} between a topological surface $\Sigma$ and complex polynomials in principal hyperbolic components consists of triples $(\Gamma, A_\Gamma, P)$, where 
	\begin{enumerate}
		\item $\Gamma$ is a Fuchsian group  uniformizing $\Sigma$,
		\item $A_\Gamma$ is a minimal mateable map orbit equivalent to $\Gamma$ on $\bS^1$, and
		\item $P$ is a polynomial in a principal hyperbolic component with $\mathrm{deg}(P)=\mathrm{deg}(A_\Gamma:\bS^1\to\bS^1)$.
	\end{enumerate}
\end{defn}
\noindent An immediate implication of Theorems~\ref{moduli_interior_mating_thm} and~\ref{thm-cfdmateable} is that the moduli space of matings between the topological surface $S_{0,k}$ ($k\geq 3$) and complex polynomials in principal hyperbolic components is disconnected. Specifically, it contains at least two components corresponding to
\begin{itemize}
\item Bowen-Series maps associated to groups in $\mathrm{Teich}(S_{0,k})$ and polynomials in $\mathcal{H}_{2k-3}$, and
\item higher Bowen-Series maps associated to groups in $\mathrm{Teich}(S_{0,k})$ and polynomials in $\mathcal{H}_{(k-1)^2}$.
\end{itemize}
We refer the readers to \cite[\S 6.4]{mj-muk-1} for further details.

Yet another application of orbit equivalence between higher Bowen-Series maps and Fuchsian punctured sphere groups is the failure of orbit equivalence rigidity for Fuchsian groups (see \cite{fisherwhyte_gafa} for general background on orbit equivalence rigidity and positive results, and \cite[\S 8]{mj-muk-1} for a precise statement of its failure in the Fuchsian case).

\subsection{Two non-examples}\label{ehbs_subsec}
In this subsection, we will consider two modifications of higher Bowen-Series maps and show that the resulting \pwfm maps are not orbit equivalent to the groups generated by their pieces.
\vspace{2mm}

\subsubsection{A non-example without folding}\label{noe_1_subsubsec} The following description of the higher Bowen-Series map $A_{\Gamma_0, \mathrm{hBS}}$ on $\bS^1$ is straightforward to check from its construction (see Subsection~\ref{sec-genfold}):

\begin{equation*}
	A_{\Gamma_0, \mathrm{hBS}}\ = \ \left\{\begin{array}{ll}
		A_{\Gamma_0, \mathrm{BS}}, & \mbox{on}\  \displaystyle \left(\bigcup_{i=1}^{k-1} \arc{\ i_- (i+1)_-\ }\right) \cup \left(\bigcup_{i=1}^{k-1}\bigcup_{j=i}^{i+k-2} \arc{\ \{i,j\} \{i,j+1\}\ } \right), \\
		A_{\Gamma_0, \mathrm{BS}}^{ 2}, & \mbox{otherwise},
	\end{array}\right. 
\end{equation*}
where $A_{\Gamma_0, \mathrm{BS}}$ denotes the Bowen-Series map of $\Gamma_0$ associated with the fundamental domain $W$.

In fact, the agreement of $A_{\Gamma_0, \mathrm{hBS}}$ and $A_{\Gamma_0, \mathrm{BS}}$ on the arcs $\arc{\ \{i,j\} \{i,j+1\}\ }$ ($i\in\{1,\cdots, k-1\}, j\in\{i,\cdots, i+k-2\}$) played an important role in the proof of orbit equivalence of $\Gamma_0$ and $A_{\Gamma_0, \mathrm{hBS}}$ (see Proposition~\ref{prop-cfdhd-orbeq}). However, if one replaces $A_{\Gamma_0, \mathrm{BS}}$ by $A_{\Gamma_0, \mathrm{BS}}^{ 2}$ on these arcs as well, the resulting minimal \pwfm map 
\begin{equation*}
	B := \ \left\{\begin{array}{ll}
		A_{\Gamma_0, \mathrm{BS}} & \mbox{on}\  \bS^1\cap\{z: \mathrm{Im}(z)\leq 0\}, \\
		A_{\Gamma_0, \mathrm{BS}}^{ 2} & \mbox{on}\  \bS^1\cap\{z: \mathrm{Im}(z)\geq 0\},
	\end{array}\right. 
\end{equation*}
is \emph{not} orbit equivalent to $\Gamma_0$. 

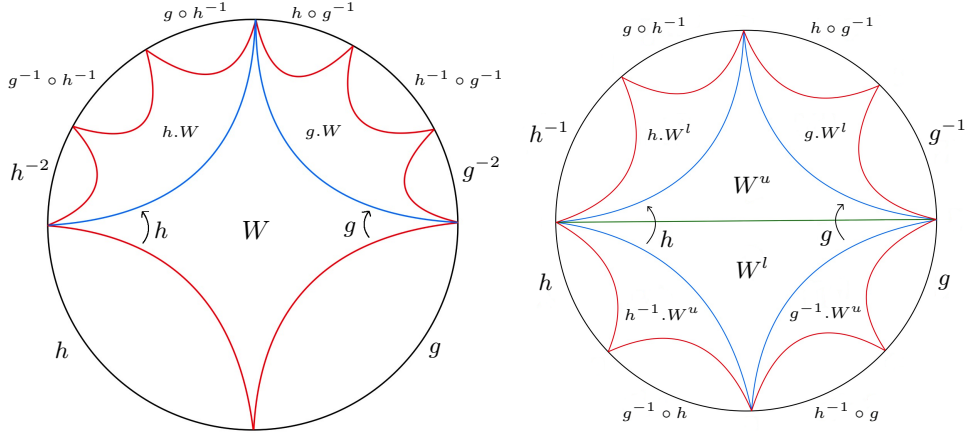
\begin{figure}[ht!]
\captionsetup{width=0.96\linewidth}
	\begin{tikzpicture}
		\node[anchor=south west,inner sep=0] at (0,0) {\includegraphics[width=0.45\linewidth]{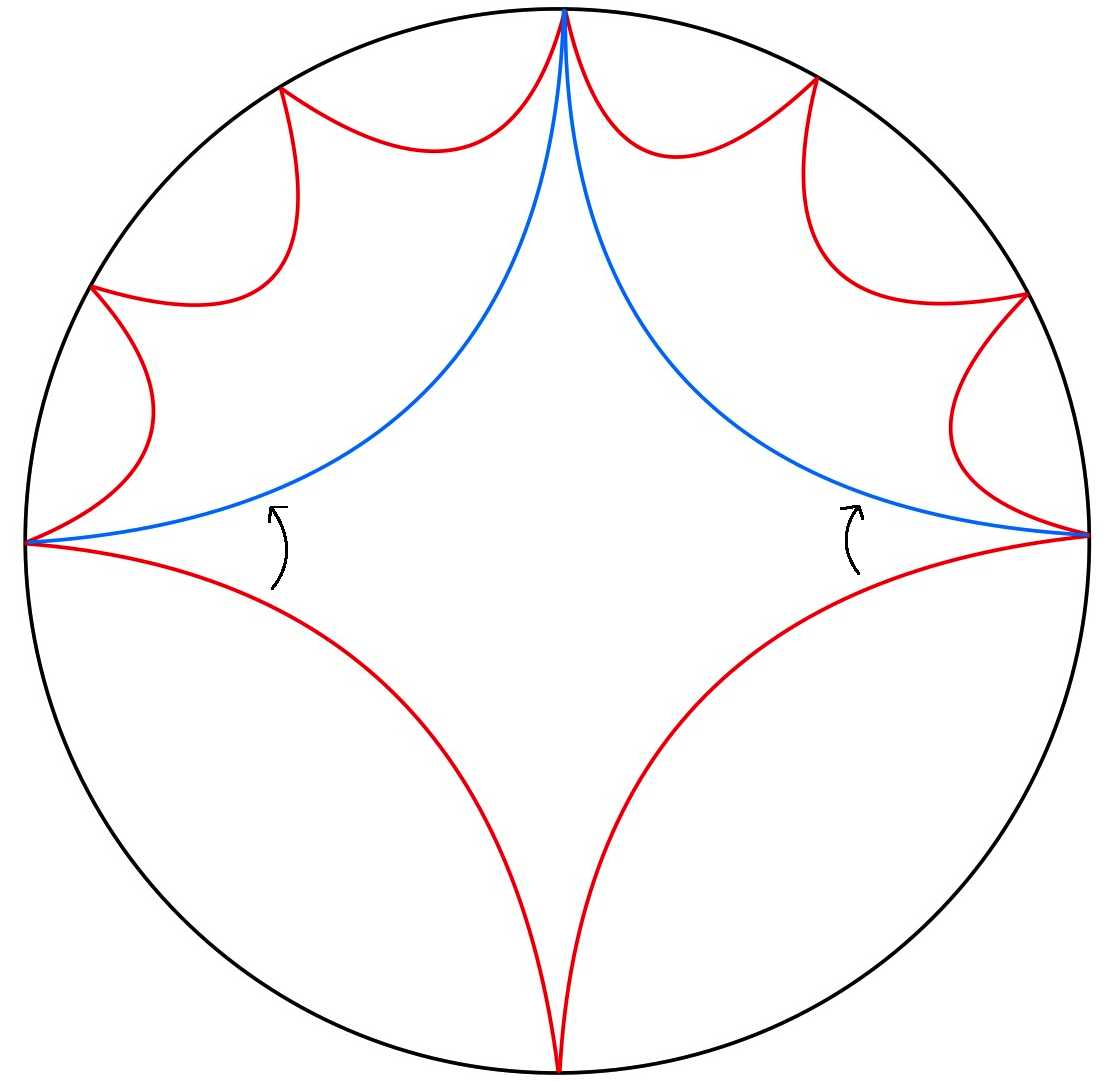}};
			\node[anchor=south west,inner sep=0] at (6.2,0) {\includegraphics[width=0.5\linewidth]{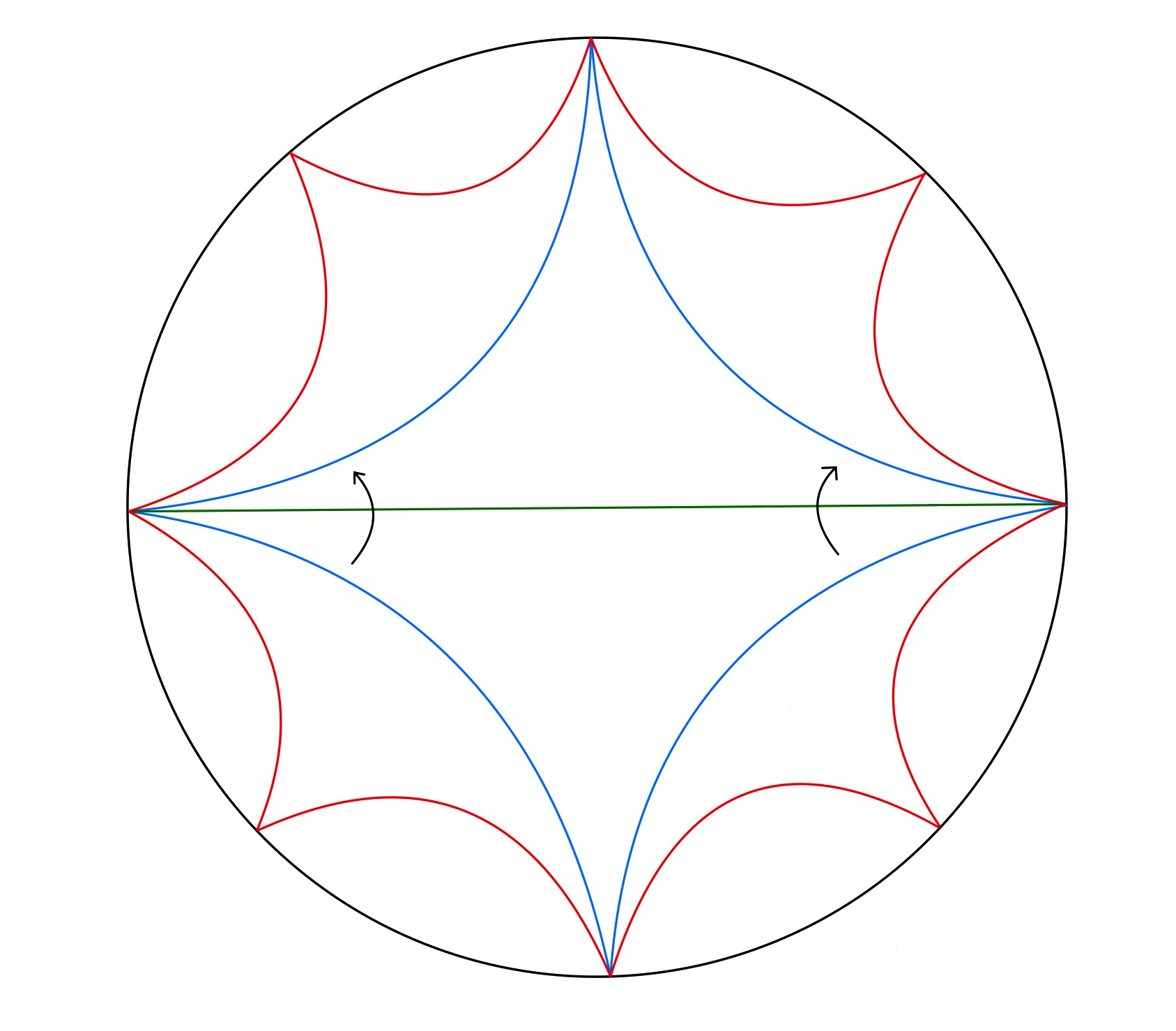}};

		\node at (9.5,3.3) {\begin{small}$W^u$\end{small}};
		\node at (9.5,2.2) {\begin{small}$W^l$\end{small}};
		\node at (10.48,2.58) {\begin{small}$g$\end{small}};	
		\node at (10.48,1.6) {\begin{tiny}$g^{-1}.W^u$\end{tiny}};	
		\node at (8.3,1.6) {\begin{tiny}$h^{-1}.W^u$\end{tiny}};	
		\node at (10.48,4) {\begin{tiny}$g.W^l$\end{tiny}};	
		\node at (8.4,4) {\begin{tiny}$h.W^l$\end{tiny}};	
		\node at (8.36,2.58) {\begin{small}$h$\end{small}};	
		
		\node at (6.75,2) {\begin{small}$h$\end{small}};	
		\node at (6.8,4) {\begin{small}$h^{-1}$\end{small}};	
		\node at (12.05,2) {\begin{small}$g$\end{small}};	
		\node at (12.1,4) {\begin{small}$g^{-1}$\end{small}};
		\node at (10.7,5.4) {\begin{tiny}$h\circ g^{-1}$\end{tiny}};
		\node at (8.2,5.4) {\begin{tiny}$g\circ h^{-1}$\end{tiny}};
		\node at (10.75,0.28) {\begin{tiny}$h^{-1}\circ g$\end{tiny}};
		\node at (8.2,0.28) {\begin{tiny}$g^{-1}\circ h$\end{tiny}};	
		
		\node at (4.15,2.7) {\begin{small}$g$\end{small}};	
		\node at (1.64,2.7) {\begin{small}$h$\end{small}};	
		\node at (0.32,1.1) {\begin{small}$h$\end{small}};	
		\node at (-0.1,3.5) {\begin{small}$h^{-2}$\end{small}};	
		\node at (5.28,1.1) {\begin{small}$g$\end{small}};	
		\node at (5.9,3.5) {\begin{small}$g^{-2}$\end{small}};
		\node at (3.8,5.6) {\begin{tiny}$h\circ g^{-1}$\end{tiny}};
		\node at (2.1,5.6) {\begin{tiny}$g\circ h^{-1}$\end{tiny}};
		\node at (5.6,4.7) {\begin{tiny}$h^{-1}\circ g^{-1}$\end{tiny}};
		\node at (0.2,4.7) {\begin{tiny}$g^{-1}\circ h^{-1}$\end{tiny}};
		\node at (2.9,2.7) {$W$};
		\node at (3.8,4) {\begin{tiny}$g.W$\end{tiny}};	
		\node at (1.9,4) {\begin{tiny}$h.W$\end{tiny}};	
		
	\end{tikzpicture}
	\caption{The two \pwfm maps $B$ (left) and $C$ (right) are depicted. They are not orbit equivalent to the Fuchsian groups generated by their pieces.}
	\label{noe_fig}
\end{figure}

\begin{prop}\cite[Proposition~4.9]{mj-muk-1}\label{not_oe_prop}
	The map $B:\bS^1\to\bS^1$ is not orbit equivalent to the Fuchsian group $\Gamma_0$ generated by its pieces.
\end{prop}

\begin{rmk}
The map $B$ has no diagonal fold, but it is not a \hdm. This is because some edges of the boundary of the fundamental domain of $B$ are not mapped to diagonals, and hence Condition~(3) of Definition~\ref{def-nofold} is violated (see Figure~\ref{noe_fig} (left)).
\end{rmk}

\subsubsection{A non-example with folding}\label{noe_2_subsubsec} Yet another example of a \pwfm map that is not orbit equivalent to the Fuchsian group generated by its pieces is given by the following symmetric version of higher Bowen-Series maps. For simplicity, we illustrate the thrice punctured sphere case. 

Consider the Fuchsian group $\Gamma_0$ of Section~\ref{sec-cfm}, the (closed) fundamental domain $W$ with vertices at the fourth roots of unity, and the generators $g, h$ that pair the sides of $W$ (as in Figure~\ref{cfd}). We set 
$$
W^u:=W\cap\{\mathrm{Im}(z)\geq 0\},\quad W^l:=W\cap\{\mathrm{Im}(z)\leq 0\}.
$$

The \pwfm map $C:\bS^1\to\bS^1$ depicted in Figure~\ref{noe_fig} (right) has degree $5$. It has diagonal folds in the sense of Definition~\ref{def-nofold}, but it is not a \cfm since Condition~(4) of Definition~\ref{def-cfm} fails. Moreover,

\begin{prop}\label{not_oe_prop_2}
	The \pwfm map $C:\bS^1\to\bS^1$ of Figure~\ref{noe_fig} (right) is not orbit equivalent to the Fuchsian group $\Gamma_0$ generated by its pieces.
\end{prop}
\begin{proof}
	Observe that $g(-i)=i$, and thus the points $\pm i$ lie in the same $\Gamma_0-$orbit. But both these points are fixed by $C$, and hence they cannot lie in the same grand orbit of $C$.
\end{proof}

The proofs of orbit equivalence for Bowen-Series and higher Bowen-Series maps only involve looking at the first iterates of the maps. On the other hand, the orbit equivalence property is ruled out for the above two \pwfm maps simply by furnishing suitable fixed points of the maps. In general, we ask the following question.

\begin{qn}\label{question_oe_test}
	Is there a general recipe to test whether a \pwfm map is orbit equivalent to the Fuchsian group generated by its pieces?
\end{qn}

\section{Invariant laminations and Bers boundary groups}

The existence of mateable maps orbit equivalent to Fuchsian punctured sphere groups leads one to the hunt for groups on boundaries of Teichm{\"u}ller spaces (of punctured spheres) that can be conformally mated with complex polynomials. Since Fuchsian realizations of Teichm{\"u}ller spaces are non-compact, the aforementioned pursuit ought to be carried out on boundaries of Bers slices of Fuchsian punctured sphere groups (note that such a realization of the Teichm{\"u}ller space is precompact in a suitable topology). As in the Fuchsian case, the first challenge one encounters in this program is to come up with the correct notion of `mateable maps' for Bers boundary groups.

Let us fix a Fuchsian punctured sphere group $\Gamma_0$ equipped with a (higher) Bowen-Series map $A_{\Gamma_0}$. We denote the Bers slice of $\Gamma_0$ by $\mathcal{B}(\Gamma_0)$. The map $A_{\Gamma_0}$ defines, for each (marked) group $\Gamma'$ in $\mathcal{B}(\Gamma_0)$, a piecewise M{\"o}bius Markov covering map $A_{\Gamma'}$ (via quasiconformal conjugation) of the limit set $\Lambda(\Gamma')$ such that $A_{\Gamma'}$ is orbit equivalent to $\Gamma'$ on $\Lambda(\Gamma')$. The map $A_{\Gamma'}$ is the (higher) Bowen-Series map associated with the map $\Gamma'$. Now let $\Gamma\in\partial\mathcal{B}(\Gamma_0)$.
Guided by the Fuchsian situation, we call a continuous self-map $A_\Gamma:\Lambda(\Gamma)\to\Lambda(\Gamma)$ the \emph{(higher) Bowen-Series map of $\Gamma\in\partial\mathcal{B}(\Gamma_0)$} if
\begin{enumerate}
	\item $A_\Gamma$ is orbit equivalent to $\Gamma$, and
	
	\item $A_\Gamma$ is the uniform limit of the (higher) Bowen-Series maps $A_{\Gamma'}$, as $\Gamma'\in\mathcal{B}(\Gamma_0)$ converges to $\Gamma$ in the strong topology.
\end{enumerate} 

The Bers density conjecture, now a theorem due to Brock-Canary-Minsky \cite{minsky-elc1,minsky-elc2} (see also \cite{brock-bromberg-density}) states that the Bers slice is dense in the space of all Kleinian surface groups with
one end carrying a fixed conformal structure. Thus, the closure of the Bers slice gives all such Kleinian surface groups. For any such group $\G$ on the Bers boundary, there  is an end-invariant
called the \emph{ending lamination}--a geodesic lamination supporting a transverse measure. Further, there exists a topological semiconjugacy \cite{mahan-split,mahan-kl}, 
called a \emph{Cannon-Thurston map} from the circle onto the limit set $\Lambda(\Gamma)$ of $\Gamma$. It was shown in \cite{mahan-elct,mahan-kl} that the Cannon-Thurston map identifies precisely
the end-points of the ending lamination.

It follows from  \cite{minsky-elc1,minsky-elc2} (see also \cite{brock-bromberg-density})  that any Bers boundary group is a strong limit of groups in the Bers slice.
Let $\Gamma_n=\phi_n\circ \Gamma_0\circ\phi_n^{-1}$ be a sequence of groups in $\mathcal{B}(\Gamma_0)$ (where $\phi_n$ is a quasiconformal homeomorphism inducing the representation $\Gamma_0\to\Gamma_n$) converging strongly to
$\Gamma \in \partial \BB(\Gamma_0)$.
By \cite{mahan-series1,mahan-series2}, \cite[Section~4.2]{mahan-motion-limit}, Cannon-Thurston maps of $\Gamma_n$ converge uniformly to the 
Cannon-Thurston map of $\Gamma$. Since the (higher) Bowen-Series map of $\Gamma_n$ is equal to $\phi_n\circ A_{\Gamma_0}\circ \phi_n^{-1}$, by definition, the (higher) Bowen-Series map $A_{\Gamma}$, \emph{if it exists,} must be given by $\phi_\infty\circ A_{\Gamma_0}\circ \phi_\infty^{-1}$, where $\phi_\infty:\bS^1\to\Lambda(\Gamma)$ is the Cannon-Thurston map of $\Gamma$. Thus, the (higher) Bowen-Series map $A_\Gamma:\Lambda(\Gamma)\to\Lambda(\Gamma)$, if it exists, must be semi-conjugate to the base (higher) Bowen-Series map $A_{\Gamma_0}:\bS^1\to\bS^1$ via the Cannon-Thurston map of $\Gamma$ (see \cite[\S 7.1]{mj-muk-1} for details):

\[ \begin{tikzcd}
	\mathbb{S}^1 \arrow{r}{A_{\Gamma_0}} \arrow[swap]{d}{\mathrm{C. T.}} & \mathbb{S}^1 \arrow{d}{\mathrm{C. T.}} \\
	\Lambda(\Gamma) \arrow[swap]{r}{A_{\Gamma}}& \Lambda(\Gamma)
\end{tikzcd}
\]

Put differently, in the light of \cite{mahan-elct}, the existence of a (higher) Bowen-Series map $A_\Gamma$ requires the ending lamination $\mathcal{L}$ of $\Gamma$ (where $\bS^1/\mathcal{L}\cong\Lambda(\Gamma)$) to be invariant under the action of $A_{\Gamma_0}$ (\cite[Lemma~7.3]{mj-muk-1}). On the other hand, such laminations are necessarily invariant under the action of $\Gamma_0$. Simultaneous invariance of $\mathcal{L}$ under the group $\Gamma_0$ and the covering map $A_{\Gamma_0}$ can be thought of as a compatibility condition between Kleinian group dynamics and polynomial dynamics, which turns out to be very restrictive.

\begin{theorem}[\bf{Bers boundary (higher) Bowen-Series maps are sparse}]\cite[Propositions 7.6, 7.8]{mj-muk-1}\label{thm_boundary_groups_1} 
	Let $\Gamma_0$ be a punctured sphere Fuchsian group. Then, there are only finitely many quasiconformal conjugacy classes of groups $\Gamma\in\partial\mathcal{B}(\Gamma_0)$ for which the Cannon-Thurston map of $\Gamma$ semi-conjugates the (higher) Bowen-Series map of $\Gamma_0$ to a self-map $A_\Gamma$ of $\Lambda(\Gamma)$ that is orbit equivalent to $\Gamma$. These Kleinian groups arise out of pinching finitely many disjoint, simple, closed curves (on the surface $\disk/\Gamma_0$) out of an explicit finite list. In particular, all such groups $\Gamma$ are geometrically finite.
\end{theorem}

\begin{rmk}
Consider the Bowen-Series map $A_{G_d}$ associated with the Fuchsian group $G_d$ equipped with the preferred fundamental domain $R$ given by the ideal polygon with vertices at the $2d$-th roots of unity (see Subsection~\ref{b_s_punc_sphere_subsec}). In this case, the explicit finite list of Theorem~\ref{thm_boundary_groups_1} is 
$$
\mathbf{S}_d:=\{g_2,\cdots,g_{d-1}\}\cup\{g_i^{-1}\circ g_j: i, j\in\{1,\cdots, d\},\  i-j >1\}
$$
(see \cite[Proposition~7.6]{mj-muk-1}), and hence every $A_{G_d}-$invariant geodesic lamination on $\disk/G_d\cong S_{0,d+1}$ is a subset of $S_d$. For $d=3$, this gives exactly two invariant laminations: $\{\{g_2\}, \{g_3^{-1}\circ g_1\}\}$. The curve corresponding to $g_2$ (respectively, $g_3^{-1}\circ g_1$) on the four times punctured sphere depicted in Figure~\ref{fund_dom_punctured_sphere_fig} is the `vertical' (respectively, `horizontal') curve which separates the punctures $[p_3], [p_4]$ from $[p_1], [p_2]$ (respectively, $[p_1], [p_4]$ from $[p_2], [p_3]$.
\end{rmk}

The (higher) Bowen-Series map of a Bers boundary group (when it exists) is piecewise M{\"o}bius and hence admits a canonical extension $\widehat{A}_\Gamma$ to a subset of the \emph{filled limit set} $K(\Gamma)$ of the group (i.e., the complement of the completely invariant component of its domain of discontinuity). 

Now let $P$ be a complex polynomial in the principal hyperbolic component $\mathcal{H}_k$, where $k=\deg\{A_\Gamma:\Lambda(\Gamma)\longrightarrow\Lambda(\Gamma)\}$. Then, the action of $P$ on its Julia set $\mathcal{J}(P)$ is topologically conjugate to $z^k\vert_{\bS^1}$. On the other hand, $A_\Gamma\vert_{\Lambda(\Gamma)}$ is a factor of $A_{\Gamma_0}\vert_{\bS^1}$, which is in turn topologically conjugate to $z^k\vert_{\bS^1}$. One can now glue the filled Julia set $\mathcal{K}(P)$ (which is a closed Jordan disk) outside the filled limit set $K(\Gamma)$ using a semi-conjugacy between $P\vert_{\mathcal{J}(P)}$ and $A_\Gamma\vert_{\Lambda(\Gamma)}$, and this produces a topological $2$-sphere. Moreover, the existence of this semi-conjugacy implies that the action of $P$ on $\mathcal{K}(P)$ and the action of $\widehat{A}_\Gamma$ on a subset of $K(\Gamma)$ paste together to yield a continuous map on the copy of $\mathbb{S}^2$ just defined. This map is called the \emph{topological mating} of $\widehat{A}_\Gamma$ and $P$. 
We say that the canonical extension $\widehat{A}_\Gamma$ of the (higher) Bowen-Series map of a Bers boundary group is \emph{conformally mateable} with a polynomial $P$ in the principal hyperbolic component $\mathcal{H}_k$ if the above topological $2$-sphere admits a complex structure that turns the topological mating into a holomorphic map (cf. \cite[\S 7.5]{mj-muk-1}).

A sophisticated surgery procedure involving David homeomorphisms yields the following conformal mateability theorem.

\begin{theorem}[\bf{Bers boundary (higher) Bowen-Series maps are mateable}]\cite[Theorem 7.19]{mj-muk-1}\label{thm_boundary_groups_2}
	Let $\Gamma\in\partial\mathcal{B}(\Gamma_0)$ be a group that admits a (higher) Bowen-Series map $A_\Gamma$. Then the canonical extension $\widehat{A}_\Gamma$ can be conformally mated with polynomials lying in the principal hyperbolic component $\mathcal{H}_k$, where $k=\deg\{A_\Gamma:\Lambda(\Gamma)\longrightarrow\Lambda(\Gamma)\}$.

\end{theorem}

For a group $\Gamma\in\partial\mathcal{B}(\Gamma_0)$ admitting a (higher) Bowen-Series map $A_\Gamma$, the corresponding geodesic lamination $\mathcal{L}$ is invariant under the base (higher) Bowen-Series map $A_{\Gamma_0}$. The associated equivalence relation $\mathcal{L}$ on $\bS^1$ satisfies the following properties.
\begin{enumerate}\upshape
\item $\mathcal{L}$ is closed in $\R/\Z\times\R/\Z$.
\item Each equivalence class $X$ of $\mathcal{L}$ is a finite subset of $\R/\Z$.
\item $\mathcal{L}-$equivalence classes are pairwise \emph{unlinked}; i.e., if $X$ and $Y$ are two distinct equivalence classes of $\mathcal{L}$, then there exist disjoint intervals $I_X, I_X\subset\R/\Z$ such that $X\subset I_X$ and $Y\subset I_Y$.
\item If $X$ is an $\mathcal{L}-$equivalence class, then $A_{\Gamma_0}(X)$ is also an $\mathcal{L}$-equivalence class.
\item If $X$ is an $\mathcal{L}-$equivalence class, then $X\mapsto A_{\Gamma_0}(X)$ is a cyclic order preserving bijection.
\end{enumerate}
On the other hand, the lamination associated with a complex polynomial $P$ with connected Julia set also enjoys analogues of the properties listed above (where the role of $A_{\Gamma_0}$ is played by the base polynomial $z^d$). Roughly speaking, the lamination associated with $P$ is a $z^d-$invariant closed equivalence relation on $\bS^1$ such that the quotient of $\bS^1$ by the equivalence relation yields a topological model of the Julia set of $P$ (cf. \cite{kiwi1}). Remarkably, the topological conjugacy between $A_{\Gamma_0}\vert_{\bS^1}$ and $z^d\vert_{\bS^1}$ (for some $d\geq 2$) provides us with a tool to pass from laminations in the group world to those in the polynomial world. This combinatorial link allows one to invoke standard realization results from polynomial dynamics and conclude that the limit set $\Lambda(\Gamma)$ is indeed homeomorphic to the Julia set of a complex polynomial in a `dynamically natural' way.

\begin{theorem}[\bf{Equivariant homeomorphism between limit and Julia set}]\cite[Theorem 7.16]{mj-muk-1}\label{thm_boundary_groups_3}
	Let $\Gamma\in\partial\mathcal{B}(\Gamma_0)$ be a group that admits a (higher) Bowen-Series map $A_\Gamma$. Then there exists a complex polynomial $P_\Gamma$ (of degree equal to that of $A_\Gamma:\Lambda(\Gamma)\longrightarrow\Lambda(\Gamma)$) such that the action of $A_\Gamma$ on the limit set $\Lambda(\Gamma)$ is topologically conjugate to the action of $P_\Gamma$ on its Julia set.
\end{theorem}

\section{Measures of maximal entropy and Patterson-Sullivan measures}\label{sec_meas_th}

In this section, we study the measure-theoretic dynamics of Bowen-Series and higher Bowen-Series maps associated with Fuchsian punctured sphere groups, thus linking this theme to another seminal piece of work by Sullivan--the Patterson-Sullivan measure. Specifically, we show that measures of maximal entropy of (higher) Bowen-Series maps acting on the circle are push-forwards of appropriate Patterson-Sullivan measures supported on Gromov boundaries of free groups.

Informally speaking, the Sullivan-Patterson measure on the Gromov boundary of a group is the weak limit of a sequence of atomic measures supported on the words of length $n$, appropriately weighted by the distances of the group elements from a fixed base point (for the free group with the standard generating set, the sequence reduces to Formula~\eqref{ps_formula}). Although we will not use the general theory of Patterson-Sullivan measures, we encourage the reader to consult \cite{patterson,sullivan_cm,coor} for the construction and basic properties of these measures in the context of Fuchsian groups, Kleinian groups, and hyperbolic groups, respectively. For background on symbolic dynamics and topological/measure-theoretic entropy, we refer the reader to \cite{Walters_book,brin-stuck}. 

\subsection{Maximal entropy measure for Bowen-Series maps}\label{sec_meas_th_bs}
For definiteness, let us fix the Fuchsian $(d+1)-$times punctured sphere group $\Gamma_0=G_d$ of Section~\ref{b_s_punc_sphere_subsec} and the  fundamental domain $R$ given by the ideal polygon with vertices at the $2d-$th roots of unity ($d\geq 2$). Further let $A\equiv A_{\Gamma_0, \textrm{BS}}:\bS^1\to\bS^1$ be the Bowen-Series map of $\Gamma_0$ associated with the fundamental domain $R$. 

The topological entropy of a dynamical system is a numerical topological conjugacy invariant that measures the complexity of the system. Roughly, it represents the exponential growth rate of the number of essentially different orbit segments of length $n$. Since $A$ is topologically conjugate to $z^{2d-1}$, the topological entropy of the $A-$action on $\bS^1$ is equal to $\ln(2d-1)$. We are interested in studying the measure of maximal entropy (MME for short) for $A$; i.e., the unique $A-$invariant measure on $\bS^1$ whose measure-theoretic entropy is equal to the topological entropy $\ln(2d-1)$ (see \cite{AKU1,AKU2} for computation of topological entropy of Bowen-Series maps associated with cocompact Fuchsian groups and results regarding their measures of maximal entropy).

\subsubsection{MME of $A$ in terms of topological dynamics}

By Proposition~\ref{b_s_poly_conjugate_prop_1}, there exists a homeomorphism 
$$\phi:\bS^1\to\bS^1
$$ 
that conjugates $p:z\mapsto z^{2d-1}$ to $A$ (this homeomorphism can be thought of as a generalization of the Minkowski question-mark function $\ciq$; see \cite[\S 4.4.2]{LLMM1} for the analogy in the anti-holomorphic context).

We denote the Haar (normalized Lebesgue) measure on $\bS^1$ by $m$. Note that $m$ is the unique measure of maximal entropy for the action of $p$ on $\bS^1$ (a straightforward computation shows that the measure-theoretic entropy of $p\vert_{\bS^1}$ with respect to $m$ is equal to the topological entropy $\ln(2d-1)$, and the uniqueness of this measure follows for instance from \cite[Theorem~9]{Lyubich_mme}). Since the homeomorphism $\phi$ is a conjugacy, we have the following.

\begin{prop}\label{mme_leb}
	$\nu=\phi_\ast m$, where $\nu$ is the unique measure of maximal entropy for the $A-$action on $\bS^1$. 
\end{prop}

\subsubsection{MME of $A$ in terms of symbolic dynamics}
The partition of $\bS^1$ determined by the $2d-$th roots of unity form a Markov partition for $A$. We denote this partition by $\{I_1,I_{-1},\cdots,I_d,I_{-d}\}$, where $I_j$ is the counter-clockwise arc of $\bS^1$ connecting $e^{2\pi i\frac{(j-1)}{2d}}$ and $e^{2\pi i\frac{j}{2d}}$,  and $I_{-j}$ is the complex conjugate of $I_j$, for $j\in\{1,\cdots, d\}$. The transition matrix for this Markov partition is
$$
M:=	\begin{bmatrix}
	1&0&1&1&\cdots&1&1\\
	0&1&1&1&\cdots&1&1\\
	1&1&1&0&\cdots&1&1\\
	1&1&0&1&\cdots&1&1\\
	\hdotsfor{7}\\
	1&1&1&1&\cdots&1&0\\
	1&1&1&1&\cdots&0&1\\
\end{bmatrix}.
$$
The above transition matrix gives rise to a one-sided subshift of finite type 
$$
\sigma: \Sigma_M^+\to\Sigma_M^+.
$$ 
Here $\Sigma_M^+$ is the collection of $M-$admissible infinite words in $\{\pm 1, \pm 2, \cdots, \pm d\}^{\mathbb{N}}$; i.e., 
$$
\Sigma_M^+:=\{(i_1,i_2,\cdots)\in\{\pm 1, \pm 2, \cdots, \pm d\}^{\mathbb{N}}: A(I_{i_j})\supset I_{i_{j+1}}\ \mathrm{for\ all}\ j\geq 1\},
$$
and $\sigma$ is the left-shift map. A \emph{cylinder set of rank $k\geq 1$} in $\Sigma_M^+$ is a set of the form 
$$
\left[r_1,\cdots,r_k\right]:=\{(i_1,i_2,\cdots)\in\Sigma_M^+: i_j=r_j,\ \mathrm{for}\ j\in\{1,\cdots, k\}\},
$$
where $(r_1,\cdots, r_k)\in\{\pm 1,\cdots, \pm d\}^k$. We metrize $\Sigma_M^+$ with the usual ultra-metric (in base $e$). 

Since $A$ is expansive, one obtains a continuous surjection 
$$
\psi:\Sigma_M^+\to\bS^1
$$ 
that semi-conjugates $\sigma$ to $A$. We may and will assume that $\psi$ carries the cylinder set $[\pm j]\subset \Sigma_M^+$ to the Markov partition piece of $A$ connecting $e^{\pm \pi i(j-1)/d}$ to $e^{\pm \pi ij/d}$.

\begin{rmk}
	See \cite{stad} for Markov partitions of Bowen-Series maps associated with more general Fuchsian punctured surface groups. These maps, however, are not continuous if the genus of the surface is greater than zero.
\end{rmk}

The unique measure of maximal entropy for the $\sigma-$action on $\Sigma_M^+$ (which is called the \emph{Parry measure} in symbolic dynamics) is given by the `uniform' Markov measure $\mu$ that assigns mass $\frac{1}{2d\cdot (2d-1)^n}$ to each cylinder set of rank $n+1$ ($n\geq 0$). The corresponding topological entropy is also $\ln(2d-1)$ (note that $2d-1$ is the largest eigenvalue of $M$). The existence of the semi-conjugacy $\psi$ now implies that

\begin{prop}\label{mme_symb}
	The measure of maximal entropy of $A$, which we denote by $\nu$, is the push-forward of the Parry measure $\mu$ under $\psi$; i.e., $\nu=\psi_\ast \mu$.
\end{prop}

\subsubsection{MME of $A$ in terms of Patterson-Sullivan measure}

Since the Bowen-Series map $A$ is cooked up from the Fuchsian group $\Gamma_0$, it is natural to ask whether the measure of maximal entropy $\nu$ of $A$ is related to the Patterson-Sullivan measure class of $\Gamma_0$. The following proposition gives a negative answer to this question (recall that a Patterson-Sullivan measure of $\Gamma_0$ lies in the class of the Haar measure $m$). 

\begin{prop}\label{not_abs_cont_lem}
	The measure $\nu$ is not mutually absolutely continuous with respect to the Haar measure $m$; i.e., $\nu$ and $m$ do not lie in the same measure class.
\end{prop}
\begin{proof} We learned this from Caroline Series. The proposition follows from the facts that
	\begin{enumerate}
		\item the action of the Fuchsian group $\Gamma_0$ on $S^1 = \partial \disk$ is of  type $III_1$.
		\item the action of the polynomial $z^k$ on $S^1 = \partial \disk$ is of  type $III_{\ln(k)}$.
	\end{enumerate} 
	See \cite{spatzier} for details.
\end{proof}

Fortunately, the free group on $d$ generators $F_{d}\cong\Gamma_0$ provides us with a Patterson-Sullivan measure (supported on the Gromov boundary of $F_d$) that is intimately related to $\nu$.

We denote the Cayley tree of $F_d$ by $X$, and equip it with the word metric. The group acts on the tree by isometries. 
The Gromov boundary of $F_d$ is denoted by $\partial X$. Note that we can naturally identify $\partial X$ with the shift space $\Sigma_M^+$. Visualizing the Cayley tree $X$ as dual to the $\Gamma_0-$tessellation of $\mathbb{D}$ (associated with the fundamental domain $R$), one sees in light of the identification $\partial X\cong \Sigma_M^+$ that the map $\psi$ is the (Floyd-)Cannon-Thurston map from $\partial X$ to $\bS^1$ (cf. \cite{Floyd}).

\begin{defn}\label{def-cone}
	Let $X$ denote a Cayley graph of a group $\G$. Let $g \in \G$ (thought of as a vertex of $X$). The cone of $g$ consists of the vertices $h \in X$ such that any geodesic $[1,g]$ followed by
	any geodesic $[g,h]$ is a geodesic $[1,h]$ in $X$ joining $1, h$.
\end{defn}
The next result enables us to connect $\nu$ to a suitable Patterson-Sullivan measure on $\partial X$.

\begin{lemma}\label{ps_parry_lem}
	The Patterson-Sullivan measure on $\partial X$ (with respect to the base point $1$ and the standard generating set) is given by the Parry measure $\mu$.
\end{lemma}
\begin{proof}
	Note that the number of words in $F_d$ of length $r$ is $2d\cdot (2d-1)^{r-1}$, for $r\geq 1$. Hence, the Patterson-Sullivan measure on $\partial X$ (with respect to the base point $1$ and the standard generating set) is a weak limit of the measures
\begin{equation}
	\mu_n:= \frac{\delta_1 + \sum_{j=1}^n \frac{1}{(2d-1)^j}\left(\sum_{\vert g\vert =j} \delta_g\right)}{1+  \sum_{j=1}^n \frac{2d\cdot (2d-1)^{j-1}}{(2d-1)^j}} = \frac{\delta_1 + \sum_{j=1}^n \frac{1}{(2d-1)^j}\left(\sum_{\vert g\vert =j} \delta_g\right)}{1+ \frac{2dn}{2d-1}}.
\label{ps_formula}
\end{equation}
	A straightforward computation now shows that the $\mu_{n+r}-$mass of the cone at a group element of length $r$ is:
	$$
	\frac{1}{1+\frac{2d(n+r)}{2d-1}} \cdot \frac{n+1}{(2d-1)^r},
	$$
	which tends to $\frac{1}{2d(2d-1)^{r-1}}$ as $n\to+\infty$. It follows that the Patterson-Sullivan measure on $\partial X$ assigns mass $\frac{1}{2d(2d-1)^{r-1}}$ to each cylinder set (in $\partial X$) of rank $r$. In view of the definition of $\mu$, the proof is now complete.
\end{proof}

Since $\nu=\psi_\ast\mu$, we conclude the following result.

\begin{prop}\label{mme_ps}
	The measure of maximal entropy $\nu$ of the Bowen-Series map $A$ is the push-forward of the Patterson-Sullivan measure $\mu$ on $\partial X$ (with respect to the base point $1$ and the standard generating set) under the (Floyd-)Cannon-Thurston map $\psi$.
\end{prop}

\subsubsection{$\Gamma_0-$invariance of the MME of $A$}

We now exploit the connection between $\nu$ and Patterson-Sullivan measures to exhibit $\Gamma_0-$invariance of the measure class of~$\nu$.

\begin{prop}\label{mme_meas_class_inv}
	For each $\gamma\in\Gamma_0$, the measures $\nu$ and $\gamma_\ast\nu$ are mutually absolutely continuous.
\end{prop}
\begin{proof}
	First note that the (Floyd-)Cannon-Thurston map $\psi$ semi-conjugates the $F_d-$action on $\partial X$ to the $\Gamma_0-$action on $\bS^1$ \cite{Floyd}. We will denote the element of $F_d$ corresponding to $\gamma\in\Gamma_0$ by $\widetilde{\gamma}$. 
	
	By Proposition~\ref{mme_ps} and the previous paragraph, the measure $\gamma_\ast\nu$ on $\bS^1$ is the push-forward of the measure $\widetilde{\gamma}_\ast\mu$ on $\partial X$ under $\psi$. Moreover, as $\mu$ is a Patterson-Sullivan measure on $\partial X$, it follows that $\widetilde{\gamma}_\ast\mu$ and $\mu$ are mutually absolutely continuous (see \cite[Theorem~5.4, Theorem~8.2]{coor}). It is now easy to see using the definition of push-forward of a measure that the measures $\nu=\psi_\ast(\mu)$ and $\gamma_\ast\nu=\psi_\ast(\widetilde{\gamma}_\ast\mu)$ are mutually absolutely continuous.
\end{proof}

\begin{rmk}\label{rmk_busemann}
For $\gamma\in\Gamma_0$, the Radon-Nikodym derivative $d(\gamma_\ast\nu)/d\nu$ can be written in terms of $\psi$ and the Radon-Nikodym derivative $d(\widetilde{\gamma}_\ast\mu)/d\mu$, which in turn can be computed from measures of cylinder sets (see \cite[\S 8]{coor} for a general method of describing such Radon-Nikodym derivatives in terms of Busemann functions).
\end{rmk}

Now observe that the Bowen-Series map $A$ does not depend  only on the group $\Gamma_0$, but also on the choice of the fundamental domain $R$. The translation of $R$ by an element $\gamma\in\Gamma_0$ is a different fundamental domain $\gamma\cdot R$ for $\Gamma_0$. We denote the Bowen-Series map of $\Gamma_0$ associated with the fundamental domain $\gamma\cdot R$ by $A^\gamma$. Clearly, $A^\gamma=\gamma\circ A\circ\gamma^{-1}$. Moreover, the unique measure of maximal entropy for the $A^\gamma-$action on $\bS^1$ is given by $\gamma_\ast\nu$. Proposition~\ref{mme_meas_class_inv} now implies the following.

\begin{cor}\label{mme_abs_cont}
	The measures of maximal entropy for the Bowen-Series maps associated with the fundamental domains $\gamma\cdot R$ (for $\gamma\in\Gamma_0$) are mutually absolutely continuous. In particular, all these measures have the same Hausdorff dimension.
\end{cor}

\subsubsection{MME for matings of Bowen-Series maps and polynomials}

Recall that Theorem~\ref{moduli_interior_mating_thm} provides us with a conformal mating of the canonical extension $\widehat{A}$ of the Bowen-Series map $A$ (associated with the fundamental domain $R$ of $\Gamma_0$) and the polynomial map $z^{2d-1}$. Also note that the restriction of this conformal mating on its Jordan curve limit set is topologically conjugate to $A\vert_{\bS^1}$.
The following description of the measure of maximal entropy of the conformal mating now follows from Propositions~\ref{mme_leb} and~\ref{mme_ps}.

\begin{prop}\label{mating_meas_max_ent}
	The unique measure of maximal entropy of the conformal mating of $\widehat{A}$ and $z^{2d-1}$ restricted to the limit set is equal to the push-forward of the normalized Lebesgue measure $m$ (which is the unique measure of maximal entropy of $z^{2d-1}\vert_{\bS^1}$) as well as the push-forward of the Patterson-Sullivan measure $\mu$ on $\partial X$ (with respect to the base point $1$ and the standard generating set) under appropriate conjugacies. In particular, the corresponding topological entropy is $\ln(2d-1)$.
\end{prop}

\subsubsection{Topological entropy of $A$ from a group-theoretic perspective}

The topological entropy $\ln(2d-1)$ of $A$ can be related to the \emph{volume entropy} of the group $F_d$, which measures
 the exponential growth rate of the number of words of length $n$ in a group (equivalently, the exponential growth rate of the number of group elements in a ball of radius $n$ around identity).

\begin{lemma}\label{vol_entropy_lem}
	The volume entropy of $F_d$ with respect to the standard (symmetric) set of generators and the critical exponent for the $F_d-$action on $X$ are both equal to $\ln(2d-1)$.
\end{lemma}
\begin{proof}
	Recall that the number of words in $F_d$ of length $r$ is $2d\cdot (2d-1)^{r-1}$, for $r\geq 1$. Hence, 
	$$
	\# \{ g\in F_d: \vert g\vert\leq n\}= 1+ 2d\sum_{r=1}^n (2d-1)^{r-1}=  1+ 2d\frac{(2d-1)^n-1}{2d-2},
	$$
	from which it follows that the volume entropy is $\ln(2d-1)$.
	
	Now consider the Poincar{\'e} series with exponent $s$:
	$$
	\zeta_{F_d}(s):= \sum_{g\in F_d}e^{-s\vert g\vert}= \sum_{n=1}^\infty\sum_{\substack{g\in F_d\\ \vert g\vert=n}} e^{-sn}= 2d \sum_{n=1}^\infty \frac{(2d-1)^{n-1}}{e^{sn}} = \frac{2d}{2d-1} \sum_{n=1}^\infty \left(\frac{2d-1}{e^s}\right)^n.
	$$
	Clearly, the series converges if and only if $2d-1< e^s \iff s> \ln(2d-1)$. In particular, the critical exponent is $\ln(2d-1)$.
\end{proof}

\begin{rmk}
	A connection between the topological entropy of Bowen-Series maps associated with cocompact Fuchsian groups and the volume entropy of suitable hyperbolic groups was established in \cite{los14}.
\end{rmk}

In \cite[Theorem~1]{haus-dim-sul}, Sullivan proved equality of  critical exponents and Hausdorff dimensions of limit sets for geometrically finite Kleinian groups. While the analogous result for hyperbolic groups follows from general consideration (cf. \cite[Theorem~8.3]{coor}  \cite[Theorem~15.8]{kap-ben}), we can give a simple proof in the present setting.

\begin{lemma}\label{hd_bdry_lem}
	The Hausdorff dimension of the Gromov boundary of $F_d$ equipped with the visual metric (in base $e$) is equal to $\ln(2d-1)$. Moreover, the $\ln(2d-1)-$dimensional Hausdorff measure $\mathscr{H}^{\ln(2d-1)}$ (on $\partial X$) 
	and $\mu$ are mutually absolutely continuous.
\end{lemma}
\begin{proof}
	The visual metric (in base $e$) on $\partial X$ is bi-Lipschitz to the ultra-metric given by $d(a,b)=e^{-\vert c\vert}$, where $c$ is the bifurcation point for the geodesic rays $[1,a)$ and $[1,b)$. Hence, it suffices to compute the Hausdorff dimension of $\partial X$ with respect to this ultra-metric.
	
	We first note that the $\mu$-measure of a cylinder of rank $n$ is equal to $\frac{1}{2d(2d-1)^{n-1}}$. On the other hand, the diameter of a cylinder of rank $n$ is $e^{-n}$.  Thus, 
	$$
	\mu(B(a,e^{-n}))=\frac{1}{2d(2d-1)^{n-1}}\ \implies \mu(B(a,r)) \sim r^{\ln(2d-1)}.
	$$
	The result now follows from standard results on Hausdorff dimension (for instance, see \cite[Proposition~4.9]{falconer}). In fact, we have shown that the $\ln(2d-1)-$dimensional Hausdorff measure is positive and finite. 
	The second statement is obvious from the above proof.
\end{proof}

\subsection{Maximal entropy measure for higher Bowen-Series maps}\label{sec_meas_th_hbs}

We now carry out a similar analysis for the measure of maximal entropy of a higher Bowen-Series map of a Fuchsian punctured sphere group. For simplicity of exposition, we work with the thrice punctured sphere case. 

Let us fix the Fuchsian thrice punctured sphere group $\Gamma_0$ of Section~\ref{sec-cfm} and a (closed) fundamental domain $W$ given by the quadrilateral with vertices at the fourth roots of unity. Further let $A\equiv A_{\Gamma_0, \textrm{hBS}}:\bS^1\to\bS^1$ be the associated higher Bowen-Series map of $\Gamma_0$. As $A\vert_{\bS^1}$ is topologically conjugate to $z^4\vert_{\bS^1}$, the topological entropy of $A$ is equal to $\ln(4)$. We denote the unique measure of maximal entropy for $A\vert_{\bS^1}$ by $\nu$.

\subsubsection{Topological dynamics} As $A$ is an expansive circle covering of degree $4$, there exists a homeomorphism 
$$\phi:\bS^1\to\bS^1
$$ 
that conjugates $p:z\mapsto z^{4}$ to $A$. Using the conjugacy $\phi$, one can write the measure of maximal entropy $\nu$ for $A\vert_{\bS^1}$ as the push-forward measure $\phi_\ast m$.

\subsubsection{Symbolic dynamics} The pieces of $A$ are given by $g^{\pm 1}, h^{\pm 1}, g\circ h^{-1}, h\circ g^{-1}$. Their intervals of definition yield a Markov partition (counter-clockwise starting at $1$) for $A$ with transition matrix
$$
M:=	\begin{bmatrix}
	1&1&1&1&0&0\\
	1&1&0&0&1&1\\
	0&0&1&1&1&1\\
	1&1&1&1&0&0\\
	1&1&0&0&1&1\\
	0&0&1&1&1&1\\
\end{bmatrix}.
$$
The above transition matrix gives rise to a one-sided subshift of finite type 
$$
\sigma: \Sigma_M^+\to\Sigma_M^+,
$$ 
where $\Sigma_M^+$ consists of $M-$admissible infinite words in $\{1, 2, \cdots, 6\}^{\mathbb{N}}$, and $\sigma$ is the left-shift map. As before, we metrize $\Sigma_M^+$ with the usual ultra-metric (in base $e$). Since $A$ is expansive, one obtains a continuous surjection 
$$
\psi:\Sigma_M^+\to\bS^1
$$ 
that semi-conjugates $\sigma$ to $A$, and sends the cylinders of rank $1$ to the Markov partition pieces of $A$.

The \emph{Parry measure} (i.e., the unique measure of maximal entropy) for the $\sigma-$action on $\Sigma_M^+$ is given by the `uniform' Markov measure $\mu$, that assigns mass $\frac{1}{6\cdot 4^n}$ to each cylinder set of rank $n+1$ ($n\geq 0$). The corresponding topological entropy is also $\ln(4)$ (note that $4$ is the largest eigenvalue of $M$), and $\nu=\psi_\ast \mu$.

\subsubsection{Patterson-Sullivan measure} We now turn our attention to the the free group $F_2\cong \langle g\rangle\ast\langle h\rangle$ with the generating set $\{g^{\pm 1}, h^{\pm 1}, g\circ h^{-1}, h\circ g^{-1}\}$ (which are precisely the pieces of $A$).

We denote the Cayley graph of $F_2$ with respect to the above (non-standard) generating set by $X$, and equip it with the word metric. 
Note that we can naturally identify the Gromov boundary $\partial X$ with the shift space $\Sigma_M^+$. With this identification, the boundary at infinity of the cone at a generator is the corresponding cylinder set in $\Sigma_M^+$.

\begin{figure}[h!]
\captionsetup{width=0.96\linewidth}
	\begin{tikzpicture}
		
		\node[anchor=south west,inner sep=0] at (0,0) {\includegraphics[width=0.9\linewidth]{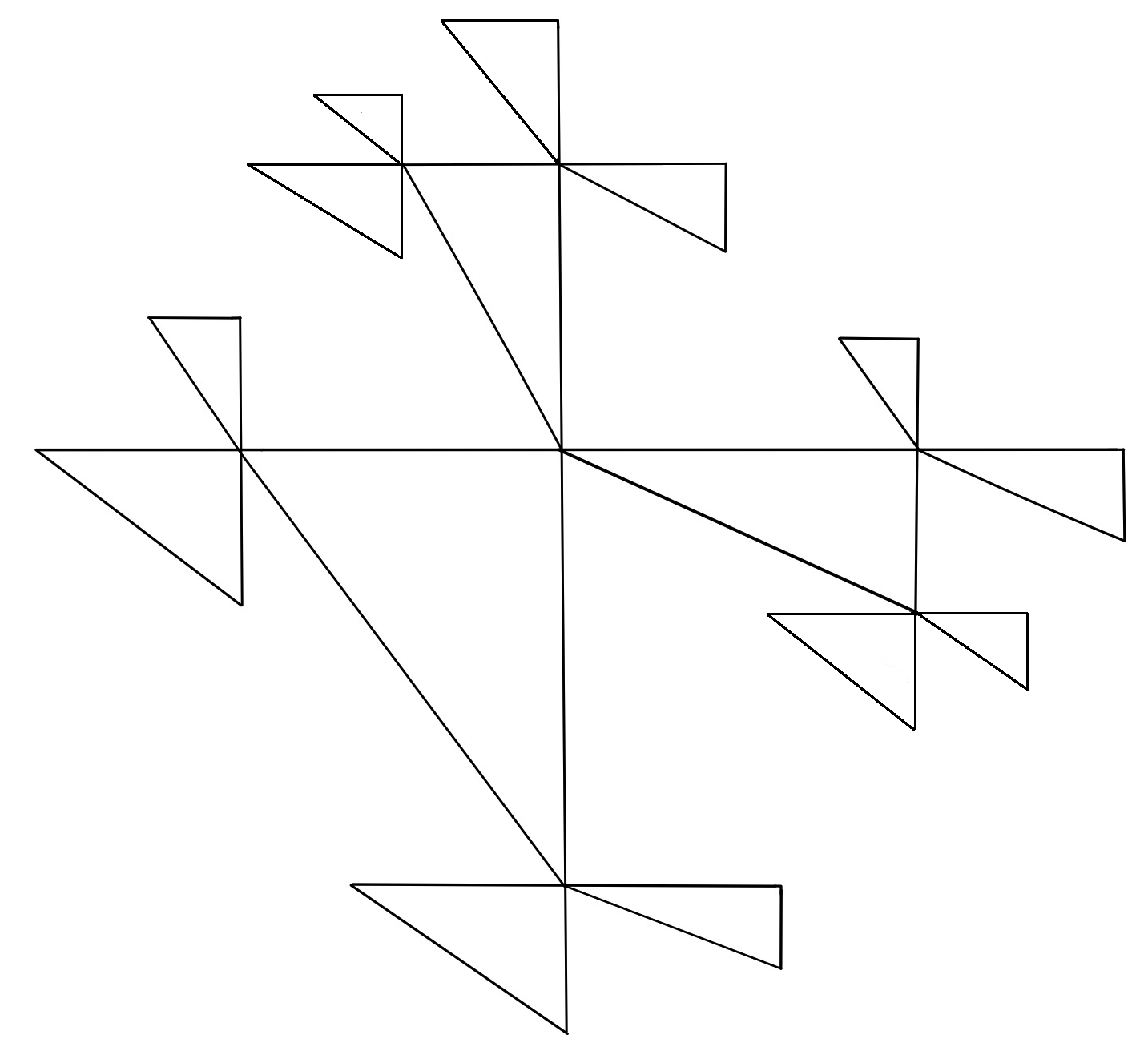}};
		\node at (5.4,5.8) {\begin{scriptsize}$1$\end{scriptsize}};
		\node at (8.84,5.8) {\begin{scriptsize}$g$\end{scriptsize}};
		\node at (9.48,4.6) {\begin{scriptsize}$gh^{-1}$\end{scriptsize}};
		\node at (7,4.3) {\begin{tiny}$gh^{-1}g^{-1}$\end{tiny}};
		\node at (9.2,3.1) {\begin{scriptsize}$gh^{-2}$\end{scriptsize}};
		\node at (10.7,4.4) {\begin{scriptsize}$gh^{-1}g$\end{scriptsize}};
		\node at (10.4,3.5) {\begin{scriptsize}$(gh^{-1})^2$\end{scriptsize}};
		\node at (2,5.8) {\begin{scriptsize}$g^{-1}$\end{scriptsize}};	
		\node at (2.84,7.3) {\begin{scriptsize}$g^{-1}h$\end{scriptsize}};	
		\node at (0.88,7.3) {\begin{tiny}$g^{-1}hg^{-1}$\end{tiny}};	
		\node at (2.4,4.2) {\begin{scriptsize}$g^{-1}h^{-1}$\end{scriptsize}};	
		\node at (5.4,8.64) {\begin{scriptsize}$h$\end{scriptsize}};
		\node at (4.4,9.04) {\begin{scriptsize}$hg^{-1}$\end{scriptsize}};
		\node at (4,9.75) {\begin{tiny}$hg^{-1}h$\end{tiny}};
		\node at (2.56,9.5) {\begin{tiny}$(hg^{-1})^2$\end{tiny}};
		\node at (4,7.7) {\begin{tiny}$hg^{-1}h^{-1}$\end{tiny}};
		\node at (2.06,8.8) {\begin{scriptsize}$hg^{-2}$\end{scriptsize}};
		\node at (7.5,8.8) {\begin{scriptsize}$hg$\end{scriptsize}};
		\node at (7.4,7.75) {\begin{scriptsize}$hgh^{-1}$\end{scriptsize}};
		\node at (5.32,1.44) {\begin{scriptsize}$h^{-1}$\end{scriptsize}};
		\node at (2.88,1.58) {\begin{scriptsize}$h^{-1}g^{-1}$\end{scriptsize}};
		\node at (8.2,1.6) {\begin{scriptsize}$h^{-1}g$\end{scriptsize}};
		\node at (7.75,0.7) {\begin{tiny}$h^{-1}gh^{-1}$\end{tiny}};
		\node at (11.1,6.3) {\begin{scriptsize}$g^2$\end{scriptsize}};
		\node at (11.2,4.9) {\begin{scriptsize}$g^2h^{-1}$\end{scriptsize}};
		\node at (0.1,5.8) {\begin{scriptsize}$g^{-2}$\end{scriptsize}};	
		\node at (5.8,10.2) {\begin{scriptsize}$h^2$\end{scriptsize}};
		\node at (4.6,10.5) {\begin{scriptsize}$h^2g^{-1}$\end{scriptsize}};
		\node at (6,0.36) {\begin{scriptsize}$h^{-2}$\end{scriptsize}};
		\node at (9.4,7.1) {\begin{scriptsize}$gh$\end{scriptsize}};
		\node at (7.8,7.1) {\begin{scriptsize}$ghg^{-1}$\end{scriptsize}};
		
	\end{tikzpicture}
	\caption{The words of length one and two in the Cayley graph of $F_2$ with respect to the generating set $\{g^{\pm 1}, h^{\pm 1}, g\circ h^{-1}, h\circ g^{-1}\}$ are displayed.}
	\label{cayley_hbs_fig}
\end{figure}
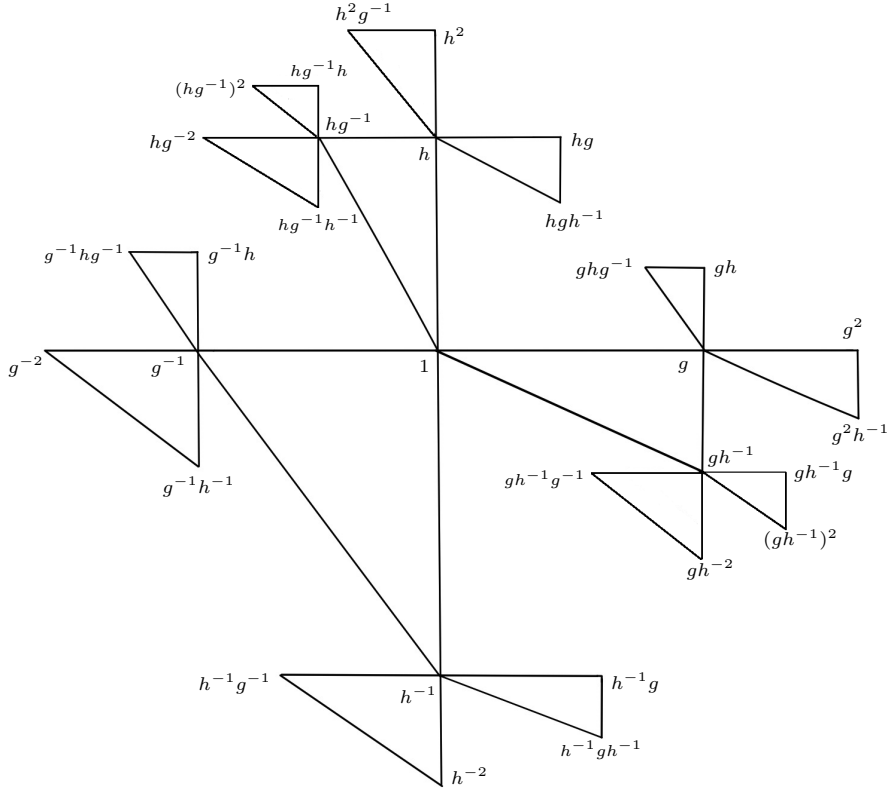

\begin{rmk}
	The higher Bowen-Series map $A$ gives rise to a Markov map $\widetilde{A}$ acting on the Gromov boundary $\partial X$ (such that $\widetilde{A}$ is orbit equivalent to the $F_2-$action on $\partial X$) in the following way: for $\alpha\in\{g^{\pm 1}, h^{\pm 1}, g\circ h^{-1}, h\circ g^{-1}\}$, the map $\widetilde{A}$ acts on the boundary at infinity of $\textrm{Cone}(\alpha)$ as $\alpha^{-1}$.
\end{rmk}

\begin{lemma}\label{ps_parry_lem_hbs}
	The Patterson-Sullivan measure on $\partial X$ (with respect to the base point $1$ and the generating set $\{g^{\pm 1}, h^{\pm 1}, g\circ h^{-1}, h\circ g^{-1}\}$) is given by the Parry measure $\mu$.
\end{lemma}
\begin{proof}
	It is easy to see from the generators and relations (equivalently, from the Cayley graph depicted in Figure~\ref{cayley_hbs_fig}) that the number of words of length $r$ in $F_2$ (with respect to generating set $\{g^{\pm 1}, h^{\pm 1}, g\circ h^{-1}, h\circ g^{-1}\}$) is $6\cdot 4^{r-1}$ ($r\geq 1$).
	A computation similar to the one in the proof of Lemma~\ref{ps_parry_lem} now readily shows that the Patterson-Sullivan measure in question (on $\partial X$) assigns mass $\frac{1}{6\cdot 4^{r-1}}$ to each cylinder set (in $\partial X$) of rank $r$. Thus, the Patterson-Sullivan measure agrees with $\mu$ on each cylinder set.
\end{proof}

\begin{prop}\label{mme_ps_hbs}
	The measure of maximal entropy $\nu$ of the higher Bowen-Series map $A$ is the push-forward of the Patterson-Sullivan measure $\mu$ on $\partial X$ (with respect to the base point $1$ and the generating set $\{g^{\pm 1}, h^{\pm 1}, g\circ h^{-1}, h\circ g^{-1}\}$) under $\psi$.
\end{prop}

\subsubsection{$\Gamma_0-$invariance of the class of $\nu$} Proposition~\ref{mme_ps_hbs}, $F_d-$invariance of the measure class of the Patterson-Sullivan measure $\mu$ (on $\partial X$) \cite[Theorem~5.4, Theorem~8.2]{coor}, and the fact that the map $\psi$ semi-conjugates the $F_d-$action on $\partial X$ to the $\Gamma_0-$action on $\bS^1$ together imply the following.

\begin{prop}\label{mme_meas_class_inv_hbs}
	For each $\gamma\in\Gamma_0$, the measures $\nu$ and $\gamma_\ast\nu$ are mutually absolutely continuous.
\end{prop}

\subsubsection{MME for mating} According to Theorem~\ref{thm-cfdmateable}, there exists a conformal mating of the canonical extension $\widehat{A}$ of the higher Bowen-Series map $A$ (associated with the closed fundamental domain $W$ of $\Gamma_0$) and the polynomial map $z^{4}$ such that the restriction of this conformal mating on its Jordan curve limit set is topologically conjugate to $A\vert_{\bS^1}$.
The interpretation of the measure $\nu$ in terms of $m$ and $\mu$ implies the following.

\begin{prop}\label{mating_meas_max_ent_hbs}
	The unique measure of maximal entropy of the conformal mating of $\widehat{A}$ and $z^4$ restricted to the limit set is equal to the push-forward of the MME of $z^{4}$ on $\bS^1$ as well as the push-forward of the Patterson-Sullivan measure on $\partial X$ (with respect to the base point $1$ and the generating set $\{g^{\pm 1}, h^{\pm 1}, g\circ h^{-1}, h\circ g^{-1}\}$) under appropriate conjugacies. In particular, the corresponding topological entropy is $\ln(4)$.
\end{prop}

\subsubsection{Topological entropy, volume entropy, and Hausdorff dimension}

Since there are $6\cdot 4^{r-1}$ words of length $r$ in $F_2$ (with respect to generating set $\{g^{\pm 1}, h^{\pm 1}, g\circ h^{-1}, h\circ g^{-1}\}$), the arguments used in the proof of Lemmas~\ref{vol_entropy_lem},~\ref{hd_bdry_lem} apply mutatis mutandis to the current setting and prove equality of volume entropy, critical exponent, and Hausdorff dimension of the Gromov boundary. Moreover, this number coincides with the topological entropy of $A\vert_{\bS^1}$.

\begin{lemma}\label{te_ve_ce_hd}
	\noindent\begin{enumerate}
		\item The volume entropy of $F_2$ with respect to the generating set $\{g^{\pm 1}, h^{\pm 1}, g\circ h^{-1}, h\circ g^{-1}\}$ and the critical exponent for the $F_2-$action on $X$ are both equal to $\ln(4)$.
		\item The Hausdorff dimension of $\partial X$ equipped with the visual metric (in base $e$) with respect to the generating set $\{g^{\pm 1}, h^{\pm 1}, g\circ h^{-1}, h\circ g^{-1}\}$ is equal to $\ln(4)$. Moreover, the $\ln(4)-$dimensional Hausdorff measure $\mathscr{H}^{\ln(4)}$ and $\mu$ (on $\partial X$) are mutually absolutely continuous.
	\end{enumerate}
\end{lemma}

\subsection{Hausdorff dimension of measure of maximal entropy: thrice punctures sphere}\label{hd_computations_subsec}

Recall from Proposition~\ref{not_abs_cont_lem} that the measure of maximal entropy of the Bowen-Series map of a Fuchsian punctured sphere group does not lie in the Lebesgue measure class. In this section, we will prove a sharper version of this fact in the thrice punctured sphere case. 

Specifically, we will show that the Hausdorff dimension 
$$
\mathrm{HD}(\nu):= \mathrm{inf} \{\mathrm{HD}(Y): Y\subset \bS^1,\ \nu(Y)=1\}
$$
of the MME $\nu$ of the (higher) Bowen-Series map of a Fuchsian thrice punctured sphere group is strictly less than $1$. This statement should be compared with the analogous result that except for some very special cases, the Hausdorff dimension of the measure of maximal entropy of a rational map is strictly smaller than the Hausdorff dimension of the Julia set \cite{Zdunik}.

In what follows, we will cook up a self-map of the interval $[0,1]$ from the (higher) Bowen-Series map under consideration, and relate the Hausdorff dimension of the MME of the (higher) Bowen-Series map to that of the MME of the associated self-map. This will allow us to obtain the desired upper bounds. We note that while this intermediate step is not essential for the Bowen-Series map (in this case, one can obtain the upper bound of Proposition~\ref{hd_less_than_one_prop} by working directly with the Bowen-Series map), this method yields additional information in the higher Bowen-Series case. Indeed, this reduction step connects the MME of the higher Bowen-Series map to a classical measure from number theory, which makes known results applicable to the current setting and gives a better estimate for the Hausdorff dimension of the MME.

\subsubsection{The Bowen-Series case}\label{bs_case_subsubsec}
Since the Teichm{\"u}ller space of a thrice punctured sphere is a singleton, we may, without loss of generality, work with the Bowen-Series map $A_{G_2}$ introduced in Subsection~\ref{b_s_punc_sphere_subsec}. Note that every non-identity element $g$ in the free group $G_2$ admits a unique shortest representation with respect to the symmetric generating set $\{g_1^{\pm 1}, g_2^{\pm 1}\}$ introduced in Subsection~\ref{b_s_punc_sphere_subsec}. The length of this shortest representation is called the \emph{length} of $g$ (the length of the identity element is defined to be zero). Recall that the ideal polygon in $\mathbb{D}$ with vertices at the fourth roots of unity is a fundamental domain for the $G_2-$action on $\disk$, and hence its translates under elements of $G_2$ yield a tiling $\mathcal{T}_{G_2}$ of $\disk$. We call this fundamental domain the \emph{rank $0$ tile} and its translate under an element $g\in G_2$ of length $k$ a \emph{rank $k$ tile} for the above tessellation. 

We will use a specific symmetric property of the tessellation $\mathcal{T}_{G_2}$ which we now describe. Following Section~\ref{bs_sec}, we denote the ideal polygon in $\disk$ with vertices at the fourth roots of unity by $R$, and its edges by $C_{\pm 1}, C_{\pm 2}$. Let us further denote the anti-M{\"o}bius reflections in these edges by $\rho_{\pm 1},\rho_{\pm 2}$, and the reflection group generated by $\rho_{\pm 1},\rho_{\pm 2}$ by $\mathscr{G}$. Note that the polygon $R$ is invariant under the actions of $\mathscr{R}$ and $\iota$, where $\mathscr{R}$ is rotation by angle $\pi/2$ and $\iota$ is the reflection in the real axis. It follows that conjugation by $\mathscr{R}$ and $\iota$ act as permutations on the generating set $\{\rho_{\pm 1}, \rho_{\pm 2}\}$ of $\mathscr{G}$, and hence $\mathscr{R}$ and $\iota$ conjugate $\mathscr{G}$ to itself. Hence, the $\mathscr{G}-$tessellation $\mathcal{T}_{\mathscr{G}}$ of $\disk$ arising from the fundamental domain $\overline{R}$ (closure taken in $\disk$) is preserved by both $\mathscr{R}$ and $\iota$. Note furthermore that the relations $g_i=\iota\circ\rho_i=\rho_{-i}\circ\iota$, $i\in\{1,2\}$, and $\iota-$invariance of the tessellation $\mathcal{T}_{\mathscr{G}}$ imply that the tessellations $\mathcal{T}_{\mathscr{G}}$ and $\mathcal{T}_{G_2}$ are the same. It follows that the tessellation $\mathcal{T}_{G_2}$ of $\mathbb{D}$ is symmetric with respect to $\pi/2-$rotation $\mathscr{R}$.

For the current purpose, it will be more convenient to work with the upper half-plane model. To this end, consider the M\"obius transformation $M(z)=i(1-z)/(1+z)$ which carries the unit disk onto the upper half-plane such that $M(1)=0, M(i)=1, M(-1)=\infty$, and $M(-i)=-1$. Hence, $M$ sends the ideal polygon in $\mathbb{D}$ with vertices at the fourth roots of unity to the ideal polygon in $\mathbb{H}$ with vertices at $-1, 0, 1$, and $\infty$.
The map $M$ conjugates $G_2$ to a discrete subgroup $\pmb{G}_2$ of $\mathrm{PSL}_2(\R)$, and transports the $G_2-$tessellation of $\disk$ defined in the previous paragraph to a $\pmb{G}_2-$tessellation of $\mathbb{H}$. One defines tiles of this tessellation and their ranks as in the previous paragraph. Moreover, $M$ conjugates the Bowen-Series map $A_{G_2}$ to the map
$$
\tau:\R\cup\{\infty\}\to\R\cup\{\infty\},\qquad 
\tau(t)= \left\{\begin{array}{ll}
                    \vspace{1mm}

                    t+2, \qquad t\in\left[-\infty,-1\right],  \\
                    \vspace{2mm}
                    
                     \frac{t}{1+2t}, \qquad t\in\left[-1,0\right],  \\
                    \vspace{2mm}
                    
                     \frac{t}{1-2t}, \qquad t\in\left[0,1\right],\\
                     \vspace{2mm}

                     t-2, \qquad t\in\left[1,+\infty\right].
                                          \end{array}\right. 
$$
By construction, $\tau$ maps $[0,\frac13]$ to $[0,1]$, $[\frac13,\frac12]$ to $[1,+\infty]$, and $[\frac12,1]$ to $[-\infty,-1]$ (see Figure~\ref{bs_uhp_fig}).

Since Euclidean isometric rotation $\mathscr{R}:z\mapsto iz$ (about the origin) respects the $G_2-$~tessellation of $\mathbb{D}$, it follows that the conformal rotation $\mathscr{R}_{\mathbb{H}}(w)= M(i\cdot M^{-1}(w))=\frac{1+w}{1-w}\in\mathrm{PSL}_2(\R)$ (about $i$) respects the corresponding $\pmb{G}_2-$tessellation of $\mathbb{H}$.

This allows one to construct a self-map of $[0,1)$ associated with $\tau$:
 $$F:[0,1)\to[0,1),\qquad \displaystyle F(t)= \left\{\begin{array}{ll}
                   \vspace{2mm}

                   \hspace{6mm}  \tau(t) \hspace{8.4mm} = \hspace{2mm} \frac{x}{1-2x} ,\qquad t\in\left[0,\frac13\right),  \\
                    \vspace{2mm}

                    \left(\mathscr{R}_{\mathbb{H}}^{-1}\circ\tau\right)(t) \hspace{0.8mm} =  \hspace{2mm} \frac{3x-1}{1-x}, \qquad t\in\left[\frac13,\frac12\right),  \\
                     \vspace{2mm}

                    \left(\mathscr{R}_{\mathbb{H}}^{ 2}\circ\tau\right)(t) \hspace{1.6mm} = \hspace{2mm} \frac{2x-1}{x}, \qquad t\in\left[\frac12,1\right).
                                          \end{array}\right. 
$$
(See Figure~\ref{first_return_derivative_fig}.) The symmetry of the $\pmb{G}_2$-tessellation of $\mathbb{H}$ under the conformal rotation $\mathscr{R}_{\mathbb{H}}$ implies that $F$ sends the ideal vertices of tiles of a given rank to the ideal vertices of tiles of the previous rank.

\begin{figure}[h!]	
\captionsetup{width=0.96\linewidth}
	\begin{tikzpicture}
		\node[anchor=south west,inner sep=0] at (0,0) {\includegraphics[width=0.46\linewidth]{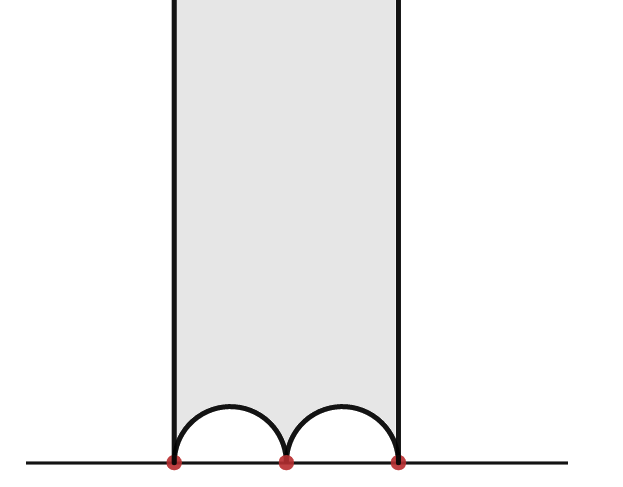}};
		\node[anchor=south west,inner sep=0] at (6,-0.1) {\includegraphics[width=0.46\linewidth]{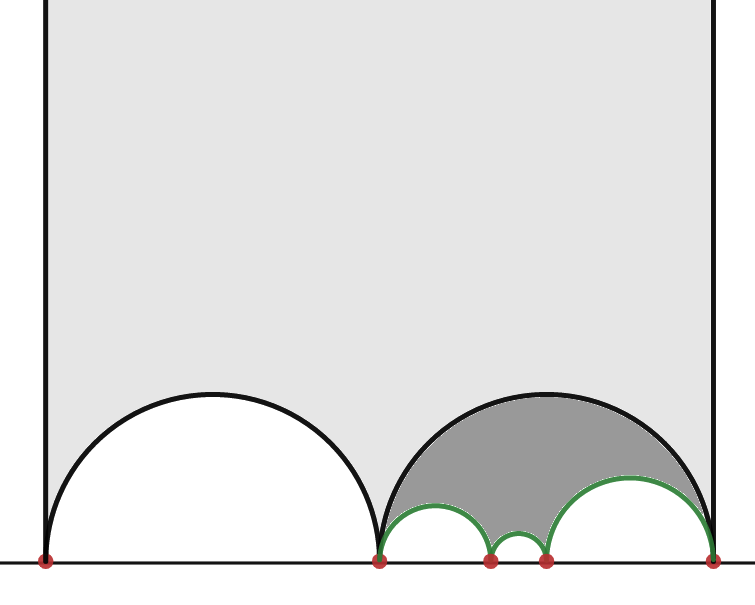}};
		\node at (1.5,-0.1) {$-1$};
		\node at (2.7,-0.1) {$0$};
		\node at (3.75,-0.1) {$1$};
		\node at (6.4,-0.1) {$-1$};
		\node at (8.94,-0.1) {$0$};
		\node at (9.8,-0.16) {$\frac13$};
		\node at (10.3,-0.16) {$\frac12$};
		\node at (11.5,-0.1) {$1$};
			\end{tikzpicture}
	\caption{Left: A fundamental polygon of the thrice punctured sphere Fuchsian group $\pmb{G}_2=M\circ G_2\circ M^{-1}$ with ideal vertices at $0, \pm 1$, and $\infty$. Right: A tile of rank one with ideal vertices at $0, \frac12, \frac13$, and $1$ is shown.}
	\label{bs_uhp_fig}
\end{figure}
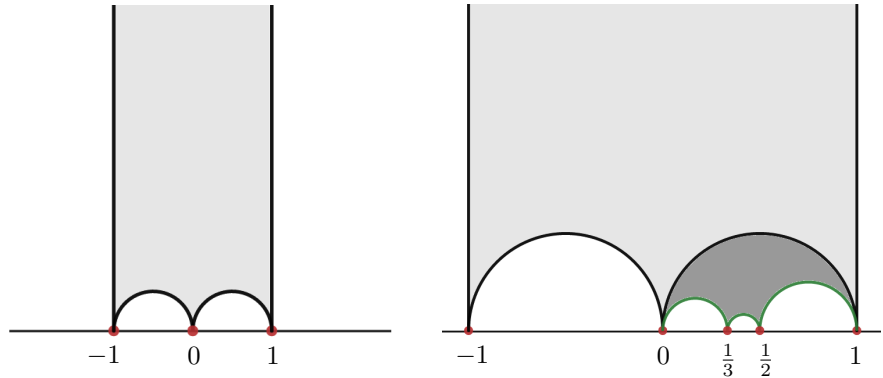

Let $\phi:\mathbb{S}^1\to\mathbb{S}^1$ be the homeomorphism conjugating $z^3$ to the Bowen-Series map $A_{G_2}$ with $\phi(1)=1$. As the chosen fundamental domain of $A_{G_2}$ is symmetric under rotation by $\pi/2$, one readily sees that the map $\phi$ commutes with $z\mapsto iz$. Using this, it is straightforward to verify that the tripling map $\times_{3}:[0,1)\to[0,1)$ 
$$
\times_{3}(x)=\left\{\begin{array}{ll}
                     \vspace{2mm}

                     3x, \hspace{12mm} x\in\left[0,\frac13\right),  \\
                       \vspace{2mm}

                     3x-1, \hspace{6mm} x\in\left[\frac13,\frac23\right),\\
                       \vspace{2mm}

                     3x-2, \hspace{6mm} x\in\left[\frac23,1\right),
                                          \end{array}\right.
$$
is topologically conjugate to $F$ via 
$$
H:[0,1]\to[0,1],\ x\mapsto M(\phi(E(x))),\ \mathrm{where}\ E(x)=e^{2\pi i\frac{x}{4}} .
$$
Due to the conjugation property, the homeomorphism $H$ sends the rational numbers $k/3^n$ (which are the $n$-th preimages of $0$ under $\times_{3}$) to the ideal vertices of tiles of rank $n$ (which are the $n$-th preimages of $0$ under $F$).

Also note that the Lebesgue measure $\overline{m}$ on $[0,1]$ is the measure of maximal entropy for $\times_{3}$, and hence, 
\begin{equation}
\nu':=H_\ast \overline{m}
\label{nu_dash_def}
\end{equation}
is the measure of maximal entropy for $F$.
\medskip

\begin{lemma}\label{hd_equal_lem} 
Let $\nu, \nu'$ be the measures of maximal entropy of the Bowen-Series map $A_{G_2}$ and the map $F$, respectively. Then, 
$\mathrm{HD}(\nu') = \mathrm{HD}(\nu)$.
\end{lemma}
\begin{proof}
We first observe that the M{\"o}bius map $M:\{e^{i\theta}: \theta\in[0,\pi/2]\}\to [0,1]$ is bi-Lipschitz, and hence preserves Hausdorff dimension (this can, for instance, be deduced from the fact that $M(e^{i\theta})=\tan(\frac{\theta}{2})$). Thus, by definition of $\nu'$ (see Equation~\ref{nu_dash_def}), we have
$$
\mathrm{HD}(\nu')= \mathrm{HD}((\phi\circ E)_\ast \overline{m}).
$$ 
Note that the measure $E_\ast(\overline{m})$ is simply the normalized Lebesgue measure on the arc $\{e^{i\theta}: \theta\in [0,\pi/2]\}\subset\mathbb{S}^1$.

Now choose $A\subset\mathbb{S}^1$ with $\nu(A)=1$, and set $A':=A\cap \{e^{i\theta}: \theta\in [0,\pi/2]\}$. By definition, the set $\phi^{-1}(A)$ has full measure with respect to the Haar measure $m$ on $\mathbb{S}^1$, and hence, $m(\phi^{-1}(A'))=1/4$ (here we have used the fact that $\phi$ maps the first quadrant of $\mathbb{S}^1$ to itself). This implies that $\phi^{-1}(A')$ is a full measure set with respect to $E_\ast \overline{m}$, and thus in turn $A'$ is a full measure set with respect to $(\phi\circ E)_\ast \overline{m}$. Therefore, 
$$
\mathrm{HD}((\phi\circ E)_\ast\overline{m})\leq \mathrm{HD}(A')\leq \mathrm{HD}(A).
$$ 
Taking the infimum over all full $\nu$-measure subsets $A$ of $\mathbb{S}^1$, we conclude that 
$$
\mathrm{HD}(\nu')= \mathrm{HD}((\phi\circ E)_\ast\overline{m})\leq \mathrm{HD}(\nu).
$$

For the opposite inequality, pick $A'\subset \{e^{i\theta}: \theta\in [0,\pi/2]\}$ with full $(\phi\circ E)_\ast \overline{m}$-measure. Define $A$ to be the symmmetrization of $A'$ under rotation by $\pi/2$. As $\phi^{-1}(A')$ has full measure with respect to $E_\ast \overline{m}$, we have that $m(\phi^{-1}(A'))=1/4$. Since $\phi$ commutes with multiplication by $i$, it now follows that $m(\phi^{-1}(A))=1$; i.e., $A$ has $\nu$-measure $1$. Therefore, 
$$
\mathrm{HD}(\nu)\leq \mathrm{HD}(A) = \mathrm{HD}(A')
$$
Finally, taking the infimum over all full $(\phi\circ E)_\ast \overline{m}$-measure subsets $A'$ of $\{e^{i\theta}: \theta\in [0,\pi/2]\}$, we have that
$$
\mathrm{HD}(\nu)\leq \mathrm{HD}((\phi\circ E)_\ast \overline{m}) = \mathrm{HD}(\nu').
$$
\end{proof}



\begin{prop}\label{hd_less_than_one_prop}
Let $\nu$ be the measure of maximal entropy of the Bowen-Series map $A_{G_2}$. Then, $\mathrm{HD}(\nu) < 1$.
\end{prop}
\begin{proof}
By Lemma~\ref{hd_equal_lem}, it suffices to show that $\mathrm{HD}(\nu') < 1$, where $\nu'$ is the measure of maximal entropy of $F$ (defined by Equation~\ref{nu_dash_def}).
The following relation between Hausdorff dimension, entropy, and Lyapunov exponent is standard (see \cite[\S 10]{PU}, \cite{HR92}):
$$
\mathrm{HD}(\nu') = \frac{\ln(3)}{\int_0^1 \ln\vert F'\vert d\nu'}\ .
$$
\begin{figure}[h!]	
\captionsetup{width=0.96\linewidth}
	\begin{tikzpicture}
		\node[anchor=south west,inner sep=0] at (0,0) {\includegraphics[width=0.5\linewidth]{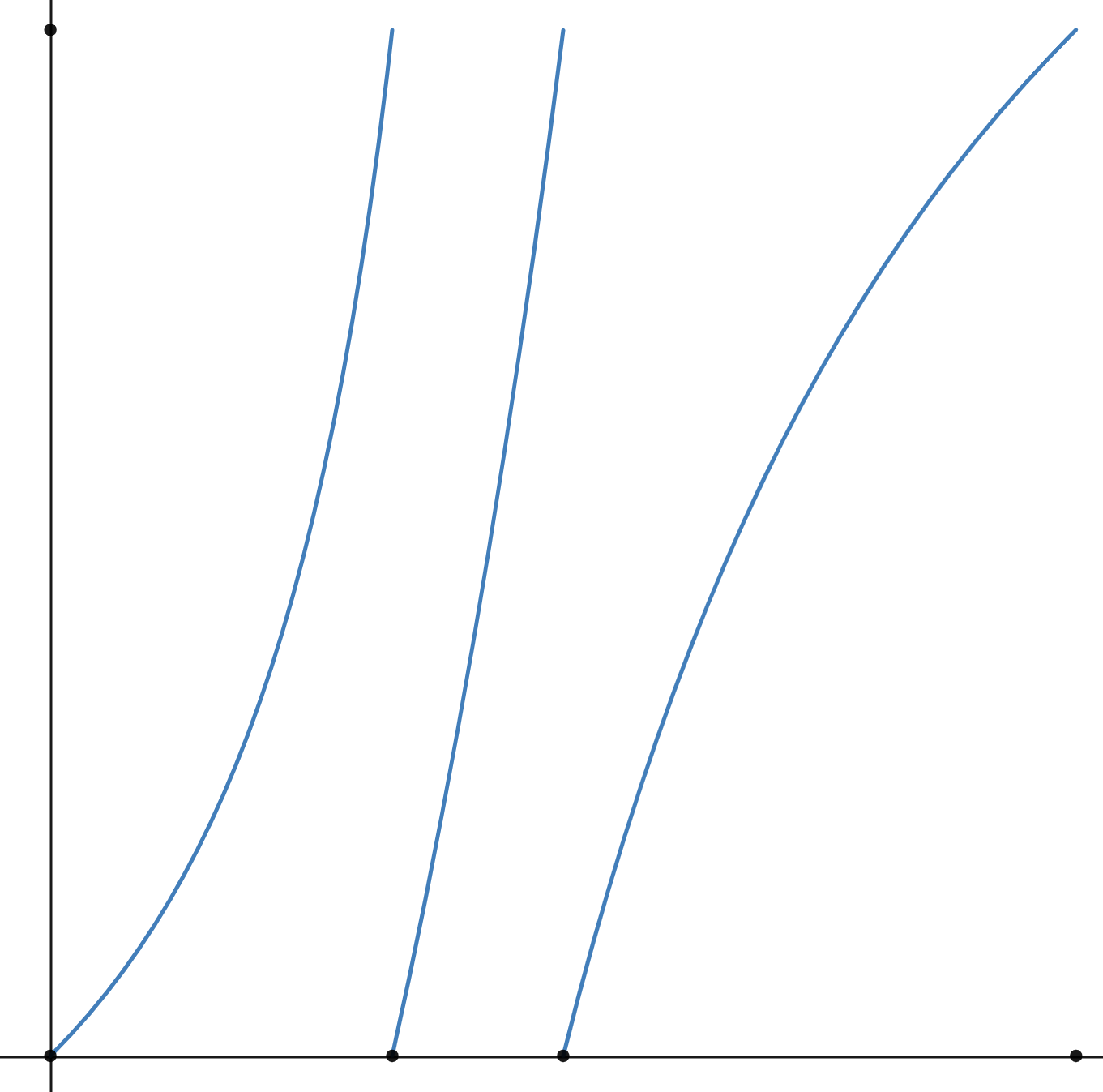}};
		\node[anchor=south west,inner sep=0] at (7.5,0) {\includegraphics[width=0.36\linewidth]{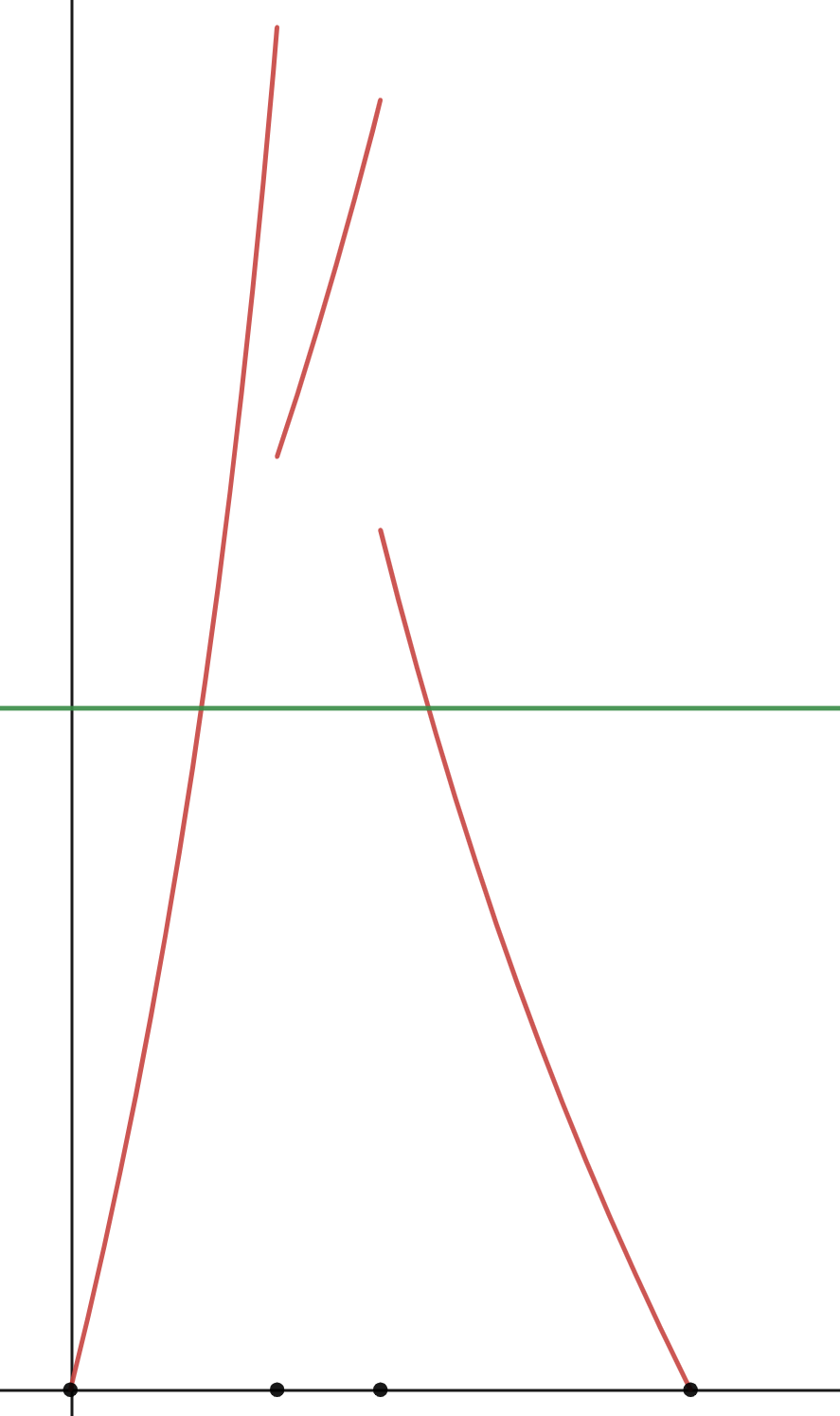}};
		\node at (0.3,-0.2) {$(0,0)$};
		\node at (2.2,-0.2) {$(\frac13,0)$};
		\node at (3.4,-0.2) {$(\frac12,0)$};
		\node at (6.2,-0.2) {$(1,0)$};
		\node at (-0.25,6.1) {$(0,1)$};
		\node at (7.9,-0.25) {\begin{small}$(0,0)$\end{small}};
		\node at (9,-0.25) {\begin{small}$(\frac13,0)$\end{small}};
		\node at (9.55,0.45) {\begin{small}$(\frac12,0)$\end{small}};
		\node at (11.2,-0.25) {\begin{small}$(1,0)$\end{small}};
		\node at (11.25, 3.54) {$y= \ln(3)$};
			\end{tikzpicture}
	\caption{Left: The graph of $F$. Right: The graph of $\ln\vert F'\vert$.}
	\label{first_return_derivative_fig}
\end{figure}
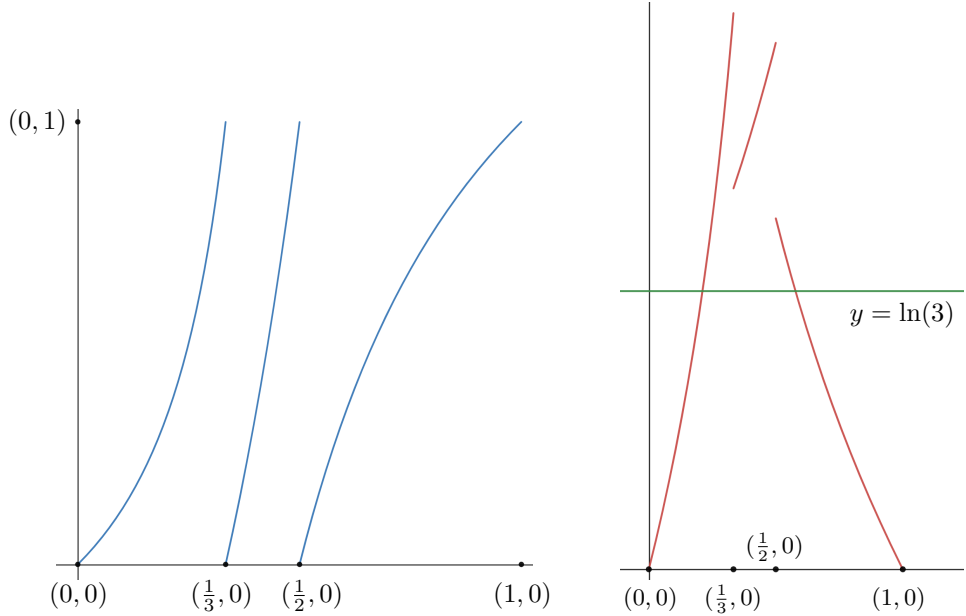

We also have the following explicit description of $\ln\vert F'\vert$ on $(0,1)\setminus\{\frac13, \frac12\}$:
$$ \displaystyle \ln\vert F'\vert(t)= \left\{\begin{array}{ll}
                   \vspace{2mm}

                    \ln\frac{1}{\left(1-2x\right)^{2}}\ ,\qquad t\in\left(0,\frac13\right),  \\
                    \vspace{2mm}

                    \ln\frac{2}{\left(1-x\right)^{2}}\ , \qquad t\in\left(\frac13,\frac12\right),  \\
                    \vspace{2mm}

                    \ln\frac{1}{x^{2}}\ , \qquad t\in\left(\frac12,1\right).
                                          \end{array}\right. 
$$
(See Figure~\ref{first_return_derivative_fig}.)

Our goal is to show that $\ln(3)$ is a strict lower bound for the Lyapunov exponent of $F$. To this end, we first note that by definition of $\nu'$, the ideal vertices (in $[0,1]$) of the tiles of rank up to three divide the unit interval into $3^3$ sub-intervals each of which has $\nu'$-mass $1/3^3$. The endpoints of these intervals are displayed in Figure~\ref{deep_ranks_tiling_fig}.
\begin{figure}[h!]

\captionsetup{width=0.96\linewidth}	
	\begin{tikzpicture}
		\node[anchor=south west,inner sep=0] at (0,0) {\includegraphics[width=0.96\linewidth]{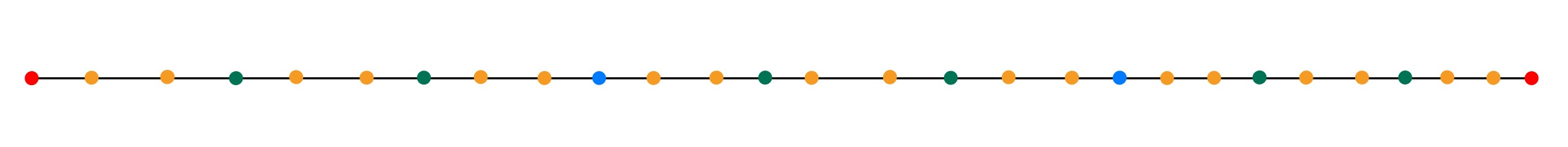}};
		\node at (0.25,0.54) [circle,fill=red,inner sep=1.8pt] {};
		\node at (1.84,0.54) [circle,fill=Green,inner sep=1.6pt] {};
		\node at (3.28,0.54) [circle,fill=Green,inner sep=1.6pt] {};
		\node at (4.66,0.54) [circle,fill=red,inner sep=1.8pt] {};
		\node at (5.94,0.54) [circle,fill=Green,inner sep=1.6pt] {};
		\node at (7.36,0.54) [circle,fill=Green,inner sep=1.6pt] {};
		\node at (8.7,0.54) [circle,fill=red,inner sep=1.8pt] {};
		\node at (9.76,0.54) [circle,fill=Green,inner sep=1.6pt] {};
		\node at (10.88,0.54) [circle,fill=Green,inner sep=1.6pt] {};
		\node at (11.9,0.54) [circle,fill=red,inner sep=1.8pt] {};
		\node at (-0.1,0.5) {$0$};
		\node at (0.7,0.1) {$\frac17$};
		\node at (1.32,0.1) {$\frac16$};
		\node at (1.84,1) {$\frac15$};
		\node at (2.28,0.1) {$\frac29$};
		\node at (2.86,0.1) {$\frac{3}{13}$};
		\node at (3.28,1) {$\frac14$};
		\node at (3.72,0.1) {$\frac{3}{11}$};
		\node at (4.24,0.1) {$\frac27$};
	        \node at (4.64,1) {$\frac13$};
	        \node at (5.06,0.1) {$\frac38$};
	        \node at (5.54,0.1) {$\frac{5}{13}$};
		\node at (5.94,1) {$\frac25$};
		\node at (6.32,0.1) {$\frac{7}{17}$};
		\node at (6.94,0.1) {$\frac{5}{12}$};
		\node at (7.36,1) {$\frac37$};
		\node at (7.8,0.1) {$\frac49$};
		\node at (8.36,0.1) {$\frac{5}{11}$};
                 \node at (8.7,1) {$\frac12$};
                 \node at (9.06,0.1) {$\frac59$};
                 \node at (9.44,0.1) {$\frac47$};
                 \node at (9.76,1) {$\frac35$};
                 \node at (10.1,0.1) {$\frac58$};
                 \node at (10.56,0.1) {$\frac{7}{11}$};
		\node at (10.88,1) {$\frac23$};
		\node at (11.2,0.1) {$\frac57$};
		\node at (11.62,0.1) {$\frac34$};
                 \node at (12.2,0.54) {$1$};
			\end{tikzpicture}
	\caption{The break-points of the piecewise definition of $F$ (which are the ideal vertices of a rank one tile) are marked in red. The new ideal vertices of the rank two, three tiles are displayed in green, orange (respectively). Each of the $27$ complementary components has $\nu'$-mass $\frac{1}{27}$.}
	\label{deep_ranks_tiling_fig}
\end{figure}
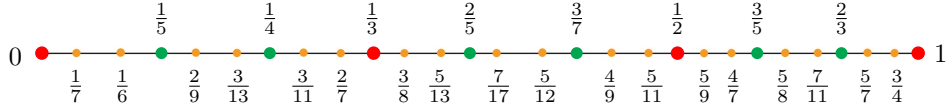

Since $\ln\vert F'\vert$ is increasing on $(0,\frac13)$ and $(\frac13,\frac12)$, and decreasing on $(\frac12,1)$, we have:
\begin{equation}\notag
\begin{split}
& \int_0^1 \ln\vert F'\vert d\nu'\\
&\geq \frac{\ln\vert F'(0)\cdot F'(\frac17)\cdot F'(\frac16)\cdot F'(\frac15)\cdot F'(\frac29)\cdot F'(\frac{3}{13})\cdot F'(\frac14)\cdot F'(\frac{3}{11})\cdot F'(\frac27) \vert}{3^3}\\
& \hspace{1mm} + \frac{\ln\vert F'(\frac13)\cdot F'(\frac38)\cdot F'(\frac{5}{13})\cdot F'(\frac25)\cdot F'(\frac{7}{17})\cdot F'(\frac{5}{12})\cdot F'(\frac37)\cdot F'(\frac49)\cdot F'(\frac{5}{11}) \vert}{3^3}\\
& \hspace{1mm} + \frac{\ln\vert F'(\frac59)\cdot F'(\frac47)\cdot F'(\frac35)\cdot F'(\frac58)\cdot F'(\frac{7}{11})\cdot F'(\frac23)\cdot F'(\frac57)\cdot F'(\frac34)\cdot F'(1) \vert}{3^3}\\
& \approx 1.201 > \ln(3).
\end{split}
\end{equation}
(The number $1.201$ above is obtained by explicit numerical computation using the formula of $\ln\vert F'\vert$ given above, and is correct up to $3$ decimal places.)

Hence, $\mathrm{HD}(\nu') = \frac{\ln(3)}{\int_0^1 \ln\vert F'\vert d\nu'} <1$.
\end{proof}

\subsubsection{The higher Bowen-Series case}\label{hbs_case_subsubsec}

We will now show that the Hausdorff dimension of the MME $\nu$ of the higher Bowen-Series map $A$ of Section~\ref{sec-cfm} is strictly less than $1$. To simplify computations, we will first apply a reduction step that will allow us to work with a degree $-2$ covering of $\mathbb{S}^1$ (note that $A$ is a degree $4$ covering of the circle). This will also relate the Hausdorff dimension of $\nu$ to that of a classically studied measure arising naturally from the Minkowski question-mark function $\ciq$ (see \cite{Denjoy,Salem,Kinney} for details on the question-mark function).

For consistency, we will use the notation employed in Section~\ref{sec-cfm}. Recall that $W$ is a (closed) ideal quadrilateral in $\mathbb{D}$ with ideal vertices at the fourth roots of unity (the quadrilateral $1236$ in Figure~\ref{cfd}). The M{\"o}bius maps $g, h$ pair the sides of this quadrilateral (as shown in Figure~\ref{cfd}), and generate a thrice punctured sphere Fuchsian group $\Gamma_0$. Moreover, $W$ is a (closed) fundamental domain of $\Gamma_0$. The fundamental domain of the higher Bowen-Series map $A$ is given by the ideal hexagon $123567$, while the inner domain of $A$ is the ideal triangle $136$. 

Let us denote reflections in the hyperbolic geodesics $\overline{13}, \overline{36}$, and $\overline{61}$ by $r_1,r_2, r_3$. With this notation, the side-pairing transformations $g$ and $h$ are given by $r_2\circ r_1$ and $r_3\circ r_1$, respectively. It is also readily checked that the map $A:\mathbb{S}^1\to\mathbb{S}^1$ is the second iterate of the piecewise reflection Markov map
$$
\mathfrak{R}:\mathbb{S}^1\to\mathbb{S}^1,\quad
 z \mapsto \left\{\begin{array}{lll}
                    r_1(z) & \mbox{if}\ z\in \arc{123}, \\
                    r_2(z) & \mbox{if}\ z\in \arc{356}, \\
                    r_3(z) & \mbox{if}\ z\in \arc{671}.
                                          \end{array}\right. 
$$
Hence, the circle endomorphisms $A$ and $\mathfrak{R}$ have the same measure of maximal entropy.

We will now relate the map $\mathfrak{R}$ to a well-studied orientation-reversing double covering of $\mathbb{S}^1$.
Note that as any pair of hyperbolic ideal triangles are M{\"o}bius equivalent, the triangle $\Delta 136$ of Figure~\ref{cfd} is M{\"o}bius equivalent to the \emph{regular} ideal polygon $\Pi\subset \mathbb{D}$ with vertices at the third roots of unity. The \emph{Nielsen map} $\pmb{\rho}_2:\mathbb{S}^1\to\mathbb{S}^1$ of the \emph{regular ideal triangle reflection group} is defined as anti-M{\"o}bius reflections in the three sides of $\Pi$ on the three corresponding arcs of $\mathbb{S}^1$ (see Figure~\ref{itg_nielsen_fig} for a pictorial illustration and \cite[\S 2]{LLMM1}, \cite[\S 4.1]{LLMM4} for the precise definition and properties of this map). The M{\"o}bius equivalence of $\Delta 136$ and $\Pi$ implies that the map $\mathfrak{R}$ is M{\"o}bius conjugate to $\pmb{\rho}_2$. Moreover, the fact that M{\"o}bius maps are bi-Lipschitz tells us that the Hausdorff dimension of the MME $\nu$ of the higher Bowen-Series map $A$ is equal to the Hausdorff dimension of the MME of $\pmb{\rho}_2$.
\begin{figure}[h!]

\begin{center}
\includegraphics[width=0.4\linewidth]{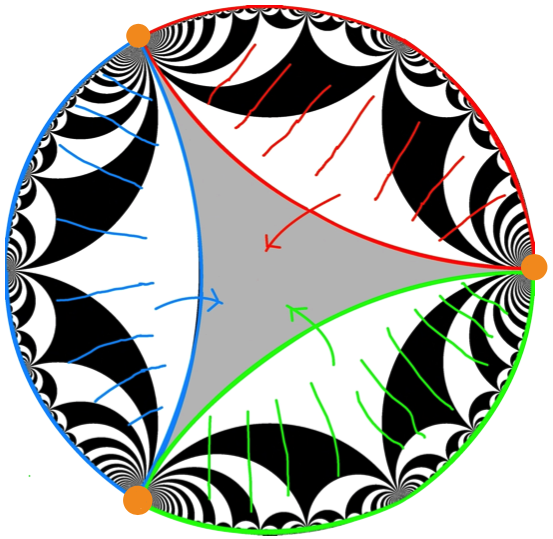}
\caption{The action of the Nielsen map $\pmb{\rho}_2$ of the ideal triangle group is depicted.}
\label{itg_nielsen_fig}
\end{center}
\end{figure}

\noindent \textbf{Figure~\ref{itg_nielsen_fig}:} The Nielsen map $\pmb{\rho}_2$ acts on the arcs $\arc{1 e^{\frac{2\pi i}{3}}}$, $\arc{e^{\frac{2\pi i}{3}} e^{\frac{4\pi i}{3}}}$, and $\arc{e^{\frac{4\pi i}{3}} 1}$  of $\bS^1$ as reflections in the bi-infinite hyperbolic geodesics $\overline{1 e^{\frac{2\pi i}{3}}}$, $\overline{e^{\frac{2\pi i}{3}} e^{\frac{4\pi i}{3}}}$, and $\overline{e^{\frac{4\pi i}{3}} 1}$, respectively. It naturally extends as a piecewise anti-M{\"o}bius map to the complement of the ideal triangle $\Pi$ (in grey) in $\disk$.

Applying the construction of Subsection~\ref{bs_case_subsubsec} to the upper half-plane model of $\pmb{\rho}_2$ (such that the ideal triangle in $\mathbb{D}$ with vertices at the third roots of unity corresponds to the ideal triangle in $\mathbb{H}$ with vertices at $0, 1, \infty$) combined with the arguments of Lemma~\ref{hd_equal_lem}, one can show that the Hausdorff dimension of the MME of $\pmb{\rho}_2$ is equal to the Hausdorff dimension of the MME of the orientation-reversing degree two map
$$
F:[0,1)\to[0,1),\qquad \displaystyle \tau(t)= \left\{\begin{array}{ll}
                     \frac{2t-1}{t-1}\ ({\rm mod}\ 1) \qquad t\in\left[0,\frac12\right),  \\
                     \frac{1-t}{t}\hspace{2.5mm} ({\rm mod}\ 1) \qquad t\in\left[\frac12,1\right).
                                          \end{array}\right. 
$$
We refer the reader to \cite[\S 9]{LLMM4} for the details of this construction. It is also shown there that the map $F$ is topologically conjugate to the orientation-reversing doubling map
$$
\times_{-2}(x)=\left\{\begin{array}{ll}
                     -2x+1\ ({\rm mod}\ 1) \qquad x\in\left[0,\frac12\right),  \\
                     -2x+2\ ({\rm mod}\ 1) \qquad x\in\left[\frac12,1\right),
                                          \end{array}\right.
$$
via the question-mark function. Hence, the MME of $F$ is given by the push-forward of the Lebesgue measure on $[0,1]$ under $\ciq^{-1}$. According to \cite{KS}, the Hausdorff dimension of this measure is strictly less than $1$. In fact, it is shown there that the Hausdorff dimension of the push-forward of the Lebesgue measure on $[0,1]$ under $\ciq^{-1}$ is approximately $0.875$ (see \cite[Figure~2, \S 3]{KS}). We collect the upshot of the above analysis in the following proposition.

\begin{prop}\label{hd_less_than_one_prop_1}
Let $\nu$ be the measure of maximal entropy of the higher Bowen-Series map $A$ of a Fuchsian thrice punctured sphere group. Then, 
$$
\mathrm{HD}(\nu)\approx 0.875 < 1.
$$
\end{prop}

\subsection{Some open questions}\label{qns_subsec}

We conjecture that the results of Section~\ref{hd_computations_subsec} hold in greater generality.

\begin{qn}\label{question_1}
	Let $\nu$ be the measure of maximal entropy of a (higher) Bowen-Series map of a Fuchsian punctured sphere group.
	Is 
	$$
	\mathrm{HD}(\nu):= \mathrm{inf} \{\mathrm{HD}(Y): Y\subset \bS^1,\ \nu(Y)=1\}
	$$
	less that $1$?
\end{qn}

Henceforth we will assume that $k>3$, so that the surface $S_{0,k}$ has a non-trivial Teichm{\"u}ller space. 

We believe that the Hausdorff dimension of the limit set of the conformal mating of $\widehat{A}_{\Gamma,\mathrm{BS}}$ (respectively, $\widehat{A}_{\Gamma,\mathrm{hBS}}$) and $P$, where $\Gamma\in\mathrm{Teich}(S_{0,k})$ and $P\in\mathcal{H}_{2k-3}$ (respectively, $P\in\mathcal{H}_{(k-1)^2}$), is strictly greater than $1$. The next question is motivated by Bowen's theorem on Hausdorff dimension of quasi-Fuchsian limit sets (cf. \cite{bowen}). 

\begin{qn}\label{question_2}
	Do the Hausdorff dimensions of limit sets of the above class of conformal matings attain its global minimum at a unique point?
\end{qn}

The following questions are motivated by results of McMullen on variation of Hausdorff dimensions of limit sets and naturally associated measures living on them (cf. \cite{ctm_td_wp}).

\begin{qn}\label{question_3}
	How does the Hausdorff dimension of the measure of maximal entropy of the (higher) Bowen-Series map vary as the marked group runs over $\mathrm{Teich}(S_{0,k})$?
\end{qn}

\begin{qn}\label{question_4}
	How does the Hausdorff dimension of the limit set of the conformal mating of $\widehat{A}_{\Gamma,\mathrm{BS}}$ (respectively, $\widehat{A}_{\Gamma,\mathrm{hBS}}$) and $P$ vary as $\Gamma$ runs over $\mathrm{Teich}(S_{0,k})$ and $P$ runs over $\mathcal{H}_{2k-1}$ (respectively, over $\mathcal{H}_{(k-1)^2}$)?
\end{qn}

\end{document}